\documentclass[reqno, 11pt, a4paper]{amsart}
\usepackage{latexsym, ifthen, xspace}
\usepackage{amsmath, amssymb, amsthm, mathabx}
\usepackage{bbm}
\usepackage{subfigure}
\usepackage{enumerate, calc}
\usepackage[colorlinks=true, pdfstartview=FitV, linkcolor=blue,
citecolor=blue, urlcolor=blue, pagebackref=false]{hyperref}
\usepackage[text={33pc,605pt},centering]{geometry}               % 11pt
\usepackage{orcidlink}
\usepackage{graphicx}                                            % graphics

\newtheorem{theorem}{Theorem}[section]
\newtheorem{lemma}[theorem]{Lemma}
\newtheorem{prop}[theorem]{Proposition}
\newtheorem{assumption}[theorem]{Assumption}
\newtheorem{coro}[theorem]{Corollary}

\newtheorem{conjecture}[theorem]{Conjecture}

\theoremstyle{definition}
\newtheorem{definition}[theorem]{Definition}

\theoremstyle{remark}
\newtheorem{remark}[theorem]{Remark}

\numberwithin{equation}{section}

\DeclareMathAlphabet{\mathsl}{OT1}{cmss}{m}{sl}
\SetMathAlphabet{\mathsl}{bold}{OT1}{cmss}{bx}{sl}

\newcommand{\me}{\ensuremath{\mathrm{e}}}

\newcommand{\md}{\ensuremath{\mathrm{d}}}

\newcommand{\scpr}[3]{%
  \ensuremath{%
    \left\langle
      #1, #2
    \right\rangle_{\raisebox{-0ex}{$\scriptstyle \ell^{\raisebox{.1ex}{$\scriptscriptstyle 2$}} (#3)$}}
  }
}

\newcommand{\Norm}[2]{%
  \ensuremath{%
    \big\lVert
      #1
    \big\rVert_{\raisebox{-.0ex}{$\scriptstyle #2$}}
  }
}

\DeclareMathOperator{\mean}{\mathbb{E}}
\DeclareMathOperator{\Mean}{\mathrm{E}}
\DeclareMathOperator{\prob}{\mathbb{P}}
\DeclareMathOperator{\Prob}{\mathrm{P}}

\DeclareMathOperator{\supp}{\mathrm{supp}}

\DeclareMathOperator{\osc}{\mathrm{osc}}

\DeclareMathOperator{\meanPhi}{\mathbf{E}}
\DeclareMathOperator{\probPhi}{\mathbf{P}}
\DeclareMathOperator{\covPhi}{\mathbf{Cov}}
\DeclareMathOperator{\varGFF}{\mathbf{Var}}

\newcommand{\ldef}{\ensuremath{\mathrel{\mathop:}=}}
\newcommand{\rdef}{\ensuremath{=\mathrel{\mathop:}}}

\newcommand{\indicator}{\ensuremath{\mathbbm{1}}}

\begin{document}

\title[Scaling limit of the inhomogeneous DGFF]{Scaling limit of the discrete Gaussian free field with degenerate random conductances}

%    Remove any unused author tags.

%    author one information
\author[S.\ Andres]{Sebastian Andres \orcidlink{0000-0002-5116-7789}}
\address{Technische Universit\"at Braunschweig}
\curraddr{%
  Institut f\"ur Mathematische Stochastik,
  Universit\"atsplatz 2,
  38106 Braunschweig
}
\email{sebastian.andres@tu-braunschweig.de}
\thanks{}

%    author two information
\author[M.\ Slowik]{Martin Slowik \orcidlink{0000-0001-5373-5754}}
\address{University of Mannheim}
\curraddr{Mathematical Institute, B6, 26, 68159 Mannheim}
\email{slowik@math.uni-mannheim.de}
\thanks{}

% author three information
\author[A.-L.\ Sokol]{Anna-Lisa Sokol}
\address{Technische Universit\"at Berlin (Alumna)}
\curraddr{Strasse des 17. Juni 136, 10623 Berlin}
\email{anna-lisa.sokol@alumni.tu-berlin.de}
\thanks{}

\subjclass[2000]{39A12; 60J35; 60J45; 60K37; 82C41}

\keywords{Gaussian free field, random walk; Green kernel; random conductance model}

\date{\today}

\dedicatory{}

\begin{abstract}
  We consider discrete Gaussian free fields with ergodic random conductances on a class of random subgraphs of $\mathbb{Z}^{d}$, $d \geq 2$, including i.i.d.\ supercritical percolation clusters, where the conductances are possibly unbounded but satisfy a moment condition. As our main result, we show that, for almost every realization of the environment, the rescaled field converges in law towards a continuum Gaussian free field.  We also present a scaling limit for the covariances of the field. To obtain the latter, we establish a quenched local limit theorem for the Green's function of the associated random walk among random conductances with Dirichlet boundary conditions.
\end{abstract}

\maketitle

\tableofcontents

\section{Introduction}
The discrete Gaussian free field (DGFF), also known as the harmonic crystal, serves as a prime example within a class of models in modern statistical mechanics that describe, under a particular idealisation, an interface between two pure thermodynamic phases, see, for instance, \cite{Bi20} for a survey. It can be defined as the Gaussian random field whose covariances are given by the Green's function of the discrete Laplace operator. In the present article we investigate the scaling behaviour of the DGFF with \emph{impurities} modelled by random conductances. In a physical context, the DGFF with constant conductances may be interpreted as a deformation field representing the microscopic displacements of atoms in a crystal from their ideal position in a \emph{homogeneous} crystal at non-zero temperature. Similarly, the microscopic structure of fluctuations in inhomogeneous crystals with impurities motivates to consider a version of the DGFF with random conductances, namely a centred Gaussian random field with covariances given by the Green's function of the discrete Laplacian with random weights. To our knowledge this model first appeared in~\cite{CI03}, and was studied in~\cite{BS11, CK12, CK15} and more recently in~\cite{CN21, CR24}, see also~\cite[Section~6]{Bi11}. Here, we prove the convergence of the suitably rescaled field towards a continuum Gaussian free field and provide a scaling limit for the covariances of the field. We allow random environments in the form of ergodic and unbounded random conductances, only requiring a moment condition, on a class of random subgraphs of $\mathbb{Z}^{d}$, including both i.i.d.\ and certain correlated supercritical percolation clusters. The main technical input will come from some homogenization results of independent interest for the associated random walk in random environment, known in the literature as random conductance model (RCM), and its Green's function, namely a certain extension of the quenched invariance principle for the random walk and a quenched local limit theorem for the Green's function with Dirichlet boundary condition on a class of bounded domains.

\subsection{The model}
We will study random Gaussian fields on $\mathbb{Z}^{d}$ (or on more general random subgraphs of $\mathbb{Z}^{d}$) with covariances given by the Green's function of a random walk among random conductances on such graphs, which we need to introduce first. For $d \geq 2$ let $(\mathbb{Z}^{d}, E_{d})$ be the $d$-dimensional Euclidean lattice with edge set, $E_{d}$, given by the set of all non-oriented nearest neighbour bonds. We call two vertices $x, y \in \mathbb{Z}^{d}$ adjacent if $\{x, y\} \in E_{d}$, and we then write $x \sim y$. Let $(\Omega, \mathcal{F}) \ldef ([0, \infty)^{E_{d}}, \mathcal{B}([0, \infty))^{\otimes E_{d}})$ be a measurable space equipped with the Borel-$\sigma$-algebra. Furthermore, for any $\omega = \{\omega(e), e \in E_{d}\} \in \Omega$, we refer to $\omega(e)$ as the conductance of the edge $e$. We call an edge $e \in E_{d}$ \emph{open} if $\omega(e) > 0$ and denote by $\mathcal{O}(\omega)$ the set of open edges.  Moreover, let $\mathcal{C}_{\infty}(\omega)$ be the subset of vertices of $\mathbb{Z}^{d}$ that belongs to infinite components, each of which is connected with respect to path along open edges.

We endow the measurable space $(\Omega, \mathcal{F})$ with the $d$-parameter group of space shifts, $(\tau_{x})_{x \in \mathbb{Z}^{d}}$, that acts on $\Omega$ as
\begin{align}\label{eq:def:space_shift}
  (\tau_{z}\, \omega)(\{x, y\})
  \;\ldef\;
  \omega(\{x+z, y+z\}),
  \qquad \forall\, \{x, y\} \in E_{d}.
\end{align}
Further, we will denote by $\prob$ a probability measure on $(\Omega, \mathcal{F})$, and we write $\mean$ for the expectation with respect to $\prob$. Throughout the paper, we will impose assumptions both on the law $\prob$ and on geometric properties of the infinite cluster.
\begin{assumption} \label{ass:law}
  Assume that $\prob$ satisfies the following conditions.
  \begin{enumerate}[(i)]
  \item
    The law $\prob$ is stationary and ergodic with respect to space shifts of $\mathbb{Z}^{d}$, that is, $\prob \circ \tau_{x}^{-1} = \prob$ for all $x \in \mathbb{Z}^{d}$ and $\prob[A] \in \{0, 1\}$ for any $A \in \mathcal{F}$ such that $\tau_{x}^{-1}(A) = A$ for all $x \in \mathbb{Z}^{d}$.     
 
  \item
    For $\prob$-a.e.\ $\omega$, the set $\mathcal{C}_{\infty}(\omega)$ is non-empty and connected, that is, there exists a unique infinite connected component -- also called infinite open cluster -- and $\prob\bigl[0 \in \mathcal{C}_{\infty}\bigr] > 0$.
  \end{enumerate}
\end{assumption}
Let $\Omega^{*} \ldef \{\omega \in \Omega : \mathcal{C}_{\infty}(\omega) \text{ is non-empty and connected} \}$ and $\Omega^{*}_{0} \ldef \{\omega \in \Omega^{*} : 0 \in \mathcal{C}_{\infty}(\omega)\}$. We write $\prob_{0}[ \,\cdot\, ] \ldef \prob\bigl[ \,\cdot \,|\, 0 \in \mathcal{C}_{\infty}\bigr]$ to denote the conditional distribution given the event $\{0 \in \mathcal{C}_{\infty}\}$, and $\mean_{0}$ for the expectation with respect to $\prob_{0}$. In the case of i.i.d.\ conductances Assumption~\ref{ass:law} is fulfilled if $\prob\bigl[\omega(e) > 0\bigr] > p_{c}$, where $p_{c} \equiv p_{c}(d)$ denotes the critical probability for bond percolation on $\mathbb{Z}^{d}$. We now introduce the inhomogeneous harmonic crystal.
\begin{definition}[Inhomogeneous DGFF]\label{def:inhomogeneous:DGFF}
  Let $\Lambda \subset \mathbb{R}^{d}$ be bounded. For any $\omega \in \Omega^{*}$, the \emph{inhomogeneous discrete Gaussian free field} on $\Lambda \cap \mathcal{C}_{\infty}(\omega)$ with zero boundary conditions is the Gaussian process $\varphi^{\Lambda} = \bigl(\varphi_{x}^{\Lambda} : x \in \mathbb{Z}^{d} \bigr)$ with law $\probPhi^{\omega}$ given by
  \begin{align}
    \probPhi^{\omega}\bigl[\varphi^{\Lambda} \in A\bigr]
    \;=\;
    \frac{1}{Z_{\Lambda}^{\omega}}
    \int_{A}
      \me^{-\tfrac{1}{2}\sum_{\{x, y\} \in E_{d}} \omega(\{x, y\}) (\varphi_{x} - \varphi_{y})^2}
      \mspace{-6mu}
    \prod_{x \in \Lambda \cap \mathcal{C}_{\infty}(\omega)}\mspace{-12mu} \mathrm{d}\varphi_{x}
    \prod_{x \in (\Lambda \cap \mathcal{C}_{\infty}(\omega))^{\mathrm{c}}}\mspace{-12mu} \delta_{0}[\mathrm{d} \varphi_{x}],
  \end{align}
  for any measurable $A \subset \mathbb{R}^{\mathbb{Z}^{d}}$, where $\delta_{0}$ denotes the Dirac measure at $0$ and $Z_{\Lambda}^{\omega}$ is a normalization constant.
\end{definition} 
Since $\Lambda$ is finite, the normalization constant $Z_{\Lambda}^{\omega}$ given by 
\begin{align*}
  Z_{\Lambda}^{\omega}
  \;\ldef\;
  \int_{\mathbb{R}^{\mathbb{Z}^{d}}}
    \me^{-\tfrac{1}{2}\sum_{\{x, y\} \in E_{d}} \omega(\{x, y\}) (\varphi_{x} -\varphi_{y})^2} \mspace{-6mu}
  \prod_{x \in \Lambda \cap \mathcal{C}_{\infty}(\omega)}\mspace{-12mu} \mathrm{d}\varphi_{x}
  \prod_{x \in (\Lambda \cap \mathcal{C}_{\infty}(\omega))^{\mathrm{c}}}\mspace{-12mu} \delta_{0}[\mathrm{d}\varphi_{x}],
\end{align*}
is finite.  Further, the field $\varphi^{\Lambda}$ as in Definition~\ref{def:inhomogeneous:DGFF} is a multivariate Gaussian process with mean and covariance given by
\begin{align} \label{eq:covField}
  \meanPhi^{\omega}\bigl[\varphi_{x}^{\Lambda}\bigr] \;=\; 0
  \quad \text{and} \quad
  \meanPhi^{\omega}\bigl[\varphi_{x}^{\Lambda}\, \varphi_{y}^{\Lambda}\bigr]
  \;=\;
  g_{\Lambda}^{\omega}(x, y),
  \qquad x, y \in \Lambda \cap \mathcal{C}_{\infty}(\omega),
\end{align}
where $g^{\omega}_{\Lambda}$ denotes the Green's function on $\Lambda \cap \mathcal{C}_{\infty}(\omega)$ associated with the operator $\mathcal{L}^{\omega}$ acting on bounded functions $f\colon \mathbb{Z}^{d} \to \mathbb{R}$ as
\begin{align} \label{eq:defL}
  \big(\mathcal{L}^{\omega} f)(x)
  \;=\; 
  \sum_{y \sim x} \omega(\{x, y\})\, \bigl(f(y) - f(x)\bigr).
\end{align}
That is, $g^{\omega}_{\Lambda}$ is the solution of the Poisson equation 
\begin{align}\label{eq:PDE_green_killed}
  \left\{
    \mspace{-6mu}
    \begin{array}{rcll}
      \big(\mathcal{L}^{\omega} g_{\Lambda}^{\omega}(\cdot, y)\big)(x)
      &\mspace{-5mu}=\mspace{-5mu}& -\delta_{y}(x),
      &x \in \Lambda \cap \mathcal{C}_{\infty}(\omega),
      \\[.5ex] 
      g_{\Lambda}^{\omega}(x, y)
      &\mspace{-5mu}=\mspace{-5mu}& 0, 
      &x \in (\Lambda \cap \mathcal{C}_{\infty}(\omega))^{\mathrm{c}}.
    \end{array}
  \right.
\end{align}
\renewcommand{\thesubfigure}{}
\begin{figure} 
  \begin{center}
    \subfigure[(a)]{\scalebox{0.578}{\includegraphics{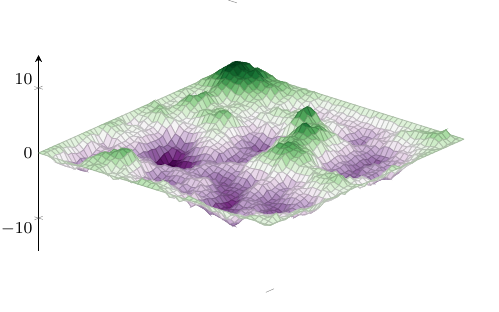}}}
    \hspace{-3ex}
    \subfigure[(b)]{\scalebox{0.578}{\includegraphics{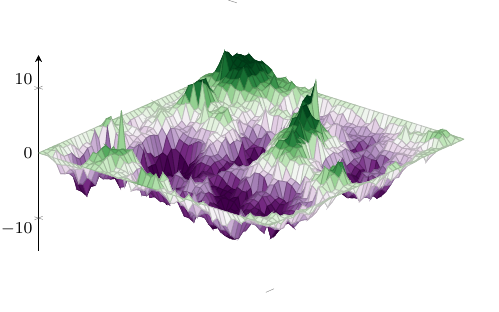}}}
    \hspace{-3ex}
    \subfigure[(c)]{\scalebox{0.578}{\includegraphics{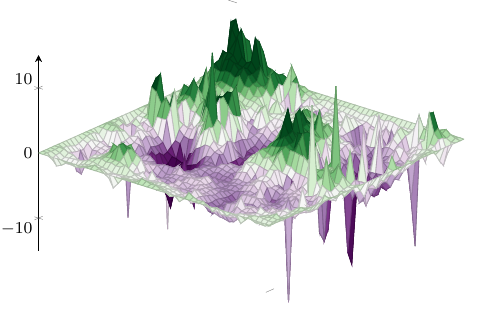}}}	
    \caption{%
      A sample of the inhomogeneous DGFF on a $50 \times 50$ square of $\mathbb{Z}^{2}$ with 
      (a) $\omega(e) = 1$;  (b) $\omega(e) \sim \mathrm{Exp}(1)$ i.i.d.;
      (c) $\omega(e) \sim \mathrm{Exp}(1)$ constant along lines, independent between different lines. In all three figures, for generating the samples the same seed value for the random number generator is used.}
  \end{center}
  \label{fig:numsim}
\end{figure}
Next, we introduce the associated random walk in random environment, which is well-known under the label random conductance model (RCM). For any realization $\omega \in \Omega^{*}$ consider the continuous-time Markov chain $X \equiv (X_{t})_{t \geq 0}$ on $\mathcal{C}_{\infty}(\omega)$ with generator $\mathcal{L}^{\omega}$ defined in \eqref{eq:defL}. Notice that $X$ is reversible with respect to the counting measure. When visiting a vertex $x \in \mathcal{C}_{\infty}(\omega)$, the random walk $X$ waits at $x$ an exponential time with mean $1/\mu^{\omega}(x)$, where we define $\mu^{\omega}(x) \ldef \sum_{y \sim x} \omega(\{x, y\})$ for any $x \in \mathbb{Z}^{d}$, and then it jumps to a vertex $y \sim x$ with probability $\omega(\{x, y\})/\mu^{\omega}(x)$.  Since the law of the waiting time depends on the location of the walk, this walk is often called the variable speed walk (VSRW) in the literature.  We denote by $\Prob^{\omega}_{x}$ the quenched law of the process $X$ starting at $x \in \mathcal{C}_{\infty}(\omega)$, and the corresponding expectation by $\Mean_{x}^{\omega}$. Furthermore, we define the \emph{heat kernel} as $p_{t}^{\omega}(x, y) \ldef \Prob_{x}^{\omega}[X_{t} = y]$ for $x, y \in \mathbb{Z}^{d}$ and $t \geq 0$. Note that the heat kernel is invariant under space shifts, that is,
\begin{align} \label{eq:HK_shift}
  p_{t}^{\tau_z \omega}(x, y) \;=\; p_{t}^{\omega}(x+z, y+z),
  \qquad z \in \mathbb{Z}^{d}.
\end{align}
Throughout the paper, for any path $w\colon [0, \infty) \rightarrow \mathbb{R}^{d}$, and any subset $\Lambda \subset \mathbb{R}^{d}$ we define its first exit time from $\Lambda$ by
\begin{align*}
  \tau_{\Lambda}(w) \ldef \inf \bigl\{t \geq 0 : w_{t} \notin \Lambda \bigr\}, 
\end{align*}
and we will suppress the dependence of $w$ in the notation when it is clear from the context. Then, for any bounded $\Lambda \subset \mathbb{R}^{d}$, the Green's function , $g_{\Lambda}^{\omega}(x, y)$, describes the expected amount of time that the VSRW $X$ spends in $y$ when starting in $x$ before exiting $\Lambda$, that is, 
\begin{align*}
  g_{\Lambda}^{\omega}(x, y)
  \;=\;
  \Mean_{x}^{\omega}\biggl[
    \int_{0}^{\tau_{\Lambda}(X)} \indicator_{\{ X_{t} = y\}}\, \mathrm{d}t
  \biggr]
  \;=\;
  \int_{0}^{\infty} \Prob_{x}^{\omega}\bigl[X_t = y,\, t < \tau_{\Lambda}(X)\bigr]\, \mathrm{d}t
\end{align*}
for any $x, y \in \Lambda \cap \mathcal{C}_{\infty}(\omega)$.

\subsection{Main results}
First, we state the assumptions on the underlying graph under which our first main result will be proved. Given $\omega \in \Omega^{*}$, let $E(\mathcal{C}_{\infty}(\omega)) \ldef \{\{x, y\} \in \mathcal{O}(\omega) : x, y \in \mathcal{C}_{\infty}(\omega)\}$ be the edge set of the infinite cluster. We denote by $d^{\omega}$ the graph‐distance on the pair 
$\bigl(\mathcal{C}_{\infty}(\omega),\,E(\mathcal{C}_{\infty}(\omega))\bigr)$, namely, for any $x, y \in \mathcal{C}_{\infty}(\omega)$, $d^\omega(x, y)$ is the minimal length of a path joining $x$ and $y$ that consists only of open edges. For $x \in \mathcal{C}_{\infty}(\omega)$ and $r \geq 0$, let $B^{\omega}(x, r) \ldef {\{ y \in \mathcal{C}_{\infty}(\omega) : d^{\omega}(x, y) \leq \lfloor r \rfloor\}}$ be the closed ball with centre $x$ and radius $r$ with respect to $d^{\omega}$. Likewise, for $x \in \mathbb{Z}^d$ we write $B(x, r) \ldef \{y \in \mathbb{Z}^d : |x-y|_{1}\leq \lfloor r \rfloor\}$
for balls in $\mathbb{Z}^d$ centred at $x$, where $|\cdot|_{1}$ is the usual graph‐distance on $\mathbb{Z}^d$. More generally, $|\cdot|_{p}$ with $p \in [1, \infty]$ denotes the standard $p$‐norm on both $\mathbb{R}^d$ and $\mathbb{Z}^d$. Finally, for any $A \subset B \subset \mathbb{Z}^{d}$ we define the \emph{relative} boundary of $A$ with respect to $B$ by
\begin{align*}
  \partial_{\!B}^{\omega} A
  \;\ldef\;
  \big\{%
    \{x,y\} \in \mathcal{O}(\omega) :
    x \in A \,\text{ and }\, y \in B \setminus A
  \big\},
\end{align*}
and we simply write $\partial^{\omega} A$ if $B \equiv \mathcal{C}_{\infty}(\omega)$.
\begin{definition}[Regular balls]\label{def:regBall}
  Let $C_{\mathrm{V}} \in (0, 1]$, $C_{\mathrm{riso}} \in (0, \infty)$ and $C_{\mathrm{W}} \in [1, \infty)$ be fixed constants.  For $x \in \mathcal{C}_{\infty}(\omega)$ and $n \geq 1$, we say a ball $B^{\omega}(x, n)$ is \emph{regular} if it satisfies the following conditions.
  \begin{enumerate}[(i)]
  \item
    Volume regularity of order $d$, that is,
    \begin{align*}
      C_{\mathrm{V}}\, n^{d} \;\leq\; |B^{\omega}(x, n)|.
    \end{align*}
     
  \item
    (Weak) relative isoperimetric inequality. There exists $\mathcal{S}^{\omega}(x, n) \subset \mathcal{C}_{\infty}(\omega)$ connected such that $B^{\omega}(x, n) \subset \mathcal{S}^{\omega}(x, n) \subset B^{\omega}(x, C_{\mathrm{W}} n)$ and
    \begin{align*}
      |\partial_{\mathcal{S}^{\omega}(x, n)}^{\omega} A|
      \;\geq\;
      C_{\mathrm{riso}}\, n^{-1}\, |A|
    \end{align*}
    for every $A \subset \mathcal{S}^{\omega}(x, n)$ with $|A| \leq \tfrac{1}{2}\, |\mathcal{S}^{\omega}(x, n)|$.
  \end{enumerate}
\end{definition}
\begin{assumption}\label{ass:cluster} 
  For some $\theta, \delta \in (0,1)$, $C_{\mathrm{V}} \in (0, 1]$, $C_{\mathrm{riso}} \in (0, \infty)$ and $C_{\mathrm{d}}, C_{\mathrm{W}} \in [1, \infty)$, assume that there exist $\Omega_{\mathrm{reg}} \in \mathcal{F}$ with $\prob[\Omega_{\mathrm{reg}}] = 1$ and a non-negative random variable $N_{0}$ such that $N_{0}(\omega) < \infty$ for any $\omega \in \Omega_{\mathrm{reg}} \cap \Omega_{0}^{*}$, and for all $n \geq N_0(\omega)$ the following hold.
  \begin{enumerate}[(i)]
  \item
    The ball $B^{\omega}(0, n)$ is $\theta$-\emph{very regular}, that is, for every $x \in B^{\omega}(0, n)$ and $r \geq n^{\theta/d}$ the ball $B^{\omega}(x, r)$ is regular with constants $C_{\mathrm{V}}$, $C_{\mathrm{riso}}$ and $C_{\mathrm{W}}$.  

  \item[(ii)]
    For any $x, y \in [-n, n]^{d} \cap \mathcal{C}_{\infty}(\omega)$,
    \begin{align*}
      d^{\omega}(x, y)
      \;\leq\;
      \bigl( C_{\mathrm{d}}\, |x - y|_{\infty} \bigr) \vee n^{1-\delta}.
    \end{align*}
  \end{enumerate}
\end{assumption}
\begin{remark} \label{rem:ass_cluster} 
  (i) Assumption~\ref{ass:cluster} holds, for instance, on supercritical i.i.d.\ Bernoulli percolation clusters for all $\theta \in (0,1)$, see~\cite{Ba04}. Moreover, Assumption~\ref{ass:cluster} is even satisfied for a class of percolation models with long range correlations, see \cite[Proposition~4.3]{Sa17} and \cite[Theorems~2.3]{DRS14}. For more details and examples, including level sets of DGFFs, we refer to \cite[Examples~1.11--1.13]{DNS18} and references therein.

  (ii) For any bounded and simply connected $D \subset \mathbb{R}^{d}$, setting $n D \ldef \{z \in \mathbb{R}^{d} : z/n \in D\}$,  there exists $r_{D} \in (0, \infty)$ such that $n D \subset [-r_{D} n, r_{D} n]^{d}$. Then, Assumption~\ref{ass:cluster}-(ii) immediately implies that, for $\prob_{0}$-a.e.\ $\omega$ and all $n \geq \max\{ N_{0}(\omega)/r_{D}, 1 \}$, we have $n D \cap \mathcal{C}_{\infty}(\omega) \subset B^{\omega}(0, C_{\mathrm{d}} r_{D} n)$.  
\end{remark}
\begin{assumption}\label{ass:pq}
  There exist $p, q \in [1, \infty]$ and $\theta \in (0, 1)$ satisfying
  \begin{align} \label{eq:cond:pq_rg}
    \frac{1}{p} \,+\, \frac{1}{q}
    \;<\;
    \frac{2(1-\theta)}{d-\theta},
  \end{align}
  such that for any $e \in E_{d}$,
  \begin{align*}
    \mean\bigl[ \omega(e)^{p} \bigr] \;<\; \infty
    \qquad \text{and} \qquad
    \mean\bigl[ \omega(e)^{-q} \indicator_{\{e \in \mathcal{O}\}} \bigr] \;<\; \infty,
  \end{align*}
  where we used the convention that $0/0 = 0$.
\end{assumption}
One of the most important results in the theory of random walks among random conductances is the \emph{quenched invariance principle} or \emph{quenched functional central limit theorem} (QFCLT), which describes a weak convergence of the rescaled random walk towards Brownian motion. Throughout the paper, for any $T>0$, we will denote by $(\Prob_{x}^{\Sigma})_{x \in \mathbb{R}^{d}}$ the law of a Brownian motion with a non-degenerate covariance matrix $\Sigma^2 = \Sigma \cdot \Sigma^{T}$, defined on the path space $C([0,T], \mathbb{R}^{d})$ (or $D([0,T],\mathbb{R}^{d})$ if convenient), and we will then write $(W_{t})_{0 \leq t \leq T}$ for the coordinate process. In particular, $\Prob_{x}^{\Sigma}\bigl[ W_{t} \in \mathrm{d}y \bigr] = k_{t}^{\Sigma}(x,y)\, \mathrm{d}y$
with
\begin{align*}
  k_{t}^{\Sigma}(x, y)
  \;=\;
  \frac{1}{\sqrt{(2\pi t)^{d} \det \Sigma^2}}
  \exp\biggl(-\frac{(x-y) (\Sigma^{2})^{-1} (x-y)}{2t}\biggr).
\end{align*}
\begin{theorem} [QFCLT \cite{DNS18}]\label{thm:QFCLT:RG}
  Suppose there exist $\theta \in (0, 1)$ and $p, q \in [1, \infty]$ such that Assumptions~\ref{ass:law}, \ref{ass:cluster}-(i) and \ref{ass:pq} hold.  Set $X_{t}^{n} \ldef n^{-1} X_{t n^{2}}$ for any $n \in \mathbb{N}$ and $t \geq 0$. Then, for $\prob_{0}$-a.e.\ $\omega$, the process $X^{n} \equiv \bigl(X_t^{n}\bigr)_{t \geq 0}$, converges (under $\Prob_{0}^\omega$) in law towards a Brownian motion on $\mathbb{R}^{d}$ with a deterministic non-degenerate covariance matrix $\Sigma^2$. That is, $\prob_0$-a.s., for every $T > 0$ and every bounded continuous function $F$ on the Skorohod space $D([0,T], \mathbb{R}^{d})$,
  \begin{align} \label{eq:convFCLT}
    \Mean_{0}^{\omega}\bigl[ F(X^{n}) \bigr]
    \overset{n\to \infty} \longrightarrow
    \Mean_{0}^{\Sigma}\bigl[ F(W) \bigr].
  \end{align}
\end{theorem}
From now on we will fix a bounded domain $D \subset \mathbb{R}^{d}$, that is, an open, bounded and simply connected subset of $\mathbb{R}^{d}$, and throughout the paper we will assume that its boundary points are regular in the following sense.
\begin{definition} \label{def:regular}
  We call a point $z \in \partial D$ \emph{strongly regular} if $\Prob_{z}^{\Sigma}[ \tau_{\bar{D}}(W) = 0] = 1$. We say that $D$ is strongly regular if every point $z \in \partial D$ is strongly regular.
\end{definition}
Note that a regular point has the property that the Brownian motion starting at it will essentially immediately hit $\bar{D}^{\mathrm{c}}$.
\begin{remark}
  This definition of strong regularity is directly adapted from the one given in \cite[Exercise~10.2.20]{St11}. Note that it differs from the common definition of a regular point in the literature, see e.g.\ \cite{KS91}, which only requires the first hitting time of $D^{\mathrm{c}}$  to be zero almost surely. A well-known criterion for a point to be regular in this sense is the cone condition. More precisely, $x \in \partial D$ is regular if there exists a cone $V$ with apex $x$ and opening angle $\alpha > 0$ such that, for some $r > 0$, $V \cap B_{r}(x) \subset D^{\mathrm{c}}$, where $B_{r}(x)$ denotes the open ball in $\mathbb{R}^{d}$, equipped with the usual Euclidean norm, centred at $x$ with radius $r$, see e.g.~\cite[Chapter~4, Theorem~2.19]{KS91}. The same proof also shows that the cone condition is sufficient for a point to be strongly regular in the sense of Definition~\ref{def:regular} for Brownian motions with a non-degenerate covariance matrix. The main ingredients are Blumenthal's zero-one law and the invariance of the cone with respect to scaling.
\end{remark}
Further, let $k^{\Sigma, D}\colon (0, \infty) \times \mathbb{R}^{d} \times \mathbb{R}^{d} \to [0, \infty)$ denote the heat kernel of the Brownian motion, $W$, with covariance matrix $\Sigma^{2}$ killed upon exiting the domain $D$. More precisely, $k^{\Sigma, D}$ is the jointly continuous function such that $\Prob_{x}^{\Sigma}\bigl[ W_{t}^{D} \in \mathrm{d}y \bigr] = k_{t}^{\Sigma, D}(x,y)\, \mathrm{d}y$ for $t > 0$ and $x \in D$, where $W^{D}$ denotes the coordinate process killed upon exiting of $D$. The associated Green's function is given by
\begin{align}\label{eq:Green_killed_cont}
  g^{\Sigma}_{D}(x,y)
  \;\ldef\;
  \int_{0}^{\infty} k_{t}^{\Sigma, D}(x,y)\, \mathrm{d}t
  \;\in\;
  [0, \infty].
\end{align}

Just as the Brownian motion can be realized as the scaling limit of many random curve ensembles, including simple random walks, the RCM (cf.\ Theorem~\ref{thm:QFCLT:RG}) and many others, as our first main result we will show that continuum Gaussian free fields emerge as the scaling limit of inhomogeneous discrete Gaussian free fields. However, the continuum Gaussian free fields cannot be defined pointwise, as the variance at any point would have to be infinite. Instead, they are defined as a restricted family of random distributions or generalised functions in the sense of Schwartz, taking values the space $H^{-s}(D)$ for some $s > 0$ the definition of which we now recall. For any $s \geq 0$, let $H_{0}^{s}(D)$ be the Sobolev space defined as the Banach space closure of $C_{0}^{\infty}(D)$ under the norm $\| \cdot \|_{H_{0}^{s}(D)}$ induced by the scalar product
\begin{align}
  \label{eq:scalar_product:Sobolev_space}
  (f, g)_{H_{0}^{s}(D)}
  \;\ldef\;
  \int_{D}
    \bigl((1 - \Delta)^{s/2}f\bigr)(x)\, \bigl((1 - \Delta)^{s/2} g\bigr)(x)\,
  \mathrm{d}x, 
\end{align}
where $(1-\Delta)^{s/2} f \ldef \mathcal{F}^{-1}\bigl((1+|\cdot|_{2}^{2})^{s/2}(\mathcal{F} f)\bigr)$ and $\mathcal{F}f$ denotes Fourier transform of $f$. Let $H^{-s}(D)$ be its dual space with respect to the $L^{2}(D)$ scalar product. For further details see, for instance, \cite[Chapter~5]{Ev10} or \cite[Chapter~4]{Ta23}.
\begin{definition}[Continuum Gaussian free field (CGFF)]
  Let $d \geq 2$,  $D \subset \mathbb{R}^{d}$ be a bounded domain, $c_{\sigma} > 0$ and $\Sigma \in \mathbb{R}^{d \times d}$ be a positive definite matrix. The \emph{continuum Gaussian free field on $D$ with parameters $c_{\sigma}$ and $\Sigma$} is an assignment $f \mapsto \Phi^{\Sigma, D}(f)$ of a random variable to each bounded measurable $f\colon D \to \mathbb{R}$ such that
  \begin{enumerate}[(i)]
  \item 
    $\Phi^{\Sigma, D}$ is a.s.\ linear, i.e.\ for all bounded measurable $f, g\colon D \to \mathbb{R}$ and $a, b \in \mathbb{R}$,
    \begin{align*}
      \Phi^{\Sigma, D}(a f + b g)
      \;=\;
      a\, \Phi^{\Sigma, D}(f) + b\, \Phi^{\Sigma, D}(g)
      \qquad \text{a.s.},
    \end{align*}

  \item for all bounded measurable $f\colon D \to \mathbb{R}$,
    \begin{align*}
      \Phi^{\Sigma, D}(f)
      \overset{\mathrm{law}}{\;=\;}
      \mathcal{N}\bigl(0, \sigma_{\Sigma}^{2}(f) \bigr),
    \end{align*}
    where
    \begin{align*}
      \sigma_{\Sigma}^{2}(f)
      \;\ldef\;
      c_{\sigma}\,
      \int_{D \times D} f(x)\, f(y)\, g_{D}^{\Sigma}(x, y)\, \mathrm{d}x\, \mathrm{d}y.
    \end{align*}
  \end{enumerate}
\end{definition}
We also extend this definition of the CGFF to let it act on functions in $H_{0}^{s}(D)$, for some $s > 0$.
\begin{theorem}[Scaling limit of inhomogeneous DGFF]\label{thm:scalingCGF_rg}
  Let $d \geq 2$ and $D \subset \mathbb{R}^{d}$ be a bounded, strongly regular domain. Suppose there exist $\theta \in (0, 1)$ and $p, q \in [1, \infty]$ such that Assumptions~\ref{ass:law}, \ref{ass:cluster} and \ref{ass:pq} hold. For any integrable function $f\colon D \to \mathbb{R}$, define
  \begin{align}
    \label{eq:def:functional:DGFF}
    \Phi_{n}^{D}(f)
    \;\ldef\;
    n^{d/2-1}
    \int_{D}
    f(x)\, \varphi^{n D}_{\lfloor nx \rfloor}\,
    \indicator_{\{ \lfloor nx \rfloor \in \mathcal{C}_{\infty}(\omega) \}}\,
    \mathrm{d}x,
    \qquad n \in \mathbb{N}.
  \end{align}
  Then, for $\prob_{0}$-a.e.\ $\omega$, under $\probPhi^{\omega}$
  \begin{align*}
    \Phi_{n}^{D}
    \underset{n \to \infty}{\overset{\text{law}}{\;\longrightarrow\;}}
    \Phi^{\Sigma, D},
  \end{align*}
  in the strong topology of $H^{-s}(D)$ for any $s > d/2$, where $\Phi^{\Sigma, D}$ is a CGFF with para\-meters $c_{\sigma} = \prob[0 \in \mathcal{C}_{\infty}]$ and $\Sigma$ being the diffusivity matrix in Theorem~\ref{thm:QFCLT:RG}. In particular,  $\Phi^{\Sigma, D}(f) \overset{\text{law}}{=} \mathcal{N}(0, \sigma_{\Sigma}^{2}(f))$ for any $f \in H_{0}^{s}(D)$ where
  \begin{align}\label{def:var_cont_CGF}
    \sigma_{\Sigma}^{2}(f)
    \;=\;
    \prob\bigl[ 0 \in \mathcal{C}_{\infty} \bigr]\,
    \int_{D \times D} f(x)\, f(y)\, g_{D}^{\Sigma}(x, y)\, \mathrm{d}x\, \mathrm{d}y,
    \qquad f \in H_{0}^{s}(D).
  \end{align}
\end{theorem}
\begin{remark} 
 (i) Note that in Theorem~\ref{thm:scalingCGF_rg} the height variables are scaled by $n^{d/2-1}$ while the conventional scaling for a central limit theorem is $n^{d/2}$. The different scaling is required due to strong correlations of the  height variables (cf.\ \cite{NS97, BS11}), in contrast to the scaling limit of the gradient field, which has weaker correlations and only requires the standard scaling $n^{d/2}$ (cf.\ \cite{GOS01, NW18}).

 (ii) Theorem~\ref{thm:scalingCGF_rg} can be shown under the slightly weaker condition where Assumption~\ref{ass:cluster}-(ii) is replaced by the statement in Remark~\ref{ass:cluster}-(ii). 
\end{remark}
The corresponding result for the homogeneous DGFF with constant conductances on $\mathbb{Z}^{d}$ can be found in \cite{Bi20}. More recently, a similar scaling limit for inhomogeneous DGFFs on convex or smooth domains or with periodic boundary conditions have been obtained in \cite{CR24} for uniformly elliptic conductances with finite range dependence.  For inhomogeneous DGFFs with uniformly elliptic ergodic conductances, existence and uniqueness of infinite-volume gradient Gibbs measures have been established in \cite{CK12, CK15}, and level set percolation has been studied in \cite{CN21}. 

Inhomogeneous DGFFs with random conductances also appear in a description for a class of gradient models with a non-convex potential, see \cite{BS11} and \cite[Sections~1.4 and 6]{Bi11}. In \cite{BS11} this connection has been used to show scaling limits similar to Theorem~\ref{thm:scalingCGF_rg} for those gradient models. We also refer to \cite[Theorem~A]{NS97}, \cite[Corollary~2.2]{GOS01}, \cite[Theorem~1.13]{AT21}  and \cite[Theorem~9]{NW18} for similar GFF scaling limits for interface models associated with gradient Gibbs measures on $\mathbb{Z}^{d}$, in particular see the fluctuation limit for interface models with stricly convex potentials on bounded domains in \cite{Mi11}.

Similar scaling limits for DGFFs on a class of manifolds have been obtained in \cite{CvG20}. Scaling limits for inhomogeneous bi-Laplacian fields, a.k.a.\ membrane models, towards continuum bi-Laplacian fields have been shown in \cite{CR24}, see \cite{CDH19} for the corresponding result for homogeneous membrane models. Such continuum bi-Laplacian fields also appear as the limit of the rescaled odometer functions in sandpile models, see \cite{CHR18}.

As our second main result, we establish a quenched local limit theorem for the Green's function, which asserts that, under diffusive scaling, the Green's function of the killed VSRW $X$ converges uniformly on compact sets to the Green kernel of the killed Brownian motion with covariance matrix $\Sigma^{2}$.

For any $\omega \in \Omega^{*}$ and $n \in \mathbb{N}$, let $\pi_{n}\colon \mathbb{R}^{d} \to \mathcal{C}_{\infty}(\omega)$ be a function that maps any $x \in \mathbb{R}^{d}$ to a closest point of $nx$ in $\mathcal{C}_{\infty}(\omega)$, where we choose a fixed ordering on $\mathbb{Z}^{d}$ to break ties. Moreover, we write $\operatorname{dist}(x, A) \ldef \inf\{|x - a|_{2} : a  \in A\}$ to denote the distance in $\mathbb{R}^{d}$ between a point $x \in \mathbb{R}^{d}$ and a non-empty subset $A \subset \mathbb{R}^{d}$.
\begin{theorem}[Quenched local limit theorem for Green's function]\label{thm:LCLT_green}
  Let $d \geq 2$ and $D \subset \mathbb{R}^{d}$ be a bounded, strongly regular domain. Suppose there exist $\theta \in (0, 1)$ and $p, q \in [1, \infty]$ such that Assumptions~\ref{ass:law},~\ref{ass:cluster} and~\ref{ass:pq} hold. For any $\varepsilon, \delta > 0$, set
  \begin{align} \label{eq:defK_eps}
    K_{\varepsilon, \delta}
    \;\ldef\;
    \big\{
      (x, y) \in D \times D :
      \operatorname{dist}(x, \partial D) \wedge \operatorname{dist}(y, \partial D)
      \geq \delta, |x - y|_{2} \geq \varepsilon
    \big\}.
  \end{align}
  Then, for any $\varepsilon, \delta > 0$ with $0 < \varepsilon < \delta$ such that $K_{\varepsilon, \delta} \neq \emptyset$ and $\prob_{0}$-a.e.\ $\omega$,
  \begin{align*}
    \lim_{n \to \infty} \sup_{(x, y) \in K_{\varepsilon, \delta}}
    \biggl|
      n^{d-2} g^{\omega}_{n D}(\pi_{n}(x), \pi_{n}(y)) -  \frac{g_{D}^{\Sigma}(x, y)}{\mathbb{P}[0 \in \mathcal{C}_{\infty}]}
    \biggr|
    \;=\;
    0,
  \end{align*}
  where $\Sigma$ is as in Theorem~\ref{thm:QFCLT:RG}.
\end{theorem}
\begin{remark} \label{rem:ass_lattice}
  Suppose that $\prob[\omega(e) > 0] = 1$ for $e \in E_{d}$. Under this additional assumption the cluster $\mathcal{C}_{\infty}$ a.s.\ coincides with the lattice $\mathbb{Z}^{d}$. Since the lattice $\mathbb{Z}^{d}$ satisfies Assumption~\ref{ass:cluster} for every $\theta \in (0, 1)$, the condition on $p$ and $q$ in \eqref{eq:cond:pq_rg} is equivalent to the condition
  \begin{align} \label{eq:moment_cond}
    \mean\bigl[ \omega(e)^{p} \bigr] \;<\; \infty
    \qquad \text{and} \qquad
    \mean\bigl[\omega(e)^{-q}\bigr]
    \;<\;
    \infty,
    \qquad \forall\, e \in E_{d}.
  \end{align}
  for $p, q \in (1, \infty)$ satisfying $1/p + 1/q < 2/d$. However, when working on the full lattice $\mathbb{Z}^{d}$, this moment condition can even be replaced by the slightly weaker condition that \eqref{eq:moment_cond} holds for $p, q \in [1, \infty)$ such that  $1/p + 1/q < 2/(d-1)$. In fact, in our proofs of Theorems~\ref{thm:scalingCGF_rg} and~\ref{thm:LCLT_green} the moment condition is only needed to ensure that the QFCLT holds and to control averages of $\omega(e)^{p}$ and $\omega(e)^{-q}$ appearing in oscillation and maximal inequalities for harmonic functions (see Theorems~\ref{thm:osc} and~\ref{thm:max_ineq:harm} below). On the full lattice $\mathbb{Z}^{d}$, those statements have been shown under the weaker condition on $p$ and $q$ in \cite{BS20, BS22}.
\end{remark}
In view of \eqref{eq:covField}, Theorem~\ref{thm:LCLT_green} immediately implies the following scaling limit for the covariances of the inhomogeneous DGFF.
\begin{coro} \label{coro:scalingCov}
  Under the assumptions of Theorem~\ref{thm:LCLT_green}, for any $0 < \varepsilon < \delta$ and for $\prob_{0}$-a.e.\ $\omega$,
  \begin{align*}
    \lim_{n \to \infty} n^{d-2}\,
    \covPhi^{\omega}\Bigl[
      \varphi^{n D}_{\pi_n(x)},  \varphi^{n D}_{\pi_n(y)}
    \Bigr]
    \;=\;
    g_{D}^{\Sigma}(x, y),
  \end{align*}
  uniformly in $(x, y) \in K_{\varepsilon, \delta}$.
\end{coro}
In dimension $d \geq 3$ the behaviour of the Green's function of $X$ on the full space, defined by $g^{\omega}(x, y) \ldef \int_{0}^{\infty} p_{t}^{\omega}(x, y)\, \mathrm{d}t$ is already quite well understood, at least in the more classical framework of random walks on $\mathbb{Z}^d$ or supercritical i.i.d.\ percolation clusters. For instance, precise estimates and asymptotics in case of general non-negative i.i.d.\ conductances have been shown \cite[Theorem~1.2]{ABDH13} and a local limit theorem for $g^{\omega}$ in the case of ergodic conductances satisfying \eqref{eq:moment_cond} in \cite[Theorem~1.14]{ADS16} (cf.\ also \cite[Theorem~5.2]{Ge20}). For local limit theorems for the associated heat kernel $p_{t}^{\omega}(x, y)$ we refer to \cite{ABDH13, BKM15,ADS16, AT21, ACS21, BS22}. In dimension $d = 2$, under the same assumptions on the conductances, precise asymptotics for the associated potential kernel have been obtained in \cite[Theorem~1.2]{ADS20} and for the Dirichlet Green's function on the diagonal in \cite[Theorem~1.3-(i)]{ADS20} (cf.\ \eqref{eq:green_killed_ondiag} below).
On the other hand, Theorem~\ref{thm:LCLT_green} constitutes a refinement of the near diagonal estimate on $g_{B(z, n)}^{\omega}(x, y)$ in \cite[Theorem~1.3-(ii)]{ADS20}, see the discussion in \cite[Remark~1.4-(v)]{ADS20}. For uniformly elliptic i.i.d.\ conductances on supercritical percolation clusters a quantitative homogenization theorem with optimal rates of convergence has been proven in \cite{DG21}.

\subsection{Main ideas and methods}
The proof of Theorem~\ref{thm:scalingCGF_rg} has two technical main ingredients, namely an exit time estimate for the random walk $X$ and an extension of the QFCLT in Theorem~\ref{thm:QFCLT:RG} to certain integral functionals for a class of initial distributions of $X$. To get the required exit time estimate we interpret the mean exit time from a ball as a solution of a Poisson problem with Dirichlet boundary conditions. Then, we use a version of De Giorgi's iteration scheme to show a maximal inequality for such equations in Proposition~\ref{prop:Max_ineq_poisson} to get an upper estimate on the mean exit time (see Corollary~\ref{cor:exit} below). We refer to \cite{FHS19} for a similar maximal inequality on $\mathbb{Z}^{d}$. Another similar maximal inequality has been shown in \cite{ADS15} for the corrector equation in order to derive the sublinearity of the corrector by Moser iteration techniques. In analytic terms, the treatment of the Poisson problem for the mean exit times is less demanding due to the fact that the solutions have finite support, which allows us to work under weaker conditions. The established exit time estimates will be used to get the aforementioned extension of the QFCLT in Theorem~\ref{thm:QFCLT:RG} where the convergence in \eqref{eq:convFCLT} is upgraded to hold also for functionals $F$ on the Skorohod space of the form $F(w) = \int_{0}^{\tau_{D}(w)} h(w_{t})\, \mathrm{d}t$ and for a general class of starting distributions of the random walk, see Theorem~\ref{thm:FCLT:Poisson} below. Our assumption that the domain $D$ is strongly regular is needed here to ensure the continuity of such functionals $F$. As a first step in our proof of Theorem~\ref{thm:FCLT:Poisson}, we establish an extension of the QFCLT, still for continuous and bounded functionals $F$, for a class of initial measures for the random walk (see Theorem~\ref{thm:QFCLT:RG:initial}) by using purely ergodic- and measure-theoretic arguments. Under stronger conditions, similar but different extensions of the QFCLT allowing for arbitrary starting points have been shown in \cite{FRS21,CCK15}.

The local limit theorem for the Green's function stated in Theorem~\ref{thm:LCLT_green} follows from the extended QFCLT in Theorem~\ref{thm:FCLT:Poisson} together with a H\"older continuity estimate for $g^{\omega}_{n D}(\cdot, y)$, see Proposition~\ref{prop:Holdercont}. The approach is analogous to the one to obtain a local limit theorem for the heat kernel established in \cite{BH09}. Based on this idea, general criteria for such local limit theorems have been stated e.g.\ in \cite{CH08, ACS21} and recently in a quantitative form in \cite{ACK24}. In order to get the required H\"older continuity estimates for $g^{\omega}_{n D}(\cdot, y)$ we borrow the corresponding estimate for $\mathcal{L}^{\omega}$-harmonic functions on balls in \cite{ACS21}, exploiting the fact that the Green's function is harmonic away from the diagonal. This is the reason why in Theorem~\ref{thm:LCLT_green} the supremum is not taken over points close to the diagonal.

\subsection{Open problem}
We finish this section by stating an open problem, namely to derive, in the setting of the present paper, the following convergence of the centred maximum of a two-dimensional inhomogeneous DGFF towards a randomly shifted Gumbel distribution.
\begin{conjecture} \label{conj:max}
  Let $d = 2$ and suppose that Assumption~\ref{ass:law}-(i) and \eqref{eq:moment_cond} hold, in particular $\prob[\omega(e) > 0] = 1$ for $e \in E_{d}$. For $n \in \mathbb{N}$ set $M_{n} \ldef \max_{x \in [-n, n] \cap \mathbb{Z}^{d}} \varphi_{x}^{\Lambda_{n}}$ with $\Lambda_{n} \ldef [-n, n]^{d}$ and
  \begin{align*}
    m_{n}
    \;\ldef\;
    \sqrt{\bar{g}}\,
    \biggl(\sqrt{2d} \log n - \frac{3}{2\sqrt{2d}} \log\log n \biggr)
  \end{align*}
  with $\bar{g}$ as in \eqref{eq:green_killed_ondiag} below. Then, for $\prob$-a.e.\ $\omega$, under $\probPhi^{\omega}$, $M_{n} - m_{n}$ converges in law to a randomly shifted Gumbel distribution, i.e.\ there exist $\beta^{*} \in (0, \infty)$ and a random variable $Z$ such that, for all $t \in \mathbb{R}$,
  \begin{align*}
    \lim_{n \to \infty} \probPhi^{\omega}\bigl[M_{n} - m_{n} \leq t\bigr]
    \;=\;
    \meanPhi^{\omega}\Bigl[ \me^{-\beta^{*} Z \me^{-\sqrt{2d}t}} \Bigr].
  \end{align*}
\end{conjecture}
Moreover, $Z$ is conjectured to be the limit of a sequence of random variables $(Z_{n})_{n \in \mathbb{N}}$ that resembles strongly the derivative martingale that occurs in the study of critical Gaussian multiplicative chaos and various branching processes, cf.\ e.g.\ the discussion in \cite{DRZ17}. 

One possible strategy to show Conjecture~\ref{conj:max} is to verify the assumptions in \cite{DRZ17}, where the convergence of the centered maximum has been shown for a class of logarithmically correlated Gaussian fields. On one hand, the first part of \cite[Assumption~(A.0)]{DRZ17} follows from \cite[Theorem~1.3-(i)]{ADS20}, stating that, $\prob$-a.s., for any $z \in \mathbb{Z}^{d}$, $\delta \in (0, 1)$ and $x \in B(z, (1 - \delta)n)$, 
\begin{align} \label{eq:green_killed_ondiag}
  \lim_{n \to \infty}\, \frac{1}{\log n}\, g_{B(z, n)}^{\omega}(x, x)
  \;=\;
  \frac{1}{\pi \sqrt{\det \Sigma^{2}} \mean\bigl[ \mu^{\omega}(0) \bigr]}
  \;\rdef\;
  \bar{g}.
\end{align}
Further, \cite[Assumption~(A.3)]{DRZ17} is a direct consequence of Theorem~\ref{thm:LCLT_green}. On the other hand, while Theorem~\ref{thm:LCLT_green} can be used to show \cite[Assumption~(A.1)]{DRZ17} and the second part of \cite[Assumption~(A.0)]{DRZ17} for points that are sufficiently far apart, the remaining assumptions, in particular \cite[Assumption~(A.2)]{DRZ17} requires control of the Green's function near the diagonal beyond the regime in Theorem~\ref{thm:LCLT_green}. The criterion in \cite{DRZ17} have recently been generalized in \cite{SZ24}, where, as an application, Conjecture~\ref{conj:max} has been proven for Gaussian fields under uniformly elliptic i.i.d.\ conductances on supercritical percolation clusters, relying on the quantitative homogenization results in \cite{AD18, DG21}.
\medskip

The rest of the paper is organised as follows. In Section~\ref{sec:poisson} we show a maximal inequality for Poisson equations and apply it to obtain exit time estimates. In Section~\ref{sec:CLT_intFunctionals} we extend the QFCLT to hold for certain integral functionals and a class of initial distributions for the random walk, and in Section~\ref{sec:scaling_limitCGF} we deduce Theorem~\ref{thm:scalingCGF_rg} from it. Finally, in Section~\ref{sec:lclt} we first discuss H\"older regularity and a-priori bounds for the Green's functions with Dirichlet boundary conditions and then use them to show Theorem~\ref{thm:LCLT_green}. Throughout the paper we write $c$ to denote a positive constant which may change on each appearance. Constants denoted $C_{i}$ will be the same through the paper.

\section{A maximum inequality for Poisson equations and applications} \label{sec:poisson}
In this section we will derive a maximal inequality for Poisson equations on a more general class of graphs. As an application we shall obtain an upper bound for the mean exit time for the random walk.

\subsection{Notation and preliminaries}
For any $p \in [1, \infty]$ and any non-empty, finite $A \subset \mathcal{C}_{\infty}(\omega)$, we define space-averaged $\ell^{p}$-norm on functions $f\colon A \to \mathbb{R}$ by 
\begin{align*}
  \Norm{f}{p, A}
  \;\ldef\;
  \bigg(
    \frac{1}{|A|}\; \sum_{x \in A}\, |f(x)|^p\
  \bigg)^{\!\!1/p}
  \qquad \text{and} \qquad
  \Norm{f}{\infty, A} \;\ldef\; \max_{x \in A} |f(x)|,
\end{align*}
where $|A|$ denotes the cardinality of the set $A$. We define two discrete measures $\mu^{\omega}$ and $\nu^{\omega}$ on $\mathbb{Z}^{d}$ by
\begin{align*}
  \mu^{\omega}(x) \;\ldef\; \sum_{y \sim x} \omega(\{x, y\})
  \qquad \text{and} \qquad
  \nu^{\omega}(x) \;\ldef\; \sum_{y \sim x} \frac{1}{\omega(\{x, y\})} \indicator_{\{\{x,y\}\in \mathcal{O}(\omega)\}}.
\end{align*}
Moreover, we define a \emph{symmetric, bilinear form}, $(\mathcal{E}^{\omega}, \mathcal{D}(\mathcal{E}^{\omega}))$ by
\begin{align} \label{eq:def:DF}
  \left\{
    \begin{array}{rcl}
      \mathcal{E}^{\omega}(f, g) & \!\!\ldef\!\! &
      {\displaystyle \sum\nolimits_{e \in E(\mathcal{C}_{\infty}(\omega))} \omega(e)\, (\nabla f)(e) (\nabla g)(e)}
      \\[1.5ex]
      \mathcal{D}(\mathcal{E}^{\omega})
      &\!\!\ldef\!\!&
      \bigl\{ f \in \ell^2(\mathcal{C}_{\infty}(\omega)) : \mathcal{E}^{\omega}(f, f) < \infty \bigr\}
    \end{array}
  \right.,
\end{align}
where we choose for any edge $e \in E(\mathcal{C}_{\infty}(\omega))$ an assignment $e^{+}$ and $e^{-}$ of vertices such that $e = \{e^{+}, e^{-}\}$ to which we refer to as head and tail of the edge $e$, and define the \emph{discrete derivative} of a function $f\colon \mathcal{C}_{\infty}(\omega) \to \mathbb{R}$ by
\begin{align*}
  (\nabla f)(e)
  \;\ldef\;
  f(e^{+}) - f(e^{-}),
  \qquad e \in E(\mathcal{C}_{\infty}(\omega)).
\end{align*}
Nothing what follows depends on the particular choice of the assignments of heads and tails.

Notice that $(\mathcal{E}^{\omega}, \mathcal{D}(\mathcal{E}^{\omega}))$ is closable and, hence, a \emph{Dirichlet form}. The associated Hunt process $X = (X_{t})_{t \geq 0}$ is a continuous-time Markov chain on $\mathcal{C}_{\infty}(\omega)$ whose \emph{generator} acts on bounded functions $f\colon \mathcal{C}_{\infty}(\omega) \to \mathbb{R}$ as
\begin{align*}
  \big(\mathcal{L}^{\omega} f)(x)
  \;=\; 
  \sum_{y \sim x} \omega(\{x, y\})\, \bigl(f(y) - f(x)\bigr).
\end{align*}
Further, by applying the discrete version of the Gauss-Green formula, we find that, on the Hilbert space $\ell^{2}(\mathcal{C}_{\infty}(\omega))$, the Dirichlet form can be expressed in terms of the generator. Indeed, for any $f, g\colon \mathcal{C}_{\infty}(\omega) \to \mathbb{R}$ with finite support
\begin{align}\label{eq:DF:L}
  \mathcal{E}^{\omega}(f,g)
  &\;\ldef\;
  \scpr{f}{-\mathcal{L}^{\omega} g}{\mathcal{C}_{\infty}(\omega)}
  \;=\;
  \scpr{\nabla f}{\omega \nabla g}{E(\mathcal{C}_{\infty}(\omega))}.
\end{align}
\begin{prop}[Weighted Sobolev inequality] \label{prop:Sobolev:ineq}
  Let $d \geq 2$, $q \in [1, \infty]$, and suppose there exists $\theta \in (0, 1)$ such that Assumptions~\ref{ass:law} and~\ref{ass:cluster}-(i) hold. Then, for any $\theta' \in (\theta, 1)$ there exists $C_{\mathrm{S}} \equiv C_{\mathrm{S}}(\theta'\!, d, q) \in (0, \infty)$ such that for any $\omega \in \Omega_{\mathrm{reg}} \cap \Omega_{0}^{*}$, $R \geq 1$ and $n \geq N_{0}(\omega) \vee R^{\theta / (\theta' - \theta)}$ the following hold. For any $x_{0} \in B^{\omega}(0, R n)$ and any $f\colon \mathcal{C}_{\infty}(\omega) \to \mathbb{R}$ with $\supp(f) \subset B^{\omega}(x_{0}, n)$, 
    \begin{align}\label{eq:Sobolev:weighted}
      \Norm{f^2}{\rho, B^{\omega}(x_{0}, n)}
      \;\leq\;
      C_{\mathrm{S}}\, n^{2}\, \Norm{\nu^{\omega}}{q, B^{\omega}(x_{0}, n)}\,
      \frac{\mathcal{E}^{\omega}(f, f)}{|B^{\omega}(x_{0}, n)|},
    \end{align}
    where
    \begin{align}\label{eq:def:rho}
      \rho
      \;\equiv\;
      \rho(d', q)
      \;=\;
      \frac{d'}{d' - 2 + d'/q} \quad \text{with} \quad d' \ldef (d-\theta')/(1-\theta').
    \end{align}
\end{prop}
\begin{proof}
  Fix some $\omega \in \Omega_{\mathrm{reg}} \cap \Omega_{0}^{*}$ and $R \geq 1$. Then, by Assumption~\ref{ass:cluster}-(i), we have that $N_{0}(\omega) < \infty$. First, we claim that, for any $\theta' > \theta$, $n \geq N_{0}(\omega) \vee R^{\theta / (\theta' - \theta)}$ and $x_{0} \in B^{\omega}(0, R n)$, the ball $B^{\omega}(x_{0}, n)$ is $\theta'$-very regular. Indeed, according to Assumption~\ref{ass:cluster}-(i), the $B^{\omega}(0, Rn)$ is $\theta$-very regular. This implies that, for any $x_{0} \in B^{\omega}(0, R n)$, the ball $B^{\omega}(x, r)$ is regular for any $x \in B^{\omega}(x_{0}, n)$ and $r \geq (R n)^{\theta}$. Since $n^{\theta'} \geq (R n)^{\theta}$ for any $n \geq R^{\theta / (\theta' - \theta)}$, it follows that the ball $B^{\omega}(x_{0}, n)$ is $\theta'$-very regular.
  
  Further, by \cite[Lemma~2.10]{DNS18}, an isoperimetric inequality on large sets holds, that is, there exists $C_{\mathrm{iso}} \in (0, \infty)$ such that
  \begin{align*}
    |\partial^{\omega} A| \;\geq\; C_{\mathrm{iso}}\, |A|^{(d-1)/d},
    \qquad \forall\, A \subset B^{\omega}(x_{0}, n) \;\text{ with }\; |A| \geq n^{\theta'}.
  \end{align*}
  In conjunction with the volume regularity, this implies the following Sobolev inequality on large scales, see \cite[Proposition~3.5]{DNS18}. There exists $C_{\mathrm{S}_{1}}(\theta'\!, d) \in [1, \infty)$ such that, for all $n \geq N_{0}(\omega) \vee R^{\theta / (\theta' - \theta)}$,
  \begin{align}\label{eq:sobolev}
    \Norm{u}{d'/(d'-1), B^{\omega}(x_{0}, n)}
    \;\leq\;
    C_{\mathrm{S_{1}}}(\theta'\!, d)\, \frac{n}{|B^{\omega}(x_{0}, n)|}\,
    \sum_{e \in  E(\mathcal{C}_{\infty}(\omega))} \bigl|(\nabla u)(e)\bigr|,
  \end{align}
  for every function $u\colon \mathcal{C}_{\infty}(\omega) \to \mathbb{R}$ with $\supp{u} \subset B^{\omega}(x_{0}, n)$ with $d'$ defined as in \eqref{eq:def:rho}. Given the Sobolev inequality in \eqref{eq:sobolev}, the weighted Sobolev inequality in \eqref{eq:Sobolev:weighted} follows as in the proof of \cite[Equation (28)]{ADS15}.
\end{proof}
In the following we need to ensure that $\rho>1$ which is the case if and only if $q > d'/2$. Note that in view of the definition of $d'$ in \eqref{eq:def:rho} the condition on $p$ and $q$ in \eqref{eq:cond:pq_rg} reads $1/p + 1/q < 2/d'$.

\subsection{Maximal inequality}
In this section we will use De Giorgi's iteration scheme to obtain, for any $\omega \in \Omega_{\mathrm{reg}} \cap \Omega_{0}^{*}$, a maximal inequality for solutions of the Poisson equation
\begin{align}\label{eq:Poisson1}
  \left\{
    \begin{array}{rcll}
      \mathcal{L}^{\omega} u & \!\!=\!\! &
      {\displaystyle -\frac{1}{n^{2}}f}
      &\quad \text{on } B^{\omega}(x_{0}, n),
      \\[1.5ex]
      u &\!\!=\!\!& 0 &\quad \text{on } B^{\omega}(x_{0}, n)^{\mathrm{c}},    
    \end{array}
  \right.
\end{align}
with $f\colon \mathcal{C}_{\infty}(\omega) \to \mathbb{R}$, $R \geq 1$, $x_{0} \in B^{\omega}(0, Rn)$ and $n \in \mathbb{N}$ sufficiently large.
\begin{prop} \label{prop:Max_ineq_poisson}
  Let $d \geq 2$, and suppose there exists $\theta \in (0, 1)$ such that Assumptions~\ref{ass:law} and~\ref{ass:cluster}-(i) hold. Then, for any $p, q \in (1, \infty]$ such that $1/p + 1/q < 2(1-\theta)/(d-\theta)$ there exists $C_{2} \equiv C_{2}(\theta, d, p, q) \in [1, \infty)$ and $\kappa_{1} \equiv \kappa_{1}(\theta, d, p, q) \in (0, \infty)$ such that, for any $\omega \in \Omega_{\mathrm{reg}} \cap \Omega_{0}^{*}$, $R \geq 1$ and $n \geq N_{0}(\omega) \vee R^{\kappa_{1}}$, the following hold.  For any $x_{0} \in B^{\omega}(0, R n)$, $h \geq 0$ and any solution $u$ of \eqref{eq:Poisson1},
  \begin{align}\label{eq:max:ineq}
    \max_{z \in B^{\omega}(x_{0}, n)} u(z)
    \;\leq\;
    h + C_{2}
    \Bigl(
      \Norm{\nu^{\omega}}{q, B^{\omega}(x_{0}, n)}\, \Norm{f}{p, B^{\omega}(x_{0}, n)}
    \Bigr)^{\! \frac{\alpha_{*}}{2 + \alpha_{*}}}\,
    \Norm{(u-h)_{+}}{2p_{*}, B^{\omega}(x_{0}, n)}^{2/(2+\alpha_{*})},
  \end{align}
  where $\alpha_{*} \ldef \rho / (\rho - p_{*})$ and $p_{*} = p /(p-1)$ denotes the Hölder conjugate of $p$. In particular, there exists $C_{3} \equiv C_{3}(\theta, d, p, q) < \infty$ such that
  \begin{align}\label{eq:global:max:ineq}
    \max_{z \in B^{\omega}(x_{0}, n)} u(z)
    \;\leq\;
    C_{3}\, \Norm{\nu^{\omega}}{q, B^{\omega}(x_{0}, n)}\, \Norm{f}{p, B^{\omega}(x_{0}, n)}.
  \end{align}
\end{prop}
\begin{proof}
  Let $p, q \in (1, \infty]$ be such that $1/p + 1/q < 2(1-\theta)/(d-\theta)$, and note that the function $(0, 1) \ni \theta \mapsto 2(1-\theta) / (d - \theta)$ is monoton decreasing. Further, let us choose $\theta' \in (\theta, 1)$ such that
  \begin{align}\label{eq:choice_theta_prime}
    \frac{2(1-\theta')}{d-\theta'}
    \;=\;
    \frac{1-\theta}{d-\theta} + \frac{1}{2} \biggl(\frac{1}{p} + \frac{1}{q}\biggr),
  \end{align}
  and set $\kappa_{1} \equiv \kappa_{1}(\theta, d, p, q) \ldef \theta / (\theta' - \theta)$. Recall the definition of $\rho$ and $d'$  in~\eqref{eq:def:rho}. Then note that the choice of $\theta'$ in \eqref{eq:choice_theta_prime} ensures that $1/p + 1/q < 2/d'$. Hence, defining $p_{*}$ and $\alpha_*$ as in the statement, we have that $\alpha \ldef \rho/p_{*} > 1$ and therefore $\alpha_{*} < \infty$. Further, fix some $\omega \in \Omega_{\mathrm{reg}} \cap \Omega_{0}^{*}$, $R \geq 1$ $x_{0} \in B^{\omega}(0, R n)$ and $n \geq N_{0}(\omega) \vee R^{\kappa_{1}}$. To simplify notation, we set $B \ldef B^{\omega}(x_{0}, n)$ in the following argument.
  
  \smallskip 
  \textsc{Step 1.} Let $0 \leq k < l < \infty$ be fixed constants. Then, by Hölder's inequality and the weighted Sobolev inequality as stated in \eqref{eq:Sobolev:weighted}, we first get
  \begin{align*}
    \Norm{(u-l)_{+}^{2}}{p_{*}, B}
    &\;\leq\;
    \Norm{(u-l)_{+}^{2}}{\alpha p_{*}, B}\,
    \Norm{\indicator_{\{u \geq l\}}}{\alpha_{*} p_{*}, B}
    % \\[.5ex]
    % &\;\leq\;
    % \Norm{(u-k)_{+}^{2}}{\alpha p_{*}, B(x_{0}, n)}\,
    % \Norm{\indicator_{\{u \geq l\}}}{p_{*}, B(x_{0}, n)}^{1/\alpha_{*}}
    \\[.5ex]
    &\overset{\!\eqref{eq:Sobolev:weighted}\!}{\;\leq\;}
    C_{\mathrm{S}}\, n^{2} \, \Norm{\nu^{\omega}}{q, B}\,
    \frac{\mathcal{E}^{\omega}\bigl((u-l)_{+}, (u-l)_{+}\bigr)}{|B|}\,
    \Norm{\indicator_{\{u \geq l\}}}{p_{*}, B}^{1/\alpha_{*}}.
  \end{align*}
  On the other hand,
  \begin{align*}
    \scpr{(u-l)_{+}}{-\mathcal{L}^{\omega} u}{\mathcal{C}_{\infty}(\omega)}
    &\overset{\!\eqref{eq:Poisson1}\!}{\;=\;}
    \scpr{(u-l)_{+}}{f/n^{2}}{B}
    % \nonumber\\[.5ex]
    % &
    \;\leq\;
    \frac{|B|}{n^{2}}\,
    \Norm{f}{p, B}\, \Norm{(u-l)_{+}}{p_{*}, B}.
  \end{align*}
  Since $\bigl(\nabla (u - l)_{+}\bigr)(e)^{2} \leq \bigl(\nabla (u - l)_{+}\bigr)(e) (\nabla u)(e)$ for any $e \in E(\mathcal{C}_{\infty}(\omega))$, we further obtain that $\mathcal{E}^{\omega}\bigl((u - l)_{+}, (u - l)_{+}\bigr) \leq \scpr{(u - l)_{+}}{-\mathcal{L}^{\omega} u}{\mathcal{C}_{\infty}(\omega)}$.  Thus,
  \begin{align}\label{eq:estimate:DF}
    \frac{\mathcal{E}^{\omega}\bigl((u-l)_{+}, (u-l)_{+}\bigr)}{|B|}
    \;\leq\;
    n^{-2}\, \Norm{f}{p, B}\, \Norm{(u-l)_{+}}{p_{*}, B}.
  \end{align}
  By combining the estimates above and using Markov's inequality, we conclude that
  \begin{align*}
    \Norm{(u - l)_{+}^{2}}{p_{*}, B}
    % &\;\leq\;
    % C_{\mathrm{S}}\, \Norm{\nu^{\omega}}{q, B}\, \Norm{f}{p, B}\,
    % \Norm{(u-l)_{+}}{p_{*}, B}\,
    % \Norm{\indicator_{\{u \geq l\}}}{p_{*}, B}^{1/\alpha_{*}}
    % \\[.5ex]
    &\;\leq\;
    \frac{%
      C_{\mathrm{S}}\, \Norm{\nu^{\omega}}{q, B}\, \Norm{f}{p, B} 
    }{(l - k)^{1+2/\alpha_{*}}}\,
    \Norm{(u - k)_{+}^{2}}{p_{*}, B}^{1 + 1/\alpha_{*}}. 
  \end{align*}
  By setting $\varphi(l) \ldef \Norm{(u-l)_{+}^{2}}{p_{*}, B}$ and $M \ldef C_{\mathrm{S}} \Norm{f}{p, B} \Norm{\nu^{\omega}}{q, B}$ we can rewrite the last inequality as
  \begin{align}\label{eq:iteration:max:ineq}
    \varphi(l)
    \;\leq\;
    \frac{M}{(l - k)^{1+2/\alpha_{*}}}\, \varphi(k)^{1+1/\alpha_{*}}.
  \end{align}
  
  \smallskip
  \textsc{Step 2.} For any $h \geq 0$ and $j \in \mathbb{N}_{0}$ define
  \begin{align*}
    k_{j} \;\ldef\; h + K(1-2^{-j}),
  \end{align*}
  with $K \ldef 2^{1 + \alpha_{*}} M^{\alpha_{*}/(2 + \alpha_{*})} \varphi(h)^{1/(2 + \alpha_{*})}$. We claim that
  \begin{align}\label{eq:iteration:claim}
    \varphi(k_{j}) \leq \frac{\varphi(h)}{r^{j}},
    \qquad \forall\, j \in \mathbb{N}_{0},
  \end{align}
  where $r \ldef 2^{2 + \alpha_{*}}$. Indeed, for $j = 0$ the assertion is immediate. Further, suppose that~\eqref{eq:iteration:claim} holds for any $i \in \{0, \ldots, j\}$. Then, in view of~\eqref{eq:iteration:max:ineq}, we obtain
  \begin{align*}
    \varphi(k_{j+1})
    &\;\leq\;
    M\, \biggl(\frac{2^{j + 1}}{K}\biggr)^{\!\!1 + 2/\alpha_{*}}\, \varphi(k_{j})
    \;\leq\;
    M\, \biggl(\frac{2^{j + 1}}{K}\biggr)^{\!\!1 + 2/\alpha_{*}}\,
    \biggl(\frac{\varphi(h)}{r^{j}}\biggr)^{\!\!1 + 1/\alpha_{*}}
    \;\leq\;
    \frac{\varphi(h)}{r^{j + 1}},
  \end{align*}
  which completes the proof of the claim \eqref{eq:iteration:claim}. Thus, by letting $j \to \infty$, we get that $\varphi(h + K) = 0$ which is equivalent to $\max_{x \in B} u(x) \leq h + K$. Hence, the assertion~\eqref{eq:max:ineq} follows directly from the definition of $K$ with $C_{2} \ldef 2^{1+\alpha_{*}} C_{\mathrm{S}}^{\alpha_{*}/(2 + \alpha_{*})}$.

  \smallskip
  \textsc{Step 3.} It remains to show \eqref{eq:global:max:ineq}. But, by using Jensen's inequality, the weighted Sobolev inequality as stated in Proposition~\ref{prop:Sobolev:ineq} and \eqref{eq:estimate:DF} with $l = 0$, we get
  \begin{align*}
    \Norm{(u)_{+}^{2}}{p_{*}, B}
    &\;\leq\;
    C_{\mathrm{S}}\, n^{2}\,
    \Norm{\nu^{\omega}}{q, B}\, \frac{\mathcal{E}^{\omega}\bigl((u)_{+}, (u)_{+}\bigr)}{|B|}
    \\[.5ex]
    &\overset{\mspace{-8mu}\eqref{eq:estimate:DF}\mspace{-8mu}}{\;\leq\;}
    C_{\mathrm{S}}\, \Norm{\nu^{\omega}}{q, B}\, \Norm{f}{p, B}\,
    \Norm{(u)_{+}}{p_{*}, B}
    \;\leq\;
    C_{\mathrm{S}}\, \Norm{\nu^{\omega}}{q, B}\, \Norm{f}{p, B}\,
    \max_{x \in B} u(x).
  \end{align*}
  By combining this estimate with~\eqref{eq:max:ineq} for $h = 0$, \eqref{eq:global:max:ineq} follows.
\end{proof}
As an application of this proposition we get an upper bound on the mean exit time for $X$ from $B^{\omega}(x_{0}, n)$. % Given $D \subset V$, we write $\tau_{D} = \tau_{D}(X) = \inf\bigl\{ t > 0 : X_{t} \not\in D \bigr\}$ for the first exit time of $X$ from $D$.
\begin{coro}\label{cor:exit}
  Let $d \geq 2$, and suppose there exists $\theta \in (0, 1)$ such  that Assumptions~\ref{ass:law} and~\ref{ass:cluster}-(i) hold. Let $q > (d - \theta)/(2(1-\theta))$. Then, there exists $C_{4} \equiv C_{4}(\theta, d, q)$ such that, for any $\omega \in \Omega_{\mathrm{reg}} \cap \Omega_{0}^{*}$, $R \geq 1$, $n \geq N_{0}(\omega) \vee R^{\kappa_{1}}$, $x_{0} \in B^{\omega}(0, Rn)$ and all $z \in B^{\omega}(x_{0}, n)$,
  \begin{align*}
    \Mean_{z}^{\omega}\bigl[ \tau_{B^{\omega}(x_{0}, n)} \bigr]
    \;\leq\;
    C_{4}\, \Norm{\nu^{\omega}}{q, B^{\omega}(x_{0}, n)}\,  n^2.
  \end{align*}
\end{coro}
\begin{proof}
  Recall that the probabilistic representation of the non-negative solution $u$ of~\eqref{eq:Poisson1} with $f\colon B^{\omega}(x_{0}, n) \to \mathbb{R}$, $f \equiv 1$ is given by $u(z) = n^{-2} \Mean_{z}^{\omega}\bigl[\tau_{B^{\omega}(x_{0}, n)}\bigr]$ for any $z \in B^{\omega}(x_{0}, n)$. Thus, the claim follows from \eqref{eq:global:max:ineq}.
\end{proof}

\section{A central limit theorem for integral functionals} \label{sec:CLT_intFunctionals}
\subsection{A quenched invariance principle for a class of initial measures}
In this subsection we extend the QFCLT as stated in Theorem~\ref{thm:QFCLT:RG} to a certain class of initial (signed) measures. Recall the notation $X^{n} = \bigl(n^{-1} X_{t n^{2}}\bigr)_{t \geq 0}$.
\begin{theorem}\label{thm:QFCLT:RG:initial}
  Let $d \geq 2$, and suppose that there exist $\theta \in (0, 1)$ and $p, q \in [1, \infty]$ such that Assumptions~\ref{ass:law}, \ref{ass:cluster}-(i) and~\ref{ass:pq} hold. Further, let $f\colon \mathbb{R}^{d} \to \mathbb{R}$ be a continuous function with bounded support. Then, $\prob_{0}$-a.s., for any $T > 0$ and every bounded continuous function $F\colon D\bigl([0, T], \mathbb{R}^{d}\bigr) \to \mathbb{R}$,
  \begin{align}
    \label{eq:QFCLT:RG:init}
    \lim_{n \to \infty}
    \frac{1}{n^{d}}\sum_{z \in \mathcal{C}_{\infty}(\omega)}\mspace{-6mu} f(z/n)
    \Mean_{z}^{\omega}\bigl[ F(X^{n}) \bigr]
    \;=\;
    \prob\bigl[0 \in \mathcal{C}_{\infty}\bigr]
    \int_{\mathbb{R}^{d}} f(x) \Mean_{x}^{\Sigma}\bigl[ F(W) \bigr]\, \mathrm{d}x,
  \end{align}
  where $\bigl(W, (\Prob_{x}^{\Sigma})_{x \in \mathbb{R}^{d}}\bigr)$ denotes a Brownian motion with non-degenerate, deterministic covariance matrix $\Sigma^{2} \ldef \Sigma \cdot \Sigma^{\mathrm{T}}$ as appearing in Theorem~\ref{thm:QFCLT:RG}. In particular, the null set does not depend on $f$.
\end{theorem}
Before we prove Theorem~\ref{thm:QFCLT:RG:initial} we recall briefly the Skorokhod $J_{1}$-topology with respect to which the spaces $D([0, T], \mathbb{R^{d}})$, $T > 0$, are Polish. For this purpose, let $\Lambda_{T}$ be the set of all continuous functions $\lambda\colon [0, T] \to [0, T]$ that are strictly increasing with $\lambda_{0} = 0$ and $\lambda_{T} = T$. Further, set
\begin{align*}
  \vvvert \lambda \vvvert_{[0, T]}
  \;\ldef\;
  \sup_{\substack{s, t \in [0, T] \\ s < t}}\;
  \biggl| \log \frac{\lambda_{t} - \lambda_{s}}{t-s} \biggr|
\end{align*}
and define for any $w, w' \in D([0, T], \mathbb{R}^{d})$,
\begin{align}
  \label{eq:def:d_[0,T]}
  d_{[0, T]}(w, w')
  \;\ldef\;
  \inf_{\lambda \in \Lambda_{T}}
  \Bigl(
    \vvvert \lambda \vvvert_{[0, T]} \vee \bigl\| w \circ \lambda - w' \bigr\|_{[0, T]}
  \Bigr),
\end{align}
where $\| w \|_{[0, T]} \ldef \sup_{t \in [0, T]} |w_{t}|_{2}$. Then, for each $T > 0$, it holds that $d_{[0, T]}$ is a metric and $D([0, T], \mathbb{R^{d}})$ equipped with $d_{[0, T]}$ is complete and separable, see e.g.\ \cite[Theorem~12.2]{Bi99}. A complete metric on $D([0, \infty), \mathbb{R}^{d})$ can be defined in terms of the metrics $d_{[0, T]}$, $T > 0$, where one needs to consider transformations of $w|_{[0,T]}$ and $w'|_{[0,T]}$ that are continuous at $T$. Alternatively, a metrisable topology for which $D([0, \infty), \mathbb{R}^{d})$ is a Polish space can also be described as follows. A sequence $(w^{n})_{n \in \mathbb{N}} \subset D([0, \infty), \mathbb{R}^{d})$ converges to some $w \in D([0, \infty), \mathbb{R}^{d})$, if and only if, there exists a sequence $(\lambda^{n}) \subset \Lambda_{\infty}$ such that
\begin{align*}
  \lim_{n \to \infty }\sup_{s \in [0, \infty)} \bigl| \lambda_{s}^{n} - s \bigr|
  \;=\;
  0
  \qquad \text{and} \qquad
  \lim_{n \to \infty} \bigl\| w^{n}\circ \lambda^{n} - w\bigr\|_{[0,T]}
  \;=\;
  0,
  \quad \forall\, T \in \mathbb{N},
\end{align*} 
where $\Lambda_{\infty}$ denotes the set of all continuous functions from $[0, \infty)$ onto itself that are strictly increasing with $\lambda_{0} = 0$ and $\lim_{t \to \infty} \lambda_{t} = \infty$.
\begin{proof}[Proof of Theorem~\ref{thm:QFCLT:RG:initial}]
  Let $T > 0$ and $F\colon D([0, T], \mathbb{R}^{d}) \to \mathbb{R}$ be a bounded continuous function. First, notice that, for any $\omega \in \Omega_{0}^{*}$ and $z \in \mathcal{C}_{\infty}(\omega)$, the law of $X$ under $\Prob_{z}^{\omega}$ coincides with the law of $z+X$ under $\Prob_{0}^{\tau_{z} \omega}$. This together with the spatial homogeneity of the limiting Brownian motion $W$ implies that
  \begin{align*}
    \Mean_{z}^{\omega}\bigl[ F(X^{n}) \bigr]
    \;=\;
    \Mean_{0}^{\tau_{z} \omega}\bigl[ F(z/n + X^{n})\bigr]
    \qquad \text{and} \qquad
    \Mean_{x}^{\Sigma}\bigl[ F(W)\bigr] \;=\; \Mean_{0}^{\Sigma}\bigl[ F(x + W) \bigr],
  \end{align*}
  for any $z \in \mathbb{Z}^{d}$, $n \in \mathbb{N}$, and $x \in \mathbb{R}^{d}$.  Thus, we need to show that, $\prob_{0}$-a.s.,
  \begin{align}
    \label{eq:QFCLT:RG:initial:shifted}
    \lim_{n \to \infty}
    \frac{1}{n^{d}}\mspace{-6mu} \sum_{z \in \mathcal{C}_{\infty}(\omega)}\mspace{-6mu}
    f(z/n) \Mean_{0}^{\tau_{z} \omega}\bigl[ F(z/n + X^{n}) \bigr]
    \;=\;
    \prob\bigl[0 \in \mathcal{C}_{\infty}\bigr]
    \int_{\mathbb{R}^{d}}
      f(x) \Mean_{0}^{\Sigma}\bigl[ F(x + W) \bigr]\,
    \mathrm{d}x.
  \end{align}
  The proof we are going to present comprises two steps.
  \smallskip
  
  \textsc{Step 1.}
  Let us additionally assume that $F$ is uniformly continuous. Fix some arbitrary $\varepsilon > 0$. Then, there exists $\delta_{\varepsilon} > 0$ such that $|F(w) - F(w')| < \varepsilon$ for any $w, w' \in D([0, T], \mathbb{R}^{d})$ with $d_{[0, T]}(w, w') < \delta_{\varepsilon}$. Since $\supp f \subset \mathbb{R}^{d}$ is assumed to be bounded, there exist $K \equiv K(\delta_{\varepsilon}) \in \mathbb{N}$ and $x_{1}, \ldots, x_{K} \in \mathbb{R}^{d}$ such that, setting $Q_{k} \ldef x_{k} + (-\delta_{\varepsilon}/2, \delta_{\varepsilon}/2]^{d}$ for any $k \in \{1, \ldots, K\}$, we have that $Q_{k} \cap Q_{\ell} = \emptyset$ for all $k, \ell \in \{1, \ldots, K\}$ with $k \ne \ell$ and $\supp f \subset \bigcup_{k = 1}^{K} Q_{k}$. Further, let $n Q_{k} \ldef \{z \in \mathbb{Z}^{d} : z/n \in Q_{k}\}$ for any $n \in \mathbb{N}$. In what follows we choose $n \in \mathbb{N}$ large enough to ensure that $n Q_{k} \ne \emptyset$ for all $k \in \{1, \ldots, K\}$.

  Then, by using both finite additivity and the triangle inequality, we get
  \begin{align}
    &\biggl|
      \frac{1}{n^{d}}\mspace{-6mu}
      \sum_{z \in \mathcal{C}_{\infty}(\omega)} \mspace{-6mu} f(z/n)
      \Mean_{0}^{\tau_{z} \omega}\bigl[ F(z/n + X^{n}) \bigr]
      -
      \prob\bigl[0 \in \mathcal{C}_{\infty}\bigr]
      \int_{\mathbb{R}^{d}} f(x) \Mean_{0}^{\Sigma}\bigl[ F(x + W) \bigr]\, \mathrm{d}x
    \biggr|
    \nonumber\\
    &\mspace{36mu}\leq\;
    \sum_{k=1}^{K}\;
    \biggl|
      \frac{1}{n^{d}}\mspace{-6mu}
      \sum_{z \in n Q_{k}} \mspace{-6mu} f(z/n)\,
      \mathbbm{1}_{\{z \in \mathcal{C}_{\infty}(\omega)\}}
      -
      \biggl( \int_{Q_{k}}\! f(x)\, \mathrm{d}x \biggr)
      \prob\bigl[0 \in \mathcal{C}_{\infty}\bigr]
    \biggr|
    \Mean_{0}^{\Sigma}\bigl[ |F(x_{k} + W)| \bigr]
    \nonumber\\
    &\mspace{66mu}+\;
    \sum_{k=1}^{K}
    \frac{1}{n^{d}}\mspace{-6mu} \sum_{z \in n Q_{k}}
    \mspace{-6mu} |f(z/n)|\, \mathbbm{1}_{\{z \in \mathcal{C}_{\infty}(\omega)\}}\,
    \Bigl|
      \Mean_{0}^{\tau_{z} \omega}\bigl[ F(x_{k} + X^{n})\bigr]
      -
      \Mean_{0}^{\Sigma}\bigl[ F(x_{k} + W) \bigr]
    \Bigr|%\, \indicator_{\{z \in \mathcal{C}_{\infty}(\omega)\}}
    \nonumber\\
    &\mspace{66mu}+\;
    \sum_{k=1}^{K} \frac{1}{n^{d}}\mspace{-6mu}
    \sum_{z \in n Q_{k}} \mspace{-6mu} |f(z/n)|\,
    \mathbbm{1}_{\{z \in \mathcal{C}_{\infty}(\omega)\}}\,
    \Mean_{0}^{\tau_{z} \omega}\bigl[
      \bigl| F(z/n + X^{n}) - F(x_{k} + X^{n})\bigr|
    \bigr]
    \nonumber\\
    &\mspace{66mu}+\;
    \prob\bigl[0 \in \mathcal{C}_{\infty}\bigr]\,
    \sum_{k=1}^{K}
    \int_{Q_{k}}
      |f(x)|\,
      \Mean_{0}^{\Sigma}\bigl[
        \bigl| F(x + W) - F(x_{k} + W)\bigr|
      \bigr]\,
    \mathrm{d}x
    \nonumber\\
    &\mspace{36mu}\rdef\;
    I_{1}(n, \varepsilon, \omega) + I_{2}(n, \varepsilon, \omega)
    + I_{3}(n, \varepsilon, \omega) + I_{4}(\varepsilon).
  \end{align}
  Now, by applying the version of the ergodic theorem in \cite[Theorem~3]{BD03}, we first obtain that $\limsup_{n \to \infty} I_{1}(n, \varepsilon, \omega) = 0$ for $\prob_{0}$-a.e.\ $\omega$. Further, for any $k \in \{1, \ldots, K\}$ and any realization of $X^{n}$ and $W$, we get from the definition of $d_{[0, T]}$ as given in \eqref{eq:def:d_[0,T]} that
  \begin{align*}
    d_{[0, T]}\bigl(z/n + X^{n}, x_{k} + X^{n}\bigr)
    &\;\leq\;
    |z/n - x_{k}|_{2}
    \;<\;
    \delta_{\varepsilon},
    &&\forall\, z \in n Q_{k},
    \\[0.5ex]
    d_{[0, T]}\bigl(x + W, x_{k} + W\bigr)
    &\;\leq\;
    |x - x_{k}|_{2}
    \;<\;
    \delta_{\varepsilon},
    &&\forall\, x \in Q_{k}.
  \end{align*}
  Since $F$ is assumed to be uniformly continuous, it follows that
  \begin{align}\label{eq:F:uniform:continuity:local}
    \bigl| F(z/n + X^{n}) - F(x_{k} + X^{n}) \bigr| \;<\; \varepsilon
    \quad \text{and} \quad
    \bigl| F(x + W) - F(x_{k} + W)\bigr| \;<\; \varepsilon
  \end{align}
  for any $z/n, x \in Q_{k}$. Hence, a further application of \cite[Theorem~3]{BD03} yields that, $\mathbb{P}_{0}$-a.s., $\limsup_{n \to \infty} I_{3}(n, \varepsilon, \omega) \leq \varepsilon \mathbb{P}\bigl[0 \in \mathcal{C}_{\infty}\bigr] \int_{\mathbb{R}^{d}} |f(x)|\, \mathrm{d}x$. Likewise, we obtain that $I_{4}(\varepsilon) \leq \varepsilon \prob_{0}\bigl[0 \in \mathcal{C}_{\infty}\bigr] \int_{\mathbb{R}^{d}} |f(x)|\, \mathrm{d}x $. These two estimates clearly imply that, $\prob_{0}$-a.s.,
  \begin{align*}
    \limsup_{\varepsilon \downarrow 0} \limsup_{n \to \infty}
    I_{3}(n, \varepsilon, \omega)
    \;=\;
    0
    \qquad \text{and} \qquad
    \limsup_{\varepsilon \downarrow 0} I_{4}(\varepsilon)
    \;=\;
    0.
  \end{align*}

  It remains to consider the term $I_{2}$. From Theorem~\ref{thm:QFCLT:RG} we know that, for $\prob_{0}$-a.e.~$\omega$, $\Mean_{0}^{\omega}\bigl[ F(x_{k} + X^{n}) \bigr] \to \Mean_{0}^{\Sigma}\bigl[ F(x_{k} + W) \bigr]$ as $n \to \infty$ for any $k \in \{1, \ldots, K\}$. By Egorov's theorem, see e.g.\ \cite[Section V.14]{Do94}, for any $\eta > 0$ there exists a measurable $\Omega_{\eta} \subset \Omega$ with $\prob_{0}\bigl[\Omega_{\eta}\bigr] \geq 1 - \eta$ such that
  \begin{align*}
    \sup_{\omega \in \Omega_{\eta}}
    \Bigl|
      \Mean_{0}^{\omega}\bigl[ F(x_{k} + X^{n}) \bigr]
      - \Mean_{x}^{\Sigma}\bigl[ F(x_{k} + W) \bigr]
    \Bigr|
    \underset{n \to \infty}{\;\longrightarrow\;}
    0,
    \qquad \forall\, k \in \{1, \ldots, K\}.
  \end{align*}
  Since, by \cite[Theorem~3]{BD03},
  \begin{align*}
    &\lim_{n \to \infty}
    \sum_{k=1}^{K} \frac{1}{n^{d}}
    \mspace{-6mu}\sum_{z \in n Q_{k}}\mspace{-6mu} |f(z/n)|\,
    \indicator_{\{z \in \mathcal{C}_{\infty}(\omega)\}}\,
    \indicator_{\{\tau_{z} \omega \in \Omega_{\eta}\}}
    \\[.5ex]
    &\mspace{72mu}=\;
    \int_{\mathbb{R}^{d}} |f(x)|\, \mathrm{d}x\,
    \prob\bigl[ 0 \in \mathcal{C}_{\infty} \bigr]
    \prob_{0}\bigl[ \Omega_{\eta} \bigr]
    \;<\; \infty,
  \end{align*}
  and
  \begin{align*}
    &\lim_{n \to \infty}
    \sum_{k=1}^{K} \frac{1}{n^{d}}
    \mspace{-6mu}\sum_{z \in n Q_{k}}\mspace{-6mu} |f(z/n)|\,
    \indicator_{\{z \in \mathcal{C}_{\infty}(\omega)\}}\,
    \indicator_{\{\tau_{z} \omega \in \Omega \setminus \Omega_{\eta}\}}
    \\[.5ex]
    &\mspace{72mu}=\;
    \int_{\mathbb{R}^{d}} |f(x)|\, \mathrm{d}x\,
    \prob\bigl[ 0 \in \mathcal{C}_{\infty} \bigr]
    \prob_{0}\bigl[ \Omega \setminus \Omega_{\eta} \bigr]
    \;\leq\;
    \eta\, \int_{\mathbb{R}^{d}} |f(x)|\, \mathrm{d}x,
  \end{align*}
  for $\prob_{0}$-a.e.\ $\omega$, and $F$ is assumed to be bounded by some $C < \infty$, say, we obtain that, for any $\varepsilon > 0$,
  \begin{align*}
    \limsup_{n \to \infty}
    I_{2}(n, \varepsilon, \omega)
    % &\;\leq\;
    % \limsup_{n \to \infty} 
    % \sum_{k=1}^{K} \frac{2C}{n^{d}}
    % \mspace{-6mu}\sum_{z \in n Q_{k}}\mspace{-6mu} f(z/n)\,
    % \indicator_{\{z \in \mathcal{C}_{\infty}(\omega)\}}\,
    % \indicator_{\{\tau_{z} \omega \in \Omega \setminus \Omega_{\eta}\}}
    % \\[.5ex]
    % &
    \;\leq\;
   2C \eta\, \int_{\mathbb{R}^{d}} f(x)\, \mathrm{d}x
    \;\underset{\eta \downarrow 0}{\rightarrow}\;
    0,
    \qquad \prob_{0}\text{-a.s.}
  \end{align*}
  By combining the estimates above, the assertion \eqref{eq:QFCLT:RG:initial:shifted} follows for any bounded, uniformly continuous functions $F\colon D([0, T], \mathbb{R}^{d}) \to \mathbb{R}$.
  \smallskip
  
  \textsc{Step 2.} It remains to remove the additional assumption that $F$ is uniformly continuous. For this purpose, define a finite measure, $\mathfrak{m}$, and, for every $\omega \in \Omega$, a sequence of finite measures, $(\mathfrak{m}_{n}^{\omega})_{n \in \mathbb{N}}$ on $D([0, T], \mathbb{R}^{d})$ by
  \begin{align*}
    \mathfrak{m}
    \;\ldef\;
    \mathbb{P}\bigl[0 \in \mathcal{C}_{\infty}\bigr]\,
    \int_{\mathbb{R}^{d}} f(x)\, \Prob_{x}^{\Sigma} \circ\, W^{-1}
    \quad \text{and} \quad
    \mathfrak{m}_{n}^{\omega}
    \;\ldef\;
    \frac{1}{n^{d}} \mspace{-6mu} \sum_{z \in \mathcal{C}_{\infty}(\omega)}\mspace{-6mu}
    f(z/n) \Prob_{z}^{\omega} \circ\, (X^{n})^{-1}.
  \end{align*}
  Then, by \emph{Step 1}, it holds that
  \begin{align*}
    \lim_{n \to \infty} \int F\, \mathrm{d}\mathfrak{m}_{n}^{\omega}
    \;=\;
    \int F\, \mathrm{d}\mathfrak{m}
  \end{align*}
  for any bounded, uniformly continuous $F\colon D([0, T], \mathbb{R}^{d})$ and $\prob_{0}$-a.s.\ $\omega$. Thus, \eqref{eq:QFCLT:RG:init} is an immediate consequence of the Portemanteau theorem, cf.~\cite[Satz~VII.4.10]{El18} or \cite[Theorem~2.1]{Bi99}.
\end{proof}

\subsection{An extended quenched invariance principle}
\begin{theorem}
  \label{thm:FCLT:Poisson}
  Let $d \geq 2$ and $D \subset \mathbb{R}^{d}$ be a bounded, strongly regular domain. Suppose that there exist $\theta \in (0, 1)$ and $p, q \in [1, \infty]$ such that Assumptions~\ref{ass:law}, \ref{ass:cluster} and~\ref{ass:pq} hold. Then, $\prob_{0}$-a.s., for every bounded continuous function $f\colon \mathbb{R}^{d} \to \mathbb{R}$ with bounded support and every $g \in C(\mathbb{R}^{d}, \mathbb{R})$,
  \begin{align}
    \label{eq:FCLT:Poisson}
    &\lim_{n \to \infty}
    \frac{1}{n^{d}} \mspace{-6mu}\sum_{z \in \mathcal{C}_{\infty}(\omega)}\mspace{-6mu}
    f(z/n)\,
    \Mean_{z}^{\omega}\biggl[
      \int_{0}^{\tau_{D}(X^{n})} g\bigl(X_{t}^{n}\bigr)\, \mathrm{d}t
    \biggr] 
    \nonumber\\
    &\mspace{72mu}=\;
    \prob\bigl[ 0 \in \mathcal{C}_{\infty} \bigr]\,
    \int_{D}
      f(x)\,
      \Mean_{x}^{\Sigma}\biggl[
        \int_{0}^{\tau_{D}(W)} g\bigl(W_t\bigr)\, \mathrm{d}t
      \biggr]
    \mathrm{d}x.
  \end{align}
\end{theorem}
Before we prove Theorem~\ref{thm:FCLT:Poisson} we first provide two preparatory lemmas.
\begin{lemma}
  \label{lem:u:lsc}
  Let $d \geq 1$ and $D \subset \mathbb{R}^{d}$ be open. Then, the map $\tau_{\bar{D}}\colon D([0, \infty), \mathbb{R}^{d}) \to [0, \infty]$ is upper semi-continuous, and the map $\tau_{D}\colon C([0, \infty), \mathbb{R}^{d}) \to [0, \infty]$ is lower semi-continuous. In particular, $\tau_{D}$ is continuous at any $w \in C([0, \infty), \mathbb{R}^{d})$ for which $\tau_{D}(w) = \tau_{\bar{D}}(w) < \infty$.
\end{lemma}
\begin{proof}
  We first show that the map $\tau_{\bar{D}}$ on $D([0, \infty), \mathbb{R}^{d})$ is upper semi-continuous. For any given $w \in D([0, \infty), \mathbb{R}^{d})$, let $(w^{n})_{n \in \mathbb{N}}$ be a sequence in $D([0, \infty), \mathbb{R}^{d})$ that converges to $w$. Since the upper semi-continuity of $\tau_{\bar{D}}$ at $w$ is trivially satisfied if $\tau_{\bar{D}}(w) = \infty$, we may assume that $t_{0} \ldef \tau_{\bar{D}}(w) < \infty$. Then, the right-continuity of $w$ implies that $w_{t_{0} + \varepsilon} \in \bar{D}^{\mathrm{c}}$ for any sufficiently small $\varepsilon > 0$. Since $\bar{D}^{\mathrm{c}}$ is open, there exists $\delta \in (0, \varepsilon]$ such that the open ball, $B_{\delta}(w_{t_{0} + \varepsilon})$, with centre $w_{t_{0} + \varepsilon}$ and radius $\delta$ is contained in $\bar{D}^{\mathrm{c}}$. In view of the characterisation of convergence in $D([0, \infty), \mathbb{R}^{d})$, there exists $(\lambda^{n})_{n \in \mathbb{N}} \subset \Lambda_{\infty}$ and $N_{\delta} \in \mathbb{N}$ such that, for any $n \geq N_{\delta}$,
  \begin{align*}
    \bigl| \lambda_{t_{0} + \varepsilon}^{n} - (t_{0} + \varepsilon) \bigr|
    \;\leq\;
    \frac{\delta}{2}
    \qquad \text{and} \qquad
    \bigl| w^{n} \circ \lambda_{t_{0} + \varepsilon}^{n} - w_{t_{0} + \varepsilon} \bigr|_{2}
    \;\leq\;
    \frac{\delta}{2}.
  \end{align*}
  As a consequence, $w^{n} \circ \lambda_{t_{0} + \varepsilon}^{n} \in B_{\delta}(w_{t_{0} + \varepsilon}) \subset \bar{D}^{\mathrm{c}}$ which implies that
  \begin{align*}
    \tau_{\bar{D}}(w^{n})
    \;\leq\;
    \lambda_{t_{0} + \varepsilon}^{n}
    \;\leq\;
    t_{0} + \varepsilon + \delta/2
    \;\leq\;
    \tau_{\bar{D}}(w) + 2\varepsilon
  \end{align*}
  for any $n \geq N_{\delta}$. Since $\varepsilon > 0$ was arbitrarily chosen, the upper semi-continuity of $\tau_{\bar{D}}$ at $w$ follows.
  
  In order to prove the lower semi-continuity of the map $\tau_{D}$ on $C([0, \infty), \mathbb{R}^{d})$ it suffices to show that, for any $T > 0$, the sub-level sets
  \begin{align*}
    A_{T}
    \;\ldef\;
    \bigl\{ w \in C([0, \infty), \mathbb{R}^{d}) : \tau_{D}(w) \leq T \bigr\}
  \end{align*}
  are closed. For this purpose, let $(w^{n})_{n \in \mathbb{N}}$ be a sequence in $A_{T}$ that converges to some $w \in C([0, \infty), \mathbb{R}^{d})$. Set $t_{n} \ldef \tau_{D}(w^{n})$ for any $n \in \mathbb{N}$. Since $D^{\mathrm{c}}$ is closed, $\tau_{D}(w^{n}) \leq T$ and $w^{n}$ is right-continuous,  we have that $w_{t_{n}}^{n} \in D^{\mathrm{c}}$ for any $n \in \mathbb{N}$. Further, as $(t_{n})_{n \in \mathbb{N}} \subset [0, T]$, there exists a sub-sequence $(t_{n_{k}})_{k \in \mathbb{N}}$ that converges to some $t \in [0, T]$. By continuity of $w$, it follows that
  \begin{align*}
    \limsup_{k \to \infty} \bigl| w_{t} - w_{t_{n_{k}}}^{n_{k}} \bigr|_{2}
    \;\leq\;
    \limsup_{k \to \infty} \Bigl( \bigl| w_{t} - w_{t_{n_{k}}} \bigr|_{2} + \bigl\| w - w^{n_{k}} \bigr\|_{[0, \infty)} \Bigr)
    \;=\;
    0.
  \end{align*}
  Thus, $w_{t} = \lim_{k \to \infty} w_{t_{n_{k}}}^{n_{k}} \in D^{\mathrm{c}}$ which implies that $\tau_{D}(w) \leq t$ and, hence, $w \in A_{T}$.
\end{proof}
\begin{lemma}
  \label{lem:strong:regular:consequence}
  Let $d \geq 2$ and $D \subset \mathbb{R}^{d}$ be a bounded, strongly regular domain. Further, let $((W_{t})_{t \geq 0}, (\Prob_{x}^{\Sigma})_{x \in \mathbb{R}^{d}})$ be a Brownian motion on $\mathbb{R}^{d}$ with non-degenerate covariance matrix $\Sigma^{2} \ldef \Sigma \cdot \Sigma^{\mathrm{T}}$. Then,
  \begin{align}
    \label{eq:strong:regular:consequence}
    \Prob_{x}^{\Sigma}\bigl[ \tau_{D}(W) = \tau_{\bar{D}}(W) \bigr] \;=\; 1,
    \qquad \forall\, x \in \mathbb{R}^{d}.
  \end{align}
\end{lemma}
\begin{proof}
  Clearly, the assertion is immediate for any $x \in \bar{D}^{\mathrm{c}}$. Further, observe that the strong regularity of $D$ implies that $\Prob_{x}^{\Sigma}[ \tau_{D}(W) = \tau_{\bar{D}}(W) = 0] = 1$ for any $x \in \partial D$. Thus, it remains to prove \eqref{eq:strong:regular:consequence} for any $x \in D$. Since $D$ is assumed to be bounded and $\Sigma^{2}$ is non-degenerate, we know from \cite[Lemma~5.7.4]{KS91} that $\Prob_{x}^{\Sigma}[\tau_{D}(W) < \infty] = 1$. Thus, by using the strong Markov property and the strong regularity of the subset $D$, we find that
  \begin{align*}
    \Prob_{x}^{\Sigma}\bigl[ \tau_{D}(W) = \tau_{\bar{D}}(W) \bigr]
    &\;=\;
    % \Mean_{x}^{\Sigma}\Bigl[
    %   \Mean_{x}^{\Sigma}\Bigl[
    %     \mathbbm{1}_{\tau_{D}(W) = \tau_{\bar{D}}(W)} \mathbbm{1}_{\tau_{D}(W) < \infty}
    %     \,\Big|\, \mathcal{F}_{\tau_{D}}^{W}
    %   \Bigr]
    % \Bigr]
    % \\[.5ex]
    % &\;=\;
    \Mean_{x}^{\Sigma}\Bigl[
      \indicator_{\{\tau_{D}(W) < \infty\}}
      \Prob_{W_{\tau_{D}(W)}}^{\Sigma}\bigl[ \tau_{\bar{D}}(W) = 0 \bigr]
    \Bigr]
    \;=\;
    1,
  \end{align*}
  which concludes the proof.
\end{proof}
\begin{proof}[Proof of Theorem~\ref{thm:FCLT:Poisson}]
  Let $D \subset \mathbb{R}^{d}$ be a strongly regular domain, and let $f\colon \mathbb{R}^{d} \to \mathbb{R}$ be a bounded continuous function with bounded support, and $g \in C(\mathbb{R}^{d}, \mathbb{R})$. To simplify notation we further define measurable functions $F$ and $(F_{k})_{k \in \mathbb{N}}$ on the Skorokhod space $D([0, \infty), \mathbb{R}^{d})$ by
  \begin{align*}
    F(w) \;\ldef\; \int_{0}^{\tau_{D}(w)} g(w_{t})\, \mathrm{d}t
    \qquad \text{and} \qquad
    F_{k}(w) \;\ldef\; \int_{0}^{\tau_{D}(w) \wedge k} g(w_{t})\, \mathrm{d}t.
  \end{align*}
  Then,
  \begin{align*}
    &\Biggl|
      \frac{1}{n^{d}} \mspace{-6mu}\sum_{z \in \mathcal{C}_{\infty}(\omega)}\mspace{-6mu}
      f(z/n)\,
      \Mean_{z}^{\omega}\bigl[ F(X^{n}) \bigr]
      -
      \prob\bigl[ 0 \in \mathcal{C}_{\infty} \bigr]\,
      \int_{D}
        f(x)\,
        \Mean_{x}^{\Sigma}\bigl[ F(W) \bigr]
      \mathrm{d}x
    \Biggr|
    \\[.5ex]
    &\mspace{32mu}\leq\;
    \Biggl|
      \frac{1}{n^{d}} \mspace{-6mu}\sum_{z \in \mathcal{C}_{\infty}(\omega)}\mspace{-6mu}
      f(z/n)\,
      \Mean_{z}^{\omega}\bigl[ F_{k}(X^{n}) \bigr]
      -
      \prob\bigl[ 0 \in \mathcal{C}_{\infty} \bigr]\,
      \int_{D}
        f(x)\,
        \Mean_{x}^{\Sigma}\bigl[ F_{k}(W) \bigr]\,
      \mathrm{d}x
    \Biggr|
    \\%[.5ex]
    &\mspace{72mu}+\;
    \frac{1}{n^{d}} \mspace{-6mu}\sum_{z \in \mathcal{C}_{\infty}(\omega)}\mspace{-6mu}
    |f(z/n)|\,
    \Mean_{z}^{\omega}\biggl[
      \int_{\tau_{D}(X^{n}) \wedge k}^{\tau_{D}(X^{n})}
        \bigl| g\bigl( X_{t}^{n} \bigr) \bigr|\,
      \mathrm{d}t
    \biggr]
    \\%[.5ex]
    &\mspace{72mu}+\;
    \int_{D}
      |f(x)|\,
      \Mean_{x}^{\Sigma}\biggl[
        \int_{\tau_{D}(W) \wedge k}^{\tau_{D}(W)}
          \bigl| g\bigl( W_{t} \bigr) \bigr|\,
        \mathrm{d}t
      \biggr]\,
    \mathrm{d}x
    \\[.5ex]
    &\mspace{32mu}\rdef\;
    I_{1}(n, k, \omega) + I_{2}(n, k, \omega) + I_{3}(k).
  \end{align*}
  We start with proving that, for any $k \in \mathbb{N}$ and $\prob_{0}$-a.e.\ $\omega$, the term $I_{1}(n, k, \omega)$ vanishes in the limit as $n \to \infty$. Clearly, due to the truncation of the integral in the definition of $F_{k}$ at level $k \in \mathbb{N}$ and the fact that $g$ is continuous and $D$ is bounded, it follows that $|F_{k}(w)| \leq k \max_{x \in \bar{D}} |g(x)| < \infty$ for any $w \in D([0, \infty), \mathbb{R}^{d})$. Moreover, in view of Lemma~\ref{lem:u:lsc} and the continuity of $g$, it follows that the map $F_{k}$ is continuous for any $w \in C([0, \infty), \mathbb{R}^{d})$ such that $\tau_{D}(w) = \tau_{\bar{D}}(w)$. By Lemma~\ref{lem:strong:regular:consequence} we therefore get that
  \begin{align}
    \label{eq:continuity:F_k}
    \Prob_{x}^{\Sigma}\bigl[
      \bigl\{F_{k} \text{ is not continuous at } W\bigr\}
    \bigr]
    \;=\;
    0,
    \qquad \forall\, x \in D,
  \end{align}
  as a consequence of the Portemanteau theorem, cf.~\cite[Theorem~2.7]{Bi99}. Note that $\Mean_{x}^{\Sigma}\bigl[F(W)\bigr] = 0$ for all $x \in D^{\mathrm{c}}$. Then, the weak convergence established in Theorem~\ref{thm:QFCLT:RG:initial} together with \eqref{eq:continuity:F_k} implies that, $\prob_{0}$-a.s.,
  \begin{align*}
    \lim_{n \to \infty} I_{1}(n, k, \omega) \;=\; 0,
    \qquad \forall\, k \in \mathbb{N}.
  \end{align*}
  In particular, the null set does not depend on $f$ and $g$.
  
  Next, let us address the term $I_{2}$. Since $g$ is continuous and $D$ is bounded, we get that, for any $k, n \in \mathbb{N}$ and $\prob_{0}$-a.e.\ $\omega$,
  \begin{align*}
    I_{2}(n, k, \omega)
    \;\leq\;
    \biggl(\max_{x \in \bar{D}} |g(x)|\,\biggr)\,
    \frac{1}{n^{d}} \mspace{-6mu}\sum_{z \in \mathcal{C}_{\infty}(\omega)}\mspace{-6mu}
    |f(z/n)|\,
    \Mean_{z}^{\omega}\bigl[
      \bigl( \tau_{D}(X^{n})-k \bigr)\, \indicator_{\{\tau_{D}(X^{n}) > k\}}
    \bigr].
  \end{align*}
  By Assumption~\ref{ass:cluster} (cf.\ Remark~\ref{rem:ass_cluster}-(ii)), there exists $r_{D} \in (0, \infty)$ such that, for any $\omega \in \Omega_{\mathrm{reg}} \cap \Omega_{0}^{*}$ and $n \geq \max\{N_{0}(\omega)/r_{D}, 1\}$, we have that $nD \cap \mathcal{C}_{\infty}(\omega) \subset B^{\omega}(0, C_{\mathrm{d}} r_{D} n)$. By using the Markov property, we get that, for any $z \in nD \cap \mathcal{C}_{\infty}(\omega)$,
  \begin{align*}
    &\Mean_{z}^{\omega}\bigl[
      \bigl( \tau_{D}(X^{n})-k \bigr)\, \indicator_{\{\tau_{D}(X^{n}) > k\}}
    \bigr]
    \\
    &\mspace{36mu}=\;
    \Mean_{z}^{\omega}\Bigl[
      \indicator_{\{\tau_{D}(X^{n}) > k\}}\,
      \Mean_{X_{k n^{2}}}^{\omega}\bigl[\tau_{D}(X^{n})\bigr]
    \Bigr]
    \;\leq\;
    \frac{1}{k}
    \biggl(
      \max_{z \in n D \cap \mathcal{C}_{\infty}(\omega)}
      \Mean_{z}^{\omega}\bigl[ \tau_{D}(X^{n})\bigr]
    \biggr)^{\!\!2}.
  \end{align*}
  Further, by Corollary~\ref{cor:exit} we also know that, for any $\omega \in \Omega_{\mathrm{reg}} \cap \Omega_{0}^{*}$ and $n \in \mathbb{N}$ such that $\lfloor C_{\mathrm{d}} r_{D} n \rfloor \geq N_{0}(\omega)$,
  \begin{align}\label{eq:mean:exit:D}
    &\max_{z \in n D \cap \mathcal{C}_{\infty}(\omega) } \!
    \Mean_{z}^{\omega}\bigl[ \tau_{D}(X^{n}) \bigr]
    \nonumber\\[.5ex]
    &\mspace{36mu}\leq\;  
    n^{-2}\mspace{-6mu} \max_{z \in B^{\omega}(0, C_{\mathrm{d}} r_{D} n)}
    \Mean_{z}^{\omega}\bigl[ \tau_{B^{\omega}(0, C_{\mathrm{d}} r_{D} n)}(X) \bigr]     
    \;\leq\;
    C_{4} (C_{\mathrm{V}} r_{D})^{2}\, \Norm{\nu^{\omega}}{q, B^{\omega}(0, C_{\mathrm{d}} r_{D} n)}.
  \end{align}
  Since the volume regularity of $B^{\omega}(0, C_{\mathrm{d}} r_{D} n)$ together with an application of the ergodic theorem yields that
  \begin{align*}
    \limsup_{n \to \infty} \Norm{\nu^{\omega}}{q, B^{\omega}(0, C_{\mathrm{d}} r_{D} n)}
    \;\leq\;
    \frac{(2 C_{\mathrm{d}} r_{D})^{d}}{C_{\mathrm{V}}}\, \mean_{0}\bigl[\nu(0)^{q}\bigr]^{1/q},
    \qquad \prob\text{-a.s.}
  \end{align*}
  we obtain that there exists $c \in (0, \infty)$ such that, $\prob$-a.s, for any $k \in \mathbb{N}$,
  \begin{align*}
    \limsup_{n \to \infty} I_{2}(n, k, \omega)
    \;\leq\;
    \frac{c}{k}\,
    \biggl(\max_{x \in \bar{D}} |g(x)|\,\biggr)\, \int_{D} |f(x)|\, \mathrm{d}x\,
    \mean_{0}\bigl[ \nu(0)^{q} \bigr]^{2/q}.
  \end{align*}
  Thus, $\limsup_{k \to \infty} \limsup_{n \to \infty} I_{2}(n, k, \omega) = 0$ for $\prob$-a.e.\ $\omega$.

  Finally, we address the term $I_{3}$. By \cite[Lemma~5.7.4]{KS91} we have $\Mean_{x}^{\Sigma}\bigl[ \tau_{D}(W) \bigr] < \infty$ so that, for any $x \in D$,
  \begin{align*}
    \Mean_{x}^{\Sigma}\biggl[
      \int_{\tau_{D}(W) \wedge k}^{\tau_{D}(W)}
        \bigl| g\bigl( W_{t} \bigr) \bigr|\,
      \mathrm{d}t
    \biggr]
    \;\leq\;
    \Bigl(\max_{x \in \bar{D}} |g(x)| \Bigr)\,
    \Mean_{x}^{\Sigma}\bigl[ \tau_{D}(W) \indicator_{\{\tau_{D}(W) > k\}} \bigr]
    \;\searrow\;
    0
  \end{align*}
  as $k \to \infty$. Thus, an application of the monotone convergence theorem yields
  \begin{align*}
    \limsup_{k \to \infty} I_{3}(k)
    \;\leq\;
    \biggl(\max_{x \in \bar{D}} |g(x)| \biggr)
    \limsup_{k \to \infty}
    \int_{D}
      |f(x)|\,
      \Mean_{x}^{\Sigma}\bigl[
        \tau_{D}(W) \indicator_{\{\tau_{D}(W) > k\}}
      \bigr]\,
    \mathrm{d}x
    \;=\;
    0,
  \end{align*}
  which completes the proof.
\end{proof}

\section{Scaling limit of the inhomogeneous discrete Gaussian free field}
\label{sec:scaling_limitCGF}
In this section we will show Theorem~\ref{thm:scalingCGF_rg}. First, we apply the extended version of the central limit theorem in Theorem~\ref{thm:FCLT:Poisson} to show the convergence of the rescaled field integrated against suitable test functions, see Theorem~\ref{thm:scalingCGF_rg:fdd}. 
%Here we follow the main strategy of the corresponding proof for the homogeneous DGFF in \cite{Bi20}. 
Afterwards, in order to conclude, we establish tightness in Proposition~\ref{prop:DGFF:tightness}.
\begin{theorem}\label{thm:scalingCGF_rg:fdd}
  Under the assumptions of Theorem~\ref{thm:scalingCGF_rg},
  % For any bounded, measurable function $f\colon D \to \mathbb{R}$, let \textcolor{blue}{Overload of notation: $\varphi^{nD} \equiv (\varphi_{x}^{nD} : x \in \mathbb{Z}^{d})$, see Definition~\ref{def:inhomogeneous:DGFF} and the rescale linear function. I suggest to rename the latter as $\Phi_{n}^{D}$. In particular, if we formulate the tighness and the convergence of the DGFF.}
  % %
  % \begin{align}
  %   \label{eq:def:functional:DGFF}
  %   \varphi^{n D}(f)
  %   \;\ldef\;
  %   n^{d/2-1}
  %   \int_{D}
  %     f(x)\, \varphi^{n D}_{\lfloor nx \rfloor}\,
  %     \indicator_{\{ \lfloor nx \rfloor \in \mathcal{C}_{\infty}(\omega) \}}\,
  %     \mathrm{d}x,
  %     \qquad n \in \mathbb{N}.
  % \end{align}
  % %
  for $\prob_{0}$-a.e.\ $\omega$, under $\probPhi^{\omega}$,
  \begin{align*}
    \Phi_{n}^{D}(f)
    \underset{n \to \infty}{\overset{\text{law}}{\;\longrightarrow\;}}
    \mathcal{N}(0, \sigma^{2}_{\Sigma}(f)),
  \end{align*}
  for any bounded, measurable function $f\colon D \to \mathbb{R}$, where $\sigma_{\Sigma}^{2}(f)$ is defined as in \eqref{def:var_cont_CGF}.
  % and
  % % 
  % \begin{align}\label{def:var_cont_CGF}
  %   \sigma_{\Sigma}^{2}(f)
  %   \;\ldef\;
  %   \prob\bigl[ 0 \in \mathcal{C}_{\infty} \bigr]\,
  %   \int_{D \times D} f(x)\, f(y)\, g_{D}^{\Sigma}(x, y)\, \mathrm{d}x\, \mathrm{d}y.
  % \end{align}
\end{theorem}
\begin{proof}
  Let $D \subset \mathbb{R}^{d}$ be a bounded, strongly regular domain and $f\colon D \to \mathbb{R}$ be a bounded, measurable function, and recall that the linear functional $\Phi_{n}^{D}(f)$ is defined as in \eqref{eq:def:functional:DGFF}. Since the random variable $\Phi_{n}^{D}(f)$ is Gaussian with mean zero and variance
  \begin{align*}
    \varGFF^{\omega}\bigl[\Phi_{n}^{D}(f)\bigr]
    \;=\;
    n^{d-2} \int_{D \times D} f(y)\, f(y')\, g^{\omega}_{n D}([n y], [n y'])\, \mathrm{d}y\, \mathrm{d}y',
  \end{align*}
  the assertion follows once we have shown that, $\prob_{0}$-a.s.,
  \begin{align}
    \label{eq:var:convergence:GFF}
    \lim_{n \to \infty} \varGFF^{\omega}\bigl[\Phi_{n}^{D}(f)\bigr]
    \;=\;
    \sigma_{\Sigma}^{2}(f).
  \end{align}
  In order to show \eqref{eq:var:convergence:GFF}, we extend $f$ to a function from $\mathbb{R}^{d}$ to $\mathbb{R}$ that vanishes outside of $D$, also denoted by $f$ with a slight abuse of notation, and we then define
  \begin{align*}
    f_{n}(x)
    \;\ldef\;
    n^{d}\, \int_{x + (-1/(2n), 1/(2n)]^{d}} f(y)\, \mathrm{d}y,
    \qquad x \in \mathbb{R}^{d}, \quad n \in \mathbb{N}.
  \end{align*}
  Then,
  \begin{align*}
    \varGFF^{\omega}\bigl[\Phi_{n}^{D}(f)\bigr]
    &\;=\;
    \frac{1}{n^{d+2}} \mspace{-6mu}\sum_{z, z' \in \mathcal{C}_{\infty}(\omega)} \mspace{-6mu}
    f_{n}(z/n)\, g_{nD}^{\omega}(z, z')\, f_{n}(z'/n).
  \end{align*}
  Since $f$ is locally integrable an application of Lebesgue–Besicovitch's Differentiation Theorem, cf.\ \cite[Theorem~1.32]{EG15}, yields that $f_{n}(x) \to f(x)$ as $n\to \infty$ for Lebesgue-almost every $x \in \mathbb{R}^{d}$. Moreover, by applying both Lusin's and Egorov's theorem, cf.\ \cite[Theorem~1.13, 1.14 and 1.16]{EG15} we have that, for any $\varepsilon > 0$, there exist a compact set $K_{\varepsilon} \subset D$ and a continuous function $f_{\varepsilon}\colon \mathbb{R}^{d} \to \mathbb{R}$ such that
  \begin{align}
    \label{eq:Egorov:Lusin}
    \lim_{n \to \infty} \max_{x \in K_{\varepsilon}} | f_{n}(x) - f(x) | \;=\; 0,
    \quad
    f \;=\; f_{\varepsilon} \quad \text{on } K_{\varepsilon},
    \quad \text{and} \quad
    |D \setminus K_{\varepsilon}| \;<\; \varepsilon,
  \end{align}
  where, for any measurable $A \subset \mathbb{R}^{d}$, we write $|A|$ to denote the Lebesgue measure of the subset $A$. In particular, we may choose $f_{\varepsilon}$ such that $\| f_{\varepsilon} \|_{\infty} \leq \| f \|_{\infty}$. To simplify notation further, we define
  \begin{align*}
    F_{\varepsilon}\colon D([0, \infty), \mathbb{R}^{d}) \to \mathbb{R},
    \quad
    w
    \;\longmapsto\;
    F_{\varepsilon}(w)
    \;\ldef\; \int_{0}^{\tau_{D}(w)} f_{\varepsilon}(w_{t})\, \mathrm{d}t.
  \end{align*}
  Recall that $g^{\omega}_{n D}(z, z') = 0$ whenever $z \in (\mathcal{C}_{\infty}(\omega) \cap nD)^{\mathrm{c}}$ or $z' \in (\mathcal{C}_{\infty}(\omega) \cap nD)^{\mathrm{c}}$. Thus, after rescaling the time by $n^{2}$, we obtain that
  \begin{align*}
    &\frac{1}{n^{d+2}} \mspace{-6mu}\sum_{z, z' \in \mathcal{C}_{\infty}(\omega)} \mspace{-6mu}
    f_{\varepsilon}(z/n)\, g_{nD}^{\omega}(z, z')\, f_{\varepsilon}(z'/n)
    % \\[.5ex]
    % &\mspace{32mu}=\;
    % \frac{1}{n^{d}} \mspace{-6mu}\sum_{z \in \mathcal{C}_{\infty}(\omega)} \mspace{-6mu}
    % f_{n}(z/n)\,
    % \Mean_{z}^{\omega}\biggl[
    %   \int_{0}^{\tau_{D}(X^{n})} f_{\varepsilon}\bigl(X_{t}^{n}\bigr)\, \mathrm{d}t
    % \biggr]
    \;=\;
    \frac{1}{n^{d}} \mspace{-6mu}\sum_{z \in \mathcal{C}_{\infty}(\omega)} \mspace{-6mu}
    f_{\varepsilon}(z/n)\,
    \Mean_{z}^{\omega}\bigl[ F_{\varepsilon}(X_{t}^{n}) \bigr].
  \end{align*}
  Since $X$ is a reversible Markov process it follows that $g_{nD}^{\omega}(z, z') = g_{nD}^{\omega}(z', z)$ for all $z, z' \in \mathbb{Z}^{d}$. By the same reasoning we also have that the Dirichlet Green kernel, $g_{D}^{\Sigma}$, is symmetric in its arguments. This yields the decomposition
  \begin{align*}
    &
    \bigl|
      \varGFF^{\omega}\bigl[\Phi_{n}^{D}(f)\bigr] - \sigma_{\Sigma}^{2}(f)
    \bigr|
    \\[.5ex]
    &\mspace{36mu}\leq\;
    \biggl|
      \frac{1}{n^{d}} \mspace{-6mu}\sum_{z \in \mathcal{C}_{\infty}(\omega)} \mspace{-6mu}
      f_{\varepsilon}(z/n)\, \Mean_{z}^{\omega}\bigl[F_{\varepsilon}(X^{n})\bigr]
      -
      \prob\bigl[0 \in \mathcal{C}_{\infty}\bigr]\,
      \int_{\mathbb{R}^{d}}
        f_{\varepsilon}(x)\, \Mean_{x}^{\Sigma}\bigl[F_{\varepsilon}(W)\bigr]\,
      \mathrm{d}x
    \biggr|
    \\
    &\mspace{72mu}+\,
    \frac{1}{n^{d+2}} \mspace{-6mu}\sum_{z \in \mathcal{C}_{\infty}(\omega)} \mspace{-6mu}
    \bigl| f_{n}(z/n) - f_{\varepsilon}(z/n) \bigr|\, g_{nD}^{\omega}(z, z')\,
    \bigl( \bigl| f_{n}(z'/n) \bigr| + \bigl| f_{\varepsilon}(z'/n) \bigr|\bigr)
    \\
    &\mspace{72mu}+\,
    \prob\bigl[0 \in \mathcal{C}_{\infty}\bigr]\,
    \int_{D \times D}
      \bigl|f(y) - f_{\varepsilon}(y)\bigr|\, g_{D}^{\Sigma}(y, y')\,
      \bigl|f(y') + f_{\varepsilon}(y')\bigr|
    \mathrm{d}y\, \mathrm{d}y'
    \\[.5ex]
    &\mspace{36mu}\rdef\;
    I_{1}(\omega, n, \varepsilon) + I_{2}(\omega, n, \varepsilon) + I_{3}(\varepsilon).
  \end{align*}
  Now, by Theorem~\ref{thm:FCLT:Poisson}, it follows that $\lim_{n \to \infty} I_{1}(\omega, n, \varepsilon) = 0$ for $\prob_{0}$-a.e.\ $\omega$ and any $\varepsilon > 0$. Let us stress that the corresponding null set is independent of $\varepsilon$.

  Next, we address the term $I_{2}$. Since $D$ is assumed to be bounded, there exists $r_{D} \in (0, \infty)$ such that, for any $\omega \in \Omega_{\mathrm{reg}} \cap \Omega_{0}^{*}$, $n D \cap \mathcal{C}_{\infty}(\omega) \in [-r_{D} n, r_{D} n]^{d}$. Thus, by Assumption~\ref{ass:cluster} (cf.\ Remark~\ref{rem:ass_cluster}-(ii)) and Corollary~\ref{cor:exit}, we find that there exists $c \in (0, \infty)$ such that, for any $\omega \in \Omega_{\mathrm{reg}} \cap \Omega_{0}^{*}$ and $n \geq \max\{N_{0}(\omega)/r_{D}, 1\}$,
  \begin{align*}
    \max_{z \in nD \cap \mathcal{C}_{\infty}(\omega)}
    \Mean_{z}^{\omega}\bigl[ \tau_{D}(X^{n}) \bigr]
    \overset{\mspace{-9mu}\eqref{eq:mean:exit:D}\mspace{-9mu}}{\;\leq\;}
    c\, \Norm{\nu^{\omega}}{q, B^{\omega}(0, C_{\mathrm{d}} r_{D} n)}.
  \end{align*}
  Since $\|f_{n}\|_{\infty} \leq \|f\|_{\infty}$, we obtain that, for any $\omega \in \Omega_{\mathrm{reg}} \cap \Omega_{0}^{*}$ and $n \geq \max\{N_{0}(\omega), 1\}$,
  \begin{align*}
    I_{2}(\omega, n, \varepsilon)
    &\;\leq\;
    2 \| f \|_{\infty} \frac{1}{n^{d}}
    \mspace{-6mu}\sum_{z \in \mathcal{C}_{\infty}(\omega)} \mspace{-6mu}
    \bigl| f_{n}(z/n) - f_{\varepsilon}(z/n) \bigr|\, \Mean_{z}^{\omega}\bigl[\tau_{D}(X^{n})\bigr]\,
    \mathbbm{1}_{\{z \in nD\}}
    \\
    &\;\leq\;
    2 c \Norm{\nu^{\omega}}{q, B^{\omega}(0, C_{\mathrm{d}} r_{D} n)} \| f \|_{\infty}\,
    \Biggl(
      \frac{1}{n^{d}}
      \mspace{-6mu}\sum_{z \in \mathcal{C}_{\infty}(\omega)} \mspace{-6mu}
      \bigl| f_{n}(z/n) - f_{\varepsilon}(z/n) \bigr|\, 
      \mathbbm{1}_{\{z/n \in D\}}
    \Biggr).
  \end{align*}
  The remaining term in brackets can be further estimated from above by
  \begin{align*}
    \biggl(
      \max_{x \in K_{\varepsilon}}\,
      \bigl| f_{n}(x) - f_{\varepsilon}(x) \bigr|
    \biggr)\, \frac{|nD|}{n^{d}}
    +
    2\|f\|_{\infty}\, \frac{1}{n^{d}}
    \sum_{z \in \mathbb{Z}^{d}} \mathbbm{1}_{\{z/n \in D \setminus K_{\varepsilon} \}}.
  \end{align*}
  Thus, in view of \eqref{eq:Egorov:Lusin}, an application of the the ergodic theorem and the volume regularity of $B^{\omega}(0, C_{\mathrm{d}} r_{D} n)$ yields that, $\prob_{0}$-a.s.,
  \begin{align*}
    \limsup_{\varepsilon \downarrow 0}\,
    \limsup_{n \to \infty} I_{2}(\omega, n, \varepsilon)
    \;\leq\;
    \limsup_{\varepsilon \downarrow 0}\,
    4 c \frac{(2C_{\mathrm{d}} r_{D})^{d}}{C_{\mathrm{V}}}\,
    \mean_{0}\bigl[\nu(0)^{q}\bigr]^{1/q}\, \|f\|_{\infty}^{2}\, \varepsilon
    \;=\;
    0.
  \end{align*}
  Finally, we address the term $I_{3}$. Again, in view of \eqref{eq:Egorov:Lusin}, we obtain that
  \begin{align*}
    I_{3}(\varepsilon)
    % &\;=\;
    % 2 \|f\|_{\infty}^{2}\,
    % \biggl(\sup_{x \in D} \Mean_{x}^{\Sigma}\bigl[\tau_{D}(W)\bigr] \biggr)\,
    % \int_{D} | f(x) - f_{\varepsilon}(x) |\, \mathrm{d}x
    % \\[.5ex]
    &\;\leq\;
    4 \|f\|_{\infty}^{2}\,
    \biggl(\sup_{x \in D} \Mean_{x}^{\Sigma}\bigl[\tau_{D}(W)\bigr] \biggr)\,
    |D \setminus K_{\varepsilon}|
    \;\leq\;
    c \|f\|_{\infty}^{2}\, \varepsilon.
  \end{align*}
  Thus, $\limsup_{\varepsilon \downarrow 0} I_{3}(\varepsilon) = 0$. This completes the proof.
\end{proof}
\begin{prop}[Tightness]
  \label{prop:DGFF:tightness}
  Let $d \geq 2$ and $D \subset \mathbb{R}^{d}$ be a bounded, strongly regular domain. Suppose there exist $\theta \in (0, 1)$ such that Assumptions~\ref{ass:law} and \ref{ass:cluster} hold. Moreover, assume that $\mean\bigl[\nu^{\omega}(0)^{d'/2} \indicator_{\{0 \in \mathcal{C}_{\infty}\}}\bigr] < \infty$ with $d' = (d-\theta)/(1-\theta)$. Then, $\prob_{0}$-a.s., the sequence $(\Phi_{n}^{D})_{n \in \mathbb{N}}$ is tight in $H^{-s}(D)$ for any $s > d/2$.
\end{prop}
\begin{proof} 
  \textsc{Step 1.} In a first step, we show that, for $\prob_{0}$-a.e.\ $\omega$,
  \begin{align}
    \label{eq:uniform:boundedness}
    \sup_{n \in \mathbb{N}}
    \meanPhi^{\omega}\Bigl[ \| \Phi_{n}^{D} \|_{H^{-s}(D)}^{2}\Bigr]
    \;<\;
    \infty.
  \end{align}
  By \cite[Theorem~6.3.1]{Da95}, for any bounded domain $D \subset \mathbb{R}^{d}$, the negative Dirichlet Laplacian, $-\Delta$, on $D$ has an empty essential spectrum and a compact resolvent. Moreover, the eigenvalues $0 < \lambda_{1}^{D} \leq \lambda_{2}^{D} \leq \ldots$, of $-\Delta$, written in increasing order and repeated according to their multiplicity, satisfy
  \begin{align}
    \label{eq:Dirichlet:eigenvalue:comparison}
    c_{1}\, k^{2/d} \;\leq\; \lambda_{k}^{D} \;\leq\; c_{2}\, k^{2/d},
    \qquad \forall\, k \in \mathbb{N}.
  \end{align}
  for some $0 < c_{1} \leq c_{2} < \infty$. Furthermore, by \cite[Corollary~4.2.3]{Da95} there exists a complete orthonormal set of eigenfunctions $(e_{k}^{D})_{k \in \mathbb{N}}$ of $-\Delta$ in $L^{2}(D)$ with corresponding eigenvalues $\lambda_{k}^{D}$. Hence, in view of \eqref{eq:scalar_product:Sobolev_space}, $\bigl((1-\Delta)^{-s/2}e_{k}^{D}\bigr)_{k \in \mathbb{N}}$ is an orthonormal basis of $H^{s}_0(D)$ and
  \begin{align*}
    \|\Phi_{n}^{D}\|_{H^{-s}(D)}^{2}
    \;\ldef\;
    \sum_{k=1}^{\infty}
    % \frac{\Phi_{n}^{D}(e_{n}^{D})^{2}}{(1+\lambda_{n}^{D})^{s}}
    \frac{1}{(1+\lambda_{k}^{D})^{s}}\, \Phi_{n}^{D}(e_{k}^{D})^{2},
  \end{align*}
  cf.\ \cite[Proof of Lemma~5.5]{GOS01}, \cite[Lemma~4.1]{CvG20} and \cite[p.~1872f.]{CR24} (note that the exponent in the definition of the Sobolev space and its dual given in \cite{CR24} differs from ours by a factor of $2$). 

  For any $\omega \in \Omega_{\mathrm{reg}} \cap \Omega_{0}^{*}$, let $(\lambda_{1}^{\omega}, u_{1})$ be the smallest solution of the (discrete) eigenvalue problem
  \begin{align}\label{eq:eigenvalue:discrete}
    \left\{
      \begin{array}{rcll}
        -\mathcal{L}^{\omega} u_{1} & \!\!=\!\! &
        \lambda_{1}^{\omega} u_{1}
        &\quad \text{on } nD \cap \mathcal{C}_{\infty}(\omega),
        \\[1ex]
        u_{1} &\!\!=\!\!& 0 &\quad \text{on } (n D \cap \mathcal{C}_{\infty}(\omega))^{\mathrm{c}}.   
      \end{array}
    \right.
  \end{align}
  By the Perron-Frobenius theorem we know that $\lambda_{1}^{\omega} > 0$ and is simple. Note that, by choosing $r_{D} \in (0, \infty)$ such that $D \subset [-r_{D}, r_{D}]^{d}$, we obtain, for any $\omega \in \Omega_{\mathrm{reg}} \cap \Omega_{0}^{*}$ and $n \geq \max\{ N_{0}(\omega)/r_{D}, 1 \}$, that $n D \cap \mathcal{C}_{\infty}(\omega) \subset B^{\omega}(0, C_{\mathrm{d}} r_{D} n)$. Hence, by the weighted Sobolev inequality in Proposition~\ref{prop:Sobolev:ineq},
  \begin{align*}
    &\Norm{u_{1}}{2, B^{\omega}(0, C_{\mathrm{d}} r_{D} n)}^{2}
    % \\
    % &\mspace{36mu}
    \;\leq\;
    C_{\mathrm{S}}\, (C_{\mathrm{d}} r_{D})^{2}\,
    \Norm{\nu^{\omega}}{d'/2, B^{\omega}(0, C_{\mathrm{d}} r_{D} n)}\,
    \lambda_{1}^{\omega} n^{2}\,
    \Norm{u_{1}}{2, B^{\omega}(0, C_{\mathrm{d}} r_{D} n)}^{2},
  \end{align*}
  which implies that
  \begin{align}
    \label{eq:Dirichlet:eigenvalue:estimate}
    \lambda_{1}^{\omega}
    \;\geq\;
    % \frac{1}{
    %   c \Norm{\nu^{\omega}}{d'/2, B^{\omega}(x_{0}, C_{\mathrm{d}} r_{D} n)} n^{2}
    % }.
    \Bigl(
      c\, \Norm{\nu^{\omega}}{d'/2, B^{\omega}(0, C_{\mathrm{d}} r_{D} n)} n^{2}
    \Bigr)^{\!-1}.
  \end{align}
  Further, recall that, $g_{n D}^{\omega}(z, z') = 0$ for any $z$ or $z'$ in $(nD \cap \mathcal{C}_{\infty}(\omega))^{\mathrm{c}}$. Hence, for $\prob_{0}$-a.e.\ $\omega$ and any $k \in \mathbb{N}$,
  \begin{align*}
    \varGFF^{\omega}\bigl[\Phi_{n}^{D}(e_{k}^{D})\bigr]
    &\;=\;
    n^{d-2}
    \int_{D \times D}
      e_{k}^{D}(y)\, e_{k}^{D}(y')\, g_{nD}^{\omega}(\lfloor ny \rfloor, \lfloor ny' \rfloor])\,
    \mathrm{d}y\, \mathrm{d}y'
    \\[.5ex]
    &\;=\;
    \frac{1}{n^{d+2}}\,
    \sum_{z, z' \in \mathcal{C}_{\infty}(\omega)}
    e_{k, n}^{D}(z/n)\, g_{nD}^{\omega}(z, z')\, e_{k, n}^{D}(z'/n)
    \\[.5ex]
    &\;\leq\;
    \frac{1}{\lambda_{1}^{\omega} n^{d+2}}
    \sum_{z \in \mathcal{C}_{\infty}(\omega)} e_{k, n}^{D}(z/n)^{2},
  \end{align*}
  where
  \begin{align*}
    e_{k,n}^{D}(x)
    \;\ldef\;
    n^{d} \int_{x + (-1/(2n), 1/(2n)]^{d}} e_{k, n}^{D}(y)\, \mathrm{d}y,
    \qquad x \in \mathbb{R}^{d}, \quad  k, n \in \mathbb{N}.
  \end{align*}
  Since, by Jensen's inequality,
  \begin{align*}
    \sum_{z \in \mathcal{C}_{\infty}(\omega)} e_{k, n}^{D}(z/n)^{2}
    \;\leq\;
    n^{d} \sum_{z \in \mathcal{C}_{\infty}(\omega)}
    \int_{z/n + (-1/(2n), 1/(2n)]^{d}} e_{k}^{D}(y)^{2}\, \mathrm{d}y
    \;\leq\;
    n^{d} \int_{D} e_{k}^{D}(y)^{d}\, \mathrm{d}y
  \end{align*}
  and $\int_{D} e_{k}^{D}(y)^{2}\, \mathrm{d}y = 1$ for any $k \in \mathbb{N}$, we obtain, for $\prob_{0}$-a.e.\ $\omega$,
  \begin{align*}
    \meanPhi^{\omega}\Bigl[\|\Phi_{n}^{D}\|_{H^{-s}(D)}^{2}\Bigr]
    \;\leq\;
    \frac{1}{\lambda_{1}^{\omega} n^{2}}\,
    \sum_{k=1}^{\infty} \frac{1}{(1 + \lambda_{k}^{D})^{s}}
    &\overset{\eqref{eq:Dirichlet:eigenvalue:comparison}}{\;\leq\;}
    \frac{1}{\lambda_{1}^{\omega} n^{2}}\,
    \sum_{k=1}^{\infty} \frac{1}{(1 + c_{1} k^{2/d})^{s}},
  \end{align*}
  which is finite as $s > d/2$. Combining this with \eqref{eq:Dirichlet:eigenvalue:estimate}, we get that
  \begin{align*}
    \meanPhi^{\omega}\Bigl[\|\Phi_{n}^{D}\|_{H^{-s}(D)}^{2}\Bigr]
    \;\leq\;
    c\, \Norm{\nu^{\omega}}{d'/2, B^{\omega}(0, C_{\mathrm{d}} r_{D} n)}.
  \end{align*}
  By the volume regularity (see Assumption~\ref{ass:cluster}) and the ergodic theorem, we obtain for $\prob_{0}$-a.e.\ $\omega$,
  \begin{align*}
    \limsup_{n \to \infty}
    \Norm{\nu^{\omega}}{d'/2, B^{\omega}(0, C_{\mathrm{d}} r_{D} n)}^{d'/2}
    &\;\leq\;
    \limsup_{n \to \infty} \frac{1}{C_{V}}\,
    \Norm{\nu^{\omega}}{d'/2, B(0, C_{\mathrm{d}} r_{D} n)}^{d'/2}
    \\[.5ex]
    &\;=\;
    \frac{1}{C_{V}} \mean\bigl[\nu(0)^{d'/2}\, \indicator_{\{0 \in \mathcal{C}_{\infty}\}}\bigr]
    \;<\;
    \infty,
  \end{align*}
  and \eqref{eq:uniform:boundedness} follows.
  \smallskip

  \textsc{Step 2.} In order to prove tightness, fix $s > d/2$. Then, there exists $t > d/2$ such that $s > t > d/2$. By \emph{Step 1}, for any $\varepsilon > 0$ there exists $R_{\varepsilon} \in (0, \infty)$ such that, for $\prob_{0}$-a.e.\ $\omega$,
  \begin{align*}
    \sup_{n \in \mathbb{N}}
    \probPhi^{\omega}\bigl[ \| \Phi_{n}^{D} \|_{H^{-t}}(D)^{2} > R_{\varepsilon}\bigr]
    \;\leq\;
    \frac{1}{R_{\varepsilon}}
    \sup_{n \in \mathbb{N}}
    \meanPhi^{\omega}\Bigl[ \| \Phi_{n}^{D} \|_{H^{-t}(D)}^{2}\Bigr]
    \;\leq\;
    \varepsilon.
  \end{align*}
  On the other hand, by Rellich's theorem, see \cite[Proposition~4.3.4]{Ta23}, the inclusion operator $H_{0}^{s}(d) \hookrightarrow H_{0}^{t}(D)$ is compact. This implies that, for any $R > 0$, the closed ball $\bigl\{G \in H^{-t}(D) : \| G \|_{H^{-t}(D)} \leq R\bigr\}$ is compact in $H^{-s}(D)$. Hence, we have shown tightness of $(\Phi_{n}^{D})_{n \in \mathbb{N}}$ in $H^{-s}(D)$. 
\end{proof}
\begin{proof}[Proof of Theorem~\ref{thm:scalingCGF_rg}]
  The result follows now immediately from Theorem~\ref{thm:scalingCGF_rg:fdd} and Proposition~\ref{prop:DGFF:tightness}.
\end{proof}

\section{Local limit theorem for the Green kernel} \label{sec:lclt}
\subsection{An a-priori bound on Green's functions}
In this subsection, we prove an upper bound on the Green's functions subject to Dirichlet boundary conditions that is needed for the proof of Theorem~\ref{thm:LCLT_green}. For this purpose, we first state a maximal inequality for harmonic functions.
\begin{theorem}\label{thm:max_ineq:harm}
  Let $d \geq 2$, and suppose there exists $\theta \in (0, 1)$ such that Assumptions~\ref{ass:law} and~\ref{ass:cluster}-(i) hold. Then, for any $p, q \in (1, \infty]$ satisfying $1/p + 1/q < 2(1 - \theta)/(d - \theta)$, there exist $C_{5} \equiv C_{5}(\theta, d, p, q) \in [1, \infty)$ and $\kappa \equiv \kappa(\theta, d, p, q) \in (1, \infty)$ such that for any $\omega \in \Omega_{\mathrm{reg}} \cap \Omega_{0}^{*}$, $R \geq 1$ and $n \geq 2(N_{0}(\omega) \vee R^{\kappa_{1}})$, where $\kappa_{1} \equiv \kappa_{1}(\theta, d, p, q) \in (0, \infty)$ is as in Proposition~\ref{prop:Max_ineq_poisson}, the following holds. For any $x \in B^{\omega}(0,R n)$ and any function $u\colon \mathcal{C}_{\infty}(\omega) \to [0, \infty)$ such that $\mathcal{L}^{\omega} u = 0$ on $B^{\omega}(x, n)$ and any $1/2 \leq \sigma' < \sigma \leq 1$,
    \begin{align} \label{eq:EHI:max_est2}
      \max_{z \in B^{\omega}(x, \sigma' n)} u(z)
      \;\leq\;   
      C_{5}
      \Biggl(
        \frac{%
          1 \vee \Norm{\mu^{\omega}}{p, B^{\omega}(x, n)}\, \Norm{\nu^{\omega}}{q, B^{\omega}(x, n)}
        }
        {(\sigma - \sigma')^{2}}
      \Biggr)^{\!\!\kappa}\; 
      \Norm{u}{1, B^{\omega}(x, \sigma n)}.
    \end{align}
\end{theorem}
\begin{proof}
 This follows from \cite[Corollary~3.4]{ADS16} with $\alpha = 1$ and $d$ replaced by $d'=(d-\theta')/(1-\theta')$ for some $\theta'\in (\theta, 1)$, which requires the weighted Sobolev inequality in Proposition~\ref{prop:Sobolev:ineq} as an input. 
Note that  by choosing  $\theta' \in (\theta, 1)$ as in \eqref{eq:choice_theta_prime}, the condition on $p$ and $q$ in the statement ensures that $1/p + 1/q < 2/d'$. 
\end{proof}
\begin{coro}\label{cor:Green_apriori}
  In the setting of Theorem~\ref{thm:max_ineq:harm} there exists $C_{6} \equiv C_{6}(\theta, d, p, q) \in (0, \infty)$ such that, for any $\omega \in \Omega_{\mathrm{reg}} \cap \Omega_{0}^{*}$, $R \geq 1$, $K \in [1, R]$ and $n \geq 2(N_{0}(\omega) \vee R^{\kappa_{1}})$ the following hold. Let $x_{0} \in B^{\omega}(0, R n)$ and $x \in B^{\omega}(x_{0}, K n) \cap B^{\omega}(0, R n)$ such that $B^{\omega}(x, n) \subset B^{\omega}(x_{0}, K n)$, and $y \in B^{\omega}(x_{0}, K n) \setminus B^{\omega}(x, n)$. Then, for any $1/2 \leq \sigma' < \sigma \leq 1$ and $z \in B^{\omega}(x, \sigma' n)$,
  \begin{align} \label{eq:green_apriori}
    % \max_{z \in B^{\omega}(x_{0}, \sigma' n)}
    g^{\omega}_{B^{\omega}(x_{0}, K n)}(z, y)
    \;\leq\;   
    C_{6} K^{2}
    \Biggl(
      \frac{%
        1 \vee \Norm{\mu^{\omega}}{p, B^{\omega}(x, n)}\,
        \Norm{\nu^{\omega}}{q, B^{\omega}(x, n)}
      }
      {(\sigma - \sigma')^{2}}
    \Biggr)^{\!\!\kappa}\;
    \Norm{\nu^{\omega}}{q, B^{\omega}(x_{0}, Kn)}\,
    n^{2-d}.
  \end{align}
\end{coro}
\begin{proof}
  Fix some $\omega \in \Omega_{\mathrm{reg}} \cap \Omega_{0}^{*}$ and $R \geq 1$. Since $y \in B^{\omega}(x_{0}, K n) \setminus B^{\omega}(x, n)$, the function $ g^{\omega}_{B^{\omega}(x_{0}, Kn)}(\cdot, y)$ is harmonic on $B^{\omega}(x, n)$. Therefore, by applying Theorem~\ref{thm:max_ineq:harm}, we obtain that, for any $n \geq 2(N_{0}(\omega) \vee R^{\kappa_{1}})$ and $1/2 \leq \sigma' < \sigma \leq 1$,
  \begin{align} \label{eq:appl_maxinequ}
    &\max_{z \in B^{\omega}(x, \sigma' n)} g^{\omega}_{B^{\omega}(x_{0}, K n)}(z, y)
    \nonumber\\[.5ex]
    &\mspace{36mu}\leq\;
    C_{5}
    \Bigg(
      \frac{
        1 \vee \Norm{\mu^{\omega}}{p, B^{\omega}(x, n)}\,
        \Norm{\nu^{\omega}}{q, B^{\omega}(x, n)}
      }{(\sigma - \sigma')^{2}}
    \Bigg)^{\!\!\kappa}\; 
    \Norm{g^{\omega}_{B^{\omega}(x_{0}, K n)}(\cdot, y)}{1, B^{\omega}(x, \sigma n)}.
  \end{align}
  Further, by using both the reversibility and the probabilistic representation of the Dirichlet Green kernel, we find that
  \begin{align*}
    \sum_{z \in B^{\omega}(x, \sigma n)}\mspace{-6mu} g^{\omega}_{B^{\omega}(x_{0}, K n)}(z, y)
    \;=\;
    \Mean_{y}^{\omega}\biggl[
      \int_{0}^{\tau_{B^{\omega}(x_{0}, K n)}} \indicator_{\{X_{t} \in B^{\omega}(x, \sigma n)\}}\, \mathrm{d}t
    \biggr]
    \;\leq\;
    \Mean_{y}^{\omega}\bigl[ \tau_{B^{\omega}(x_{0}, K n)} \bigr].
  \end{align*}
  Since the ball $B^{\omega}(x, n)$ is regular, it follows that $|B^{\omega}(x, \sigma n)| \geq C_{\mathrm{V}} (\sigma n)^{d} \geq C_{\mathrm{V}} n^{d} / 2^{d}$. Therefore, by applying the mean exit time estimate established in Corollary~\ref{cor:exit} with $n$ replace by $Rn$ and considering that $n \geq 2(N_{0}(\omega) \vee R^{\kappa_{1}})$ implies $Kn \geq N_{0} \vee (R/K)^{\kappa_{1}}$, we conclude that, for any $y \in B(x_{0}, Kn) \setminus B^{\omega}(x, n)$,
  \begin{align*}
    \Norm{g^{\omega}_{B^{\omega}(x_{0}, K n)}(\cdot, y)}{1, B^{\omega}(x, \sigma n)}
    &\;\leq\;
    \frac{2^{d}}{C_{\mathrm{V}} n^{d}}\, \Mean_{y}^{\omega}\bigl[ \tau_{B^{\omega}(x_{0}, K n)} \bigr]
    \\[.5ex]
    &\;\leq\;
    \frac{2^{d} C_{4}}{C_{\mathrm{V}}}\, K^{2}\, \Norm{\nu^{\omega}}{q, B^{\omega}(x_{0}, K n)}\, n^{2-d}.
  \end{align*}
  By combining this with \eqref{eq:appl_maxinequ} we get the claim.
\end{proof}

\subsection{H\"older regularity of Green's functions}
In this subsection we aim to establish H\"older regularity estimates for Green's functions with Dirichlet boundary conditions that are needed for the proof of Theorem~\ref{thm:LCLT_green}. Recall that, by \eqref{eq:PDE_green_killed}, for any $\omega \in \Omega_{0}^{*}$ and $y \in n D \cap \mathcal{C}_{\infty}(\omega)$, the function $\mathcal{C}_{\infty}(\omega) \ni x \mapsto g^{\omega}_{n D}(x, y)$ is $\mathcal{L}^{\omega}$-harmonic on $n D \setminus \{y\}$. Moreover, denote the oscillation of a function $u$ over a bounded subset $B$ of $\mathbb{Z}^{d}$ by 
\begin{align*}
  \osc\displaylimits_{x \in B} u(x)
  \;\ldef\;
  \sup_{x \in B} u(x) - \inf_{x \in B} u(x).
\end{align*}
\begin{theorem}[Oscillation inequality]\label{thm:osc}
  Let $d \geq 2$, and suppose there exists $\theta \in (0, 1)$ such that Assumptions~\ref{ass:law} and~\ref{ass:cluster}-(i) hold. Then, there exist $\vartheta \in (0, 1/2)$ such that for any $p, q \in (1, \infty]$ satisfying $1/p + 1/q < 2(1 - \theta)/(d - \theta)$, any $\omega \in \Omega_{\mathrm{reg}} \cap \Omega_{0}^{*}$, $R \geq 1$ and $n \geq 2^{8d} C_{\mathrm{W}} ( N_{0}(\omega) \vee R^{\kappa_{1}})$, where $\kappa_{1} \equiv \kappa_{1}(\theta, d, p, q) \in (0, \infty)$ is as in Proposition~\ref{prop:Max_ineq_poisson}, the following holds. For any $x_{0} \in B^{\omega}(0, R n)$ there exists
  \begin{align*}
    \gamma^{\omega}
    \;\equiv\;
    \gamma^{\omega}(x_{0}, n)
    \;\ldef\;
    \gamma \Bigl(
      \Norm{1 \vee \mu^{\omega}}{p, B^{\omega}(x_{0}, n)},
      \Norm{1 \vee \nu^{\omega}}{q, B^{\omega}(x_{0}, n)}
    \Bigr)
    \in
    (0, 1), 
  \end{align*}
  where $\gamma\colon [0, \infty)^2 \to (0, 1)$ is continuous and increasing in both components, such that for any positive function $u\colon \mathcal{C}_{\infty}(\omega) \to (0, \infty)$ with $\mathcal{L}^{\omega} u = 0$ on $B^{\omega}(x_{0}, n)$,
  \begin{align}\label{eq:oscillation_ineq}
    \osc\displaylimits_{B^{\omega}(x_{0}, \vartheta n)} u
    \;\leq\;
    \gamma^{\omega}\, \osc\displaylimits_{B^{\omega}(x_{0}, n)} u.
  \end{align}
\end{theorem}
\begin{proof}
  This has been shown for more general space-time harmonic functions for operators with possibly time-dependent conductances in \cite[Theorem~2.4]{ACS21}.

  Note that our assumptions ensure that \cite[Assumption~2.1]{ACS21} is satisfied. Indeed, part (i) holds as it is included in Assumption~\ref{ass:cluster}-(i). Next, let us choose $\theta' \in (\theta, 1)$ as in the proof of Proposition~\ref{prop:Max_ineq_poisson}. Then, the Sobolev inequality in part (ii) of \cite[Assumption~2.1]{ACS21} is given by \eqref{eq:sobolev} with $d'$ as in \eqref{eq:def:rho}. The weak Poincar\'{e} inequality appearing in part (iii) of \cite[Assumption~2.1]{ACS21} follows from the relative isoperimetric inequality in Assumption~\ref{ass:cluster}-(i) by applying a discrete co-area formula. More precisely, by \cite[Lemma~3.3.3]{Sa96}, there exists a constant $C_{\mathrm{P}}\in (0, \infty)$ such that for any $\omega \in \Omega^{*}_{0}$, $x_{0} \in \mathcal{C}_{\infty}(\omega)$ and $u\colon \mathcal{C}_{\infty}(\omega) \to \mathbb{R}$,
  \begin{align}\label{eq:PI1}
    \inf_{a \in \mathbb{R}}
    \Norm{u - a}{1, B^{\omega}(x_{0}, n)}
    \;\leq\; 
    C_{\mathrm{P}}
    \frac{n}{|B^{\omega}(x_{0}, n)|}\,
    \sum_{%
      \substack{x, y \in B^{\omega}(x_{0}, C_{\mathrm{W}} n) \\ x \sim y}
    } \bigl| u(x) - u(y) \bigr|.
  \end{align}
  Further, note that by the triangle inequality, for any $a \in \mathbb{R}$ and any non-empty $\mathcal{N} \subseteq B^{\omega}(x_{0}, n)$, writing $(u)_{\mathcal{N}} \ldef |\mathcal{N}|^{-1} \sum_{x \in \mathcal{N}} u(x)$, we have
  \begin{align*}
    \Norm{u - (u)_{\mathcal{N}}}{1, B^{\omega}(x_{0}, n)}
    &\;\leq\;
    \Norm{u-a}{1, B^{\omega}(x_{0}, n)}
    \,+\,
    \bigl| (u)_{\mathcal{N}} - a \bigr|
    \\
    &\;\leq\;
    \biggl( 1 + \frac{|B^{\omega}(x_{0}, n)|}{|\mathcal{N}|} \biggr)\, \Norm{u - a}{1, B^{\omega}(x_{0}, n)},
  \end{align*}
  and the desired weak Poincar\'{e} inequality follows from \eqref{eq:PI1}.
\end{proof}
The following generalisation of the ergodic theorem will help us to control ergodic averages on scaled balls with varying centre points.
\begin{prop}\label{prop:krengel_pyke}
  Let $R > 0$ and  $\mathcal{B} \ldef \bigl\{ B : B \text{ open Euclidean ball in } [-R, R]^{d}\bigr\}$.  Suppose that Assumption~\ref{ass:law}-(i) holds.  Then, for any $f \in L^{1}(\Omega)$,
  \begin{align*}
    \lim_{n \to \infty} \sup_{B \in \mathcal{B}} \,
    \biggl|
      \frac{1}{n^{d}}\, \sum_{x \in (nB) \cap \mathbb{Z}^{d}}\mspace{-12mu} f \circ \tau_{x}
      \,-\,
      |B| \cdot \mean\bigl[ f \bigr]
    \biggr|
    \;=\;
    0, 
    \qquad \prob\text{-a.s.},
  \end{align*}
  where $|B|$ denotes the Lebesgue measure of $B$.
\end{prop}
\begin{proof}
  See, for instance, \cite[Theorem~1]{KP87}.
\end{proof}
Recall the definition of $\pi_{n}$ right before Theorem~\ref{thm:LCLT_green}. 
\begin{coro} \label{cor:KrengelPyke}
  Let $d \geq 2$ and $r > 0$. Suppose there exist $\theta \in (0, 1)$ and $p, q \in [1, \infty]$ such that Assumptions~\ref{ass:law}, \ref{ass:cluster}-(i) and~\ref{ass:pq} hold. 
  \begin{enumerate}[(i)]
  \item
    For any $\delta > 0$, there exists a random variable $N_{1}\equiv N_{1}(\omega, \delta)$ such that, for $\prob$-a.e.~$\omega$, $N_{1}(\omega, \delta) < \infty$  and for any $x \in [-r, r]^{d}$,
    \begin{align}\label{eq:dist:pi_n(x)}
      |\pi_{n}(x) - n x|_{2} \;\leq\; \delta n,
      \qquad \forall\, n \geq N_{1}(\omega, \delta). 
    \end{align}

  \item
    There exist $\Omega_{c} \in \mathcal{F}$ with $\prob[\Omega_{c}] = 1$ and $\bar{\mu}, \bar{\nu} \in (0, \infty)$ such that for any  $\delta \in (0,1)$ the following holds. There exists a random variable $N_{2}\equiv N_{2}(\omega, \delta)$ such that, for all $\omega \in \Omega_{\mathrm{reg}} \cap \Omega_{0}^{*} \cap \Omega_{c}$, $x \in [-r, r]^{d}$,
    \begin{align*}
      \sup_{n \geq N_2} \Norm{1 \vee \mu^{\omega}}{p, B^{\omega}(\pi_{n}(x), \delta n)}
      \;\leq\;
      \bar{\mu}
      \quad \text{and} \quad
      \sup_{n \geq N_2} \Norm{1 \vee \nu^{\omega}}{q, B^{\omega}(\pi_{n}(x), \delta n)}
      \;\leq\;
      \bar{\nu}.
    \end{align*}
    In particular, $\bar \gamma \ldef \gamma(\bar \mu, \bar\nu)\in (0,1)$.
  \end{enumerate}
\end{coro}
\begin{proof}
  (i) Recall that $B_{\delta}(x) \subset \mathbb{R}^{d}$  denotes an open Euclidean ball with centre $x$ and radius $\delta$. Then, for any $x \in [-r, r]^{d}$ and $\delta > 0$, an application of Proposition~\ref{prop:krengel_pyke} with $R = r + \delta$ yields
  \begin{align*}
    \lim_{n \to \infty} \frac{1}{n^{d}}
    \mspace{-6mu}\sum_{z \in (n B_{\delta}(x)) \cap \mathbb{Z}^{d}}\mspace{-6mu}
    \mathbbm{1}_{\{z \in \mathcal{C}_{\infty}(\omega)\}}
    \;=\;
    |B_{\delta}(x)|\, \prob[0 \in \mathcal{C}_{\infty}],
    \qquad \prob\text{-a.s.}
  \end{align*}
  Hence, there exists a non-negative random variable $N_1$ such that, for $\prob$-a.e.~$\omega$, it holds that $N_{1}(\omega,\delta) < \infty$ and $n B_{\delta}(x) \cap \mathcal{C}_{\infty}(\omega) \ne \emptyset$ for any $n \geq N_{1}(\omega,\delta)$. Consequently, $\pi_{n}(x) \in n B_{\delta}(x) \cap \mathcal{C}_{\infty}(\omega)$ which implies that $|\pi_{n}(x) - nx|_{2} \leq \delta n$.

  (ii) Using (i), for $\prob$-a.e.~$\omega$, we have $n B_{\delta}(x) \cap \mathcal{C}_{\infty}(\omega) \ne \emptyset$ for any $n \geq N_1(\omega,\delta)$. Consequently, $\pi_{n}(x) \in n B_{\delta}(x) \cap \mathcal{C}_{\infty}(\omega)$ for all $n \geq N_{1}(\omega, \delta)$. Therefore, by using the volume regularity (see Assumption~\ref{ass:cluster}-(i)) and Proposition~\ref{prop:krengel_pyke}, there exist $\Omega_{c} \in \mathcal{F}$ with $\prob[\Omega_{c}] = 1$ and a random variable $N_{2}$ such that, for all $\omega \in \Omega_{\mathrm{reg}} \cap \Omega_{0}^{*} \cap \Omega_{c}$, $N_{2}(\omega, \delta) < \infty$ with
  \begin{align*}
    \sup_{n \geq N_{2}}
    \Norm{1 \vee \mu^{\omega}}{p, B^{\omega}(\pi_{n}(x), \delta n)}^{p}
    &\;\leq\;
    \sup_{n \geq N_{2}}
    \frac{1}{C_{\mathrm{V}} (\delta n)^{d}}
    \sum_{z \in (n B_{2 \delta}(x) ) \cap \mathbb{Z}^{d}} \bigl(1 \vee \mu^{\omega}(z)\bigr)^{p}\,
    \indicator_{\{z \in \mathcal{C}_{\infty}(\omega)\}}
    \\[.5ex]
    &\;\leq \;
    % \frac{|B_{2\delta}(x)|}{C_{\mathrm{V}}}
    2 \frac{4^{d}}{C_{\mathrm{V}}}
    \mean\bigl[
      \bigl(1 \vee \mu(0) \bigr)^{p} \indicator_{\{0 \in \mathcal{C}_{\infty}\}}
    \bigr]
    \;\rdef\;
    \bar{\mu}^{p},
  \end{align*}
  and similarly,
  \begin{align*}
    \sup_{n \geq N_{2}}
    \Norm{1 \vee \nu^{\omega}}{q, B^{\omega}(\pi_{n}(x), \delta n)}^{q}
    &\;\leq\;
    2 \frac{4^d}{C_{\mathrm{V}}}
    \mean\bigl[
      \bigl(1 \vee \nu(0)\bigr)^{q} \indicator_{\{0 \in \mathcal{C}_{\infty}\}}
    \bigr]
    \;\rdef\;
    \bar{\nu}^{q}.
  \end{align*}
  Since $\bar{\mu}, \bar{\nu} \in (0, \infty)$ we have $\bar \gamma \ldef \gamma(\bar{\mu}, \bar{\nu}) \in (0, 1)$.
\end{proof}
Now, we are able to show the required H\"older regularity estimate for $g^{\omega}_{n D}$. Recall the definition of the set $K_{\varepsilon, \delta}$ in \eqref{eq:defK_eps}.
\begin{prop} \label{prop:Holdercont}
  Let $d \geq 2$ and $D \subset \mathbb{R}^{d}$ be a bounded domain. Suppose there exist $\theta \in (0, 1)$ and $p, q \in [1, \infty]$ such that Assumptions~\ref{ass:law},~\ref{ass:cluster} and~\ref{ass:pq} hold. Then, exist $C_{5} \in (0, \infty)$ and $\vartheta \in (0, 1/2)$ such that for any $\varepsilon, \delta > 0$ such that $0 < \varepsilon < \delta$ and $K_{\varepsilon, \delta} \ne \emptyset$ the following hold. For any $\eta \in (0, \varepsilon / 8)$ there exists a random variable $N_{3}$ (depending on $d, p, q, \theta, \eta$) that is $\prob_{0}$-a.s.\ finite such that, for any $(x, y) \in K_{\varepsilon, \delta}$ and $n \geq N_{3}$,
  \begin{align}
    \label{eq:oscilation:Green:final}
    \sup_{z, z' \in B^{\omega}(\pi_{n}(x), \eta n)}
    \bigl|
      g^{\omega}_{nD}(z, \pi_{n}(y)) -g^{\omega}_{nD}(z', \pi_{n}(y))
    \bigr|
    \;\leq\;
    C_{5}\frac{\eta^{\varrho}}{\varepsilon^{d}}\, n^{2-d},
  \end{align}
  where $\varrho \ldef \log\bar{\gamma} / \log\vartheta$.
\end{prop}
\begin{proof}
  Fix $\varepsilon, \delta > 0$ such that $0 < \varepsilon < \delta$ and $K_{\varepsilon, \delta} \ne \emptyset$, and choose $r_{D} \in (0, \infty)$ such that $D \subset [-r_{D}, r_{D}]^{d}$. Then, for any $\omega \in \Omega_{\mathrm{reg}} \cap \Omega_{0}^{*}$ and $n \geq \max\{ N_{0}(\omega)/r_{D}, 1 \}$, we have that $n D \cap \mathcal{C}_{\infty}(\omega) \subset B^{\omega}(0, C_{\mathrm{d}} r_{D} n)$. The remaining proof comprises two steps.
  \smallskip
  
  \noindent  
  \textsc{Step 1.} For any $k \in \mathbb{N}_{0}$ set $\delta_{k} \ldef \varepsilon \vartheta^{k} / 2$ with $\vartheta$ as Theorem~\ref{thm:osc}. Then, for any $\eta \in (0, \varepsilon / 8)$ there exists $k_{0} \equiv k_{0}(\eta) \in \mathbb{N}_{0}$ such that $\delta_{k_{0} + 1} < \eta \leq \delta_{k_{0}}$. Further, for any $\omega \in \Omega_{\mathrm{reg}} \cap \Omega_{0}^{*}$ and $n \in \mathbb{N}$, let $x_{0}, y_{0} \in nD \cap \mathcal{C}_{\infty}(\omega)$ be such that $B^{\omega}(x_{0}, \varepsilon n) \subset n D \cap \mathcal{C}_{\infty}(\omega)$ and $y_{0} \not\in B^{\omega}(x_{0}, \varepsilon n)$. This ensures that the function $z \mapsto g_{n D}^{\omega}(z, y_{0})$ is positive and harmonic on $B^{\omega}(x_{0}, \delta_{k} n)$ for any $k \in \{0, \ldots, k_{0}\}$.
  
  Now, let $\omega \in \Omega_{\mathrm{reg}} \cap \Omega_{0}^{*} \cap \Omega_{c}$ and
  \begin{align*}
    n
    \;\geq\;
    \tilde N_3(\omega)
    \;\ldef\;
    \eta^{-1}
    \bigl(
      2^{8d} C_{\mathrm{W}}
      \bigl( N_{0}(\omega) \vee R^{\kappa_{1}} \bigr)
      \vee
      \max_{0 \leq k \leq k_{0}} N_{2}(\omega, \delta_{k})
    \bigr)
  \end{align*}
  with $R \ldef C_{\mathrm{d}} r_{D} / \eta$ and $\kappa_{1} \equiv \kappa_{1}(\theta, d, p, q) \in (0, \infty)$ as in Proposition~\ref{prop:Max_ineq_poisson}. Then, by Theorem~\ref{thm:osc},
  \begin{align*}
    % \osc\displaylimits_{B_{k}} g_{nD}^{\omega}(\cdot, y_{0})
    \osc\displaylimits_{B^{\omega}(x_{0}, \delta_{k} n)} g_{nD}^{\omega}(\cdot, y_{0})
    \;\leq\;
    \bar{\gamma}\,
    % \osc\displaylimits_{B_{k-1}} g_{nD}^{\omega}(\cdot, y_{0}),
    \osc\displaylimits_{B^{\omega}(x_{0}, \delta_{k-1} n)} g_{nD}^{\omega}(\cdot, y_{0}),
    \qquad \forall\, k \in \{1, \ldots k_{0}\},
  \end{align*}
  and by iterating this we obtain
  \begin{align*}
    % \osc\displaylimits_{B_{k_{0}}} g_{nD}^{\omega}(\cdot, y_{0})
    \osc\displaylimits_{B^{\omega}(x_{0}, \delta_{k_{0}} n)} g_{nD}^{\omega}(\cdot, y_{0})
    \;\leq\;
    \bar{\gamma}^{k_{0}}\,
    % \sup_{B_{0}} g_{nD}^{\omega}(\cdot, y_{0}).
    \sup_{B^{\omega}(x_{0}, \delta_{0} n)} g_{nD}^{\omega}(\cdot, y_{0}). 
  \end{align*}
  By the choice of $k_{0}$, we have $B^{\omega}(x_{0}, \eta n) \subset B^{\omega}(x_{0}, \delta_{k_{0}} n)$ and $\bar{\gamma}^{k_{0}} \leq (\frac{8}{\varepsilon}\eta)^\varrho$. Hence,
  \begin{align}
    \label{eq:osc:Green}
    \sup_{z, z' \in B^{\omega}(x_{0}, \eta n)}
    \bigl| g_{n D}^{\omega}(z, y_{0}) - g_{nD}^{\omega}(z'\!, y_{0})\bigr| 
    \;\leq\; 
    c\, \eta^{\varrho} \sup_{B^{\omega}(x_{0}, \varepsilon n / 2)} g_{nD}^{\omega}(\cdot, y_{0}).
  \end{align}
  Moreover, by applying Corollary~\ref{cor:Green_apriori} with the choice $\sigma' = 1/2$, $\sigma = 1$ and $R = K = C_{\mathrm{d}} r_{D} / \varepsilon$, we get
  \begin{align}
    \label{eq:max:Green:osc}
    \sup_{B^{\omega}(x_{0}, \varepsilon n / 2)} g_{nD}^{\omega}(\cdot, y_{0})
    \;\leq\;
    C_{6} \frac{(C_{\mathrm{d}} r_{D})^{2}}{\varepsilon^{2}}
    \bigl(4 \bar{\mu} \bar{\nu}\bigr)^{\kappa} \bar{\nu} (\varepsilon n)^{2-d}
    \;\leq\;
    \frac{c}{\varepsilon^{d}}\, n^{2-d}.
  \end{align}
  By combining \eqref{eq:osc:Green} and \eqref{eq:max:Green:osc} we  obtain that, for any $\omega \in \Omega_{\mathrm{reg}} \cap \Omega_{0}^{*} \cap \Omega_{c}$ and $n \geq  \tilde N_{3}$,
  \begin{align}
    \sup_{z, z' \in B^{\omega}(x_{0}, \eta n)}
    \bigl| g_{n D}^{\omega}(z, y_{0}) - g_{nD}^{\omega}(z'\!, y_{0})\bigr| 
    \;\leq\; 
    c\, \frac{\eta^{\varrho}}{\varepsilon^{d}}\, n^{2-d}.
  \end{align}

  \textsc{Step 2.} Set $r \ldef (\delta - \varepsilon) \wedge \varepsilon/8$. Then, by applying Corollary~\ref{cor:KrengelPyke}-(i) with $R = r_{D}$, we get that, for $\prob_{0}$-a.e.\ $\omega$, every $n \geq N_{1}(\omega,r)$ and any $(x, y) \in K_{\varepsilon, \delta}$,
  \begin{align*}
    |\pi_{n}(x) - n x|_{2} \;\leq\; r n
    \qquad \text{and} \qquad
    |\pi_{n}(y) - n y|_{2} \;\leq\; r n.
  \end{align*}
  Hence,
  \begin{align*}
    \operatorname{dist}(\pi_{n}(x), \partial (n D))
    \;\geq\;
    \operatorname{dist}(n x, \partial (n D)) - |\pi_{n}(x) - n x|_{2}
    % \;\geq\;
    % (\delta - r) n
    \;\geq\;
    \varepsilon n
    \;>\;
    0,
  \end{align*}
  which implies that $\pi_{n}(x) \in n D$. Since, by construction, $\pi_{n}(x) \in \mathcal{C}_{\infty}(\omega)$ it follows that $\pi_{n}(x) \in n D \cap \mathcal{C}_{\infty}(\omega)$. In particular,
  \begin{align*}
    B^{\omega}(\pi_{n}(x), \varepsilon n / 2)
    \;\subset\;
    n D \cap \mathcal{C}_{\infty}(\omega),
    \qquad \text{for } \prob_{0}\text{-a.e. } \omega \text{ and any } n \geq N_{1}(\omega,r).  
  \end{align*}
  Similarly we have that $\pi_{n}(y) \in n D \cap \mathcal{C}_{\infty}(\omega)$ for $\prob_{0}$-a.e.\ $\omega$ and $n \geq N_{1}(\omega,r)$. Since
  \begin{align*}
    d^{\omega}(\pi_{n}(x), \pi_{n}(y))
    &\;\geq\;
    |\pi_{n}(x) - \pi_{n}(y)|_{2}
    \\
    &\;\geq\;
    |n x - n y|_{2} - |\pi_{n}(x) - n x|_{2} - |\pi_{n}(y) - n y|_{2}
    % \;\geq\;
    % (\varepsilon - 2r) n
    \;\geq\;
    \frac{3}{4} \varepsilon n,
  \end{align*}
  it follows that $\pi_{n}(y) \in (n D \cap \mathcal{C}_{\infty}(\omega)) \setminus B^{\omega}(\pi_{n}(x), \varepsilon n / 2)$ for $\prob_{0}$-a.e.\ $\omega$ and $n \geq N_{1}(\omega,r)$. Hence, for $\prob_{0}$-a.e.\ $\omega$ and any $n \geq N_{3}(\omega) \ldef \tilde N_3(\omega) \vee N_{1}(\omega,r)$, choosing $x_{0} = \pi_{n}(x)$, $y_{0} = \pi_{n}(y)$, we may apply \emph{Step~1} (with $\varepsilon$ replaced by $\varepsilon / 2$) to obtain \eqref{eq:oscilation:Green:final}.
\end{proof}

\subsection{Proof of the local limit theorem}
\begin{proof}[Proof of Theorem~\ref{thm:LCLT_green}]
  First, fix any $\delta, \varepsilon > 0$ with $0 < \varepsilon < \delta$ such that $K_{\varepsilon, \delta} \ne \emptyset$. Further, for $x \in \mathbb{R}^{d}$ and $r > 0$ set $C_{r}(x) \ldef x + [-r, r]^{d}$. Recall that, for any $\omega \in \Omega^{*}$ and $n \in \mathbb{N}$, we denote by $\pi_{n}\colon \mathbb{R}^{d} \to \mathcal{C}_{\infty}(\omega)$ the function that maps any $x \in \mathbb{R}^{d}$ to a closest point of $nx$ in $\mathcal{C}_{\infty}(\omega)$.
  
  The proof is divided in two steps. We first show a pointwise limit and then, using a covering argument, we establish the full result.  
  \smallskip

  \textsc{Step 1.} Fix any $(x_{0}, y_{0}) \in K_{\varepsilon, \delta}$ and $r \in (0, \varepsilon/(24 C_{\mathrm{d}}))$. Further, choose $r_{D} \in (0, \infty)$ such that $D \subset [-r_{D}, r_{D}]^{d}$. Then, $C_{r}(x_{0}) \subset D$ and $C_{r}(y_{0}) \subset D$. Let $f_{1}, f_{2} \in C(\mathbb{R}^{d}, \mathbb{R})$ be two non-negative functions with $\supp f_{1} \subset C_{r}(x_{0})$, $\supp f_{2} \subset C_{r}(y_{0})$ and $\int_{\mathbb{R}^{d}} f_{1}(x) \md x = \int_{\mathbb{R}^{d}} f_{2}(x) \md x = 1$. We claim that
  \begin{align}
    \label{eq:lclt:claim}
    \lim_{n \to \infty} n^{d-2} g_{n D}^{\omega}(\pi_{n}(x_{0}), \pi_{n}(y_{0}))
    \;=\;
    \frac{g_{D}^{\Sigma}(x_{0}, y_{0})}{\prob[0 \in \mathcal{C}_{\infty}]},
    \qquad \prob_{0}\text{-a.s.}
  \end{align}
  Define $M(n, r, \omega) \ldef n^{-2d} \sum_{z_{1}, z_{2} \in \mathcal{C}_{\infty}(\omega)} f_{1}(z_{1} / n) f_{2}(z_{2}/n)$ and the measurable function $F_{2}\colon D([0, \infty), \mathbb{R}^{d}) \to [0, \infty)$ by $w \mapsto F_{2}(w) \ldef \int_{0}^{\tau_{D}(w)} f_{2}(w_{t})\, \mathrm{d}t$, and set
  \begin{align}
    \label{eq:lclt:split}
    J(n, r, \omega)
    &\;\ldef\;
    \frac{1}{n^{d}}\mspace{-6mu}\sum_{z \in \mathcal{C}_{\infty}(\omega)}\mspace{-6mu}
    f_{1}(z/n)\, \Mean_{z}^{\omega}\bigl[F_{2}(X^{n})\bigr]
    -
    \prob\bigl[0 \in \mathcal{C}_{\infty}\bigr]\,
    \int_{\mathbb{R}^{d}} f_{1}(x) \Mean_{x}^{\Sigma}\bigl[F_{2}(W)\bigr]\, \mathrm{d}x
    \nonumber\\[.5ex]
    &\mspace{10mu}=\;
    J_{1}(n, r, \omega) + J_{2}(n, r, \omega) + J_{3}(n, r, \omega) + J_{4}(r),
  \end{align}
  where
  \begin{align*}
    J_{1}(n, r, \omega)
    &\;\ldef\;
    \frac{1}{n^{d+2}}\mspace{-6mu}\sum_{z_{1}, z_{2} \in \mathcal{C}_{\infty}(\omega)}\mspace{-12mu}
    f_{1}(z_{1}/n)\, f_{2}(z_{2}/n)\,
    \Bigl(
      g_{nD}^{\omega}(z_{1}, z_{2})
      -
      g_{nD}^{\omega}\bigl(\pi_{n}(x_{0}), \pi_{n}(y_{0})\bigr)
    \Bigr),
    \\[.5ex]
    J_{2}(n, r, \omega)
    &\;\ldef\;
    % \prod_{i=1}^{2}
    % \biggl(
    %   \frac{1}{n^{d}} \mspace{-6mu}\sum_{z \in \mathcal{C}_{\infty}(\omega)}\mspace{-6mu} f_{i}(z/n)
    % \biggr)
    M(n, r, \omega)\,
    \biggl(
      n^{d-2}\,g_{nD}^{\omega}(\pi_{n}(x_{0}), \pi_{n}(y_{0}))
      -
      \frac{g_{D}^{\Sigma}(x_{0}, y_{0})}{\prob\bigl[0 \in \mathcal{C}_{\infty}\bigr]}
    \biggr),
    \\[.5ex]
    J_{3}(n, r, \omega)
    &\;\ldef\;
    \biggl(
      % \prod_{i=1}^{2}
      % \biggl(
      %   \frac{1}{n^{d}} \mspace{-6mu}\sum_{z \in \mathcal{C}_{\infty}(\omega)}\mspace{-6mu} f_{i}(z/n)
      % \biggr)\,
      % \frac{1}{\prob\bigl[0 \in \mathcal{C}_{\infty}\bigr]}
      \frac{M(n, r, \omega)}{\prob\bigl[0 \in \mathcal{C}_{\infty}\bigr]}
      -
      \prob\bigl[0 \in \mathcal{C}_{\infty}\bigr]
    \biggr)\,
    g_{D}^{\Sigma}(x_{0}, y_{0}),
    \\[.5ex]
    J_{4}(r)
    &\;\ldef\;
    \prob\bigl[0 \in \mathcal{C}_{\infty}\bigr]\,
    \int_{\mathbb{R}^{d}}
      \int_{\mathbb{R}^{d}}
        f_{1}(x)\, f_{2}(y) \bigl(g_{D}^{\Sigma}(x, y) - g_{D}^{\Sigma}(x_{0}, y_{0})\bigr)\,
      \mathrm{d}y\,
    \mathrm{d}x.
  \end{align*}
  By rearranging those terms we obtain that
  \begin{align*}
    &\biggl|
      n^{d-2}\, g_{nD}^{\omega}(\pi_{n}(x_{0}), \pi_{n}(y_{0}))
      - \frac{g_{D}^{\Sigma}(x_{0}, y_{0})}{\prob\bigl[0 \in \mathcal{C}_{\infty}\bigr]}
    \biggr|
    \\[.5ex]
    &\mspace{36mu}\leq\;
    % \prod_{i=1}^{2} \biggl(
    %   \frac{1}{n^{d}} \mspace{-6mu}\sum_{z \in \mathcal{C}_{\infty}(\omega)}\mspace{-6mu} f_{i}(z/n)
    % \biggr)^{\!\!-1}
    \frac{1}{M(n, r, \omega)}\,
    \bigl(
      |J(n, r, \omega)| + |J_{1}(n, r, \omega)| + |J_{3}(n, r, \omega)| + |J_{4}(r)|
    \bigr).
  \end{align*}
  Thus, it suffices to show that the right-hand side vanishes when we first take the limit $n \to \infty$ and then $r \downarrow 0$. First, by applying \cite[Theorem~3]{BD03}, it follows that $\lim_{n \to \infty} M(n, r, \omega) = \prob[0 \in \mathcal{C}_{\infty}]^{2}$ for $\prob$-a.e.~$\omega$. Together with Theorem~\ref{thm:FCLT:Poisson} this implies that, for $\prob$-a.e.\ $\omega$,
  \begin{align*}
    \lim_{n \to \infty} \frac{|J(n, r, \omega)| + |J_{3}(n, r, \omega)|}{M(n, r, \omega)}
    \;=\;
    0.
  \end{align*}
  Moreover, by the uniform continuity of Green kernel $g_{D}^{\Sigma}$, it follows that, for $\prob$-a.e.~$\omega$,
  \begin{align*}
    \lim_{r \downarrow 0} \lim_{n \to \infty} |J_{4}(r)| / M(n, r, \omega)
    \;=\;
    \lim_{r \downarrow 0} |J_{4}(r)| / \prob[0 \in \mathcal{C}_{\infty}]^{2}
    \;=\;
    0.
  \end{align*}

  It remains to consider the term $J_{1}$. Since $x_{0} \in [-r_{D}, r_{D}]^{d}$, Corollary~\ref{cor:KrengelPyke}-(i) implies that $|\pi_{n}(x_{0}) - nx_{0}| < r n /2$ for $\prob$-a.e.\ $\omega$ and $n \geq N_{1}(\omega, r/2)$. Thus, by using Assumption~\ref{ass:cluster}-(ii), we find that, for any $\omega \in \Omega_{\mathrm{reg}} \cap \Omega_{0}^{*}$ and any $n \geq N_{0}(\omega) / r_{D} \vee N_{1}(\omega, r/2)$,
  \begin{align}
    \label{eq:set:inclusion:1}
    n C_{r}(x_{0}) \cap \mathcal{C}_{\infty}(\omega)
    \;\subset\;
    \pi_{n}(x_{0}) + [-n 3r/2, n 3r/2]^{d} \cap \mathcal{C}_{\infty}(\omega)
    \;\subset\;
    B^{\omega}(\pi_{n}(x_{0}), 3 C_{\mathrm{d}} r n / 2 ),
  \end{align}
  and, similarly,
  \begin{align}
    \label{eq:set:inclusion:2}
    n C_{r}(y_{0}) \cap \mathcal{C}_{\infty}(\omega)
    \;\subset\;
    B^{\omega}(\pi_{n}(y_{0}), 3 C_{\mathrm{d}} r n / 2 ).
  \end{align}
  Further, by using the symmetry of the Green kernel with zero boundary condition, we obtain for any $z_{1} \in n C_{r}(x_{0}) \cap \mathcal{C}_{\infty}(\omega)$ and $z_{2} \in n C_{r}(y_{0}) \cap \mathcal{C}_{\infty}(\omega)$,
  \begin{align}
    \label{eq:split:sum:Green}
    &\bigl|
      g_{nD}^{\omega}(z_{1}, z_{2})
      -
      g_{nD}^{\omega}\bigl(\pi_{n}(x_{0}), \pi_{n}(y_{0})\bigr)
    \bigr|
    \nonumber\\[1ex]
    &\mspace{36mu}\leq\;
    \bigl|
      g_{nD}^{\omega}(z_{2}, z_{1})
      -
      g_{nD}^{\omega}\bigl(\pi_{n}(y_{0}), z_{1})\bigr)
    \bigr|
    +
    \bigl|
      g_{nD}^{\omega}(z_{1}, \pi_{n}(y_{0}))
      -
      g_{nD}^{\omega}\bigl(\pi_{n}(x_{0}), \pi_{n}(y_{0})\bigr)
    \bigr|.
  \end{align}
  Thus, recalling that $z_{1} \in B^{\omega}(\pi_{n}(x_{0}), 3 C_{\mathrm{d}}r n/2)$ by \eqref{eq:set:inclusion:1}, the Hölder continuity estimate in Proposition~\ref{prop:Holdercont} immediately yields that, for $\prob_{0}$-a.e.\ $\omega$ and any $n \geq N_{3}(\omega)$,
  \begin{align}
    \label{eq:split:sum:Green:2}
    \bigl|
      g_{nD}^{\omega}(z_{1}, \pi_{n}(y_{0}))
      -
      g_{nD}^{\omega}\bigl(\pi_{n}(x_{0}), \pi_{n}(y_{0})\bigr)
    \bigr|
    \;\leq\;
    C_{5} \frac{(3 C_{d} r / 2)^{\varrho}}{\varepsilon^{d}} n^{2-d}.
  \end{align}
  In order to estimate the first term on the right-hand side of \eqref{eq:split:sum:Green}, first note that for any $n \in \mathbb{N}$ and $z_{1} \in n C_{r}(x_{0})$,
  \begin{align*}
    |z_{1}/n - y_{0}|_{2}
    \;\geq\;
    |x_{0} - y_{0}|_{2} - |z_{1}/n - x_{0}|_{2}
    \;\geq\;
    \varepsilon - \frac{r}{2}
    \;\geq\;
    \varepsilon - \frac{\varepsilon}{16 C_{\mathrm{d}}}
    \;>\;
    \varepsilon/2,
  \end{align*}
  and, since $\delta > \varepsilon$,
  \begin{align*}
    \operatorname{dist}(z_{1}/n, \partial D)
    \;\geq\;
    \operatorname{dist}(x_{0}, \partial D) 
    \; \geq \;
    \delta - \frac{r}{2} 
    \;\geq\;
    \delta - \frac{\varepsilon}{16 C_{\mathrm{d}}}
    \;>\;
    \delta/2.
  \end{align*}
  Hence, $(z_{1}/n, y_{0}) \in K_{\varepsilon/2, \delta/2}$ for all $z_{1}\in n C_{r}(x_{0})$. In particular, every $z_{1} \in n C_{r}(x_{0}) \cap \mathcal{C}_{\infty}(\omega)$ is of the form $z_{1} = \pi_{n}(z_1')$ for some $z_{1}' \in [-r_{D}, r_{D}]^{d}$ such that $(z_{1}', y_{0}) \in K_{\varepsilon/2, \delta/2}$. Therefore, recalling that $z_{2} \in B^{\omega}(\pi_{n}(y_{0}), 3 C_{\mathrm{d}}r n/2)$ by \eqref{eq:set:inclusion:2}, by another application of Proposition~\ref{prop:Holdercont} on $K_{\varepsilon/2, \delta/2}$, we obtain for $\prob_{0}$-a.e.\ $\omega$, any $n \geq N_{3}(\omega)$ and $z_{2} \in n C_{r}(y_{0}) \cap \mathcal{C}_{\infty}(\omega)$,
  \begin{align}
    \label{eq:split:sum:Green:1}
    &\bigl|
      g_{nD}^{\omega}(z_{2}, z_{1})
      -
      g_{nD}^{\omega}\bigl(\pi_{n}(y_{0}), z_{1}\bigr)
    \bigr|
    \nonumber\\[.5ex]
    &\mspace{36mu}=\;
    \bigl|
      g_{nD}^{\omega}(z_{2}, \pi_{n}(z_{1}'))
      -
      g_{nD}^{\omega}\bigl(\pi_{n}(y_{0}), \pi_{n}(z_{1}')\bigr)
    \bigr|
    \;\leq\;
    C_{5} \frac{(3 C_{d} r / 2)^{\varrho}}{(\varepsilon/2)^{d}} n^{2-d}.
  \end{align}
  By combining \eqref{eq:split:sum:Green:1} and \eqref{eq:split:sum:Green:2} with \eqref{eq:split:sum:Green},
  \begin{align*}
    \bigl|
      g_{nD}^{\omega}(z_{1}, z_{2})
      -
      g_{nD}^{\omega}\bigl(\pi_{n}(x_{0}), \pi_{n}(y_{0})\bigr)
    \bigr|
    \;\leq\;
    c\, \frac{r^{\rho}}{\varepsilon^{d}}\, n^{2-d}.
  \end{align*}
  Hence, $\limsup_{r \downarrow 0} \limsup_{n \to \infty} |J_{1}(n, r, \omega)| \leq \limsup_{r \downarrow 0} c r^{\rho} \varepsilon^{-d} / \prob[0 \in \mathcal{C}_{\infty}]^{2} = 0$ for $\prob_{0}$-a.e.~$\omega$.
  \smallskip
  
  \textsc{Step 2.} To extend the pointwise convergence shown in \emph{Step~1} to uniform convergence in $K_{\varepsilon, \delta}$ we use a covering argument similar to one  in \cite{BH09,ACS21}. The idea is to approximate $K_{\varepsilon, \delta}$ by a finite lattice on which the pointwise convergence is uniform and then to exploit H\"older regularity estimates to conclude the uniform convergence on $K_{\varepsilon, \delta}$.

  For any $\eta \in (0, 1)$, we set $K_{\varepsilon, \delta, \eta} \ldef K_{\varepsilon, \delta} \cap (\eta \mathbb{Z}^{d} \times \eta \mathbb{Z}^{d})$. Clearly, the set $K_{\varepsilon, \delta, \eta}$ is non-empty for any sufficiently small $\eta$. Furthermore, since the number of elements in $K_{\varepsilon, \delta, \eta}$ is finite, \emph{Step~1} implies that, for any $\epsilon_{1} > 0$, $\prob_{0}$-a.e.~$\omega$ and $\eta \in (0, 1) \cap \mathbb{Q}$, there exists $N_{4}(\omega, \epsilon_{1}, \eta) \in \mathbb{N}$ such that
  \begin{align} \label{eq:conv_X}
    \sup_{(x_{0}, y_{0}) \in K_{\varepsilon, \delta, \eta}}
    \bigl|
      n^{d-2} g_{n D}^{\omega}(\pi_{n}(x_{0}), \pi_{n}(y_{0}))
      -
      g_{D}^{\Sigma}(x_{0}, y_{0})
    \bigr|
    \;\leq\;
    \epsilon_{1},
    \qquad \forall\, n \geq N_{4}(\omega, \epsilon_{1}, \eta).
  \end{align}
  Furthermore, for any $(x, y) \in K_{\varepsilon, \delta} \setminus K_{\varepsilon, \delta, \eta}$, we denote by $(x_{0}, y_{0})$ a closest point to $(x,y)$ in $K_{\varepsilon, \delta, \eta}$. Then, 
  \begin{align*}
    |x - x_{0}|_{\infty}
    \;\leq\;
    \eta
    \qquad \text{and} \qquad
    |y - y_{0}|_{\infty}
    \;\leq\;
    \eta.
  \end{align*}
  Since $x, x_{0}, y, y_{0} \in [-r_{D}, r_{D}]^{d}$, by Corollary~\ref{cor:KrengelPyke}-(i), for $\prob$-a.e.\ $\omega$ and $n \geq N_{1}(\omega, \eta/2)$,
  \begin{align*}
    \sup_{z \in \{x, x_{0}, y, y_{0}\}} |\pi_{n}(z) - n z|_{2}
    \;\leq\;
    \frac{\eta n}{2}.
  \end{align*}
  Then, by the same arguments as in \eqref{eq:set:inclusion:1}, we get that $\pi_{n}(x) \in B^{\omega}(\pi_{n}(x_{0}), 2 C_{\mathrm{d}} n)$ and $\pi_{n}(y) \in B^{\omega}(\pi_{n}(y_{0}), 2 C_{\mathrm{d}} n)$. Thus, by applying again Proposition~\ref{prop:Holdercont}, we obtain for $\prob_{0}$-a.e.\ $\omega$, any $\eta \in (0, \varepsilon/(16 C_{\mathrm{d}})) \cap \mathbb{Q}$ and $n \geq N_{3}(\omega)$,
  \begin{align}
    \label{eq:uniform_hoelder:1}
    \bigl|
      g_{n D}^{\omega} (\pi_{n}(y), \pi_{n}(x))
      -
      g_{n D}^{\omega} (\pi_{n}(y_{0}), \pi_{n}(x))
    \bigr|
    \;\leq\;
    C_{5} \frac{(2 C_{\mathrm{d}} \eta)^{\varrho}}{\varepsilon^{d}} n^{2-d}, 
  \end{align}
  and
  \begin{align}
    \label{eq:uniform_hoelder:2}
    \bigl|
      g_{n D}^{\omega} (\pi_{n}(x), \pi_{n}(y_{0}))
      -
      g_{n D}^{\omega} (\pi_{n}(x_{0}), \pi_{n}(y_{0}))
    \bigr|
    \;\leq\;
    C_{5} \frac{(2 C_{\mathrm{d}} \eta)^{\varrho}}{\varepsilon^{d}} n^{2-d}.
  \end{align}
  By the symmetry of the $g_{n D}^{\omega}$,
  \begin{align*}
    &\bigl|
      n^{d-2} g_{n D}^{\omega}(\pi_{n}(x), \pi_{n}(y)) - g_{D}^{\Sigma}(x, y)
    \bigr|
    \\[.5ex]
    &\mspace{36mu}\leq\;
    \bigl|
      n^{d-2}
      \bigl(
        g_{n D}^{\omega}(\pi_{n}(y), \pi_{n}(x))
        -
        g_{n D}^{\omega}(\pi_{n}(y_{0}), \pi_{n}(x))
      \bigr)
    \bigr|
    \\[.5ex]
    &\mspace{72mu}+
    \bigl|
      n^{d-2}\
      \bigl(
        g_{n D}^{\omega}(\pi_{n}(x), \pi_{n}(y_{0}))
        -
        g_{n D}^{\omega}(\pi_{n}(x_{0}), \pi_{n}(y_{0}))
      \bigr)
    \bigr|
    \\[.5ex]
    &\mspace{72mu}+ 
    \bigl|
      n^{d-2} g_{n D}^{\omega}(\pi_{n}(x_{0}), \pi_{n}(y_{0})) - g_{D}^{\Sigma}(x_{0}, y_{0})
    \bigr| 
    +
    \bigl|
      g_{D}^{\Sigma}(x_{0}, y_{0}) - g_{D}^{\Sigma}(x, y)
    \bigr|.
  \end{align*}
  By \eqref{eq:conv_X} the third term on the right hand side is smaller than $\epsilon_{1}$. Furthermore, by \eqref{eq:uniform_hoelder:1} and \eqref{eq:uniform_hoelder:2} for $\eta$ small any enough the first and the second term are both smaller than $\epsilon_{1}$. The last term also is smaller than $\epsilon_{1}$ for $\eta$ sufficiently small by the uniform continuity of $g_{D}^{\Sigma}$. Thus, for $n$ large enough,
  \begin{align*}
    \sup_{(x,y) \in K_{\varepsilon, \delta}}
    \bigl|
      n^{d-2} g_{n D}^{\omega}(\pi_{n}(x), \pi_{n}(y)) - g_{D}^{\Sigma}(x, y)
    \bigr|
    \;\leq\;
    4 \epsilon_{1},
  \end{align*}
  which is the desired conclusion.
\end{proof}

\subsection*{Acknowledgement.} We are grateful to Lisa Hartung for valuable discussions that initiated this work. We also thank the Weierstrass Institute for Applied Analysis and Stochastics in Berlin for their hospitality on numerous occasions.
%
%
%
%

% \bibliographystyle{abbrv}
% \bibliography{literature}

\begin{thebibliography}{10}

\bibitem{ABDH13}
S.~Andres, M.~T. Barlow, J.-D. Deuschel, and B.~M. Hambly.
\newblock {Invariance principle for the random conductance model}.
\newblock {\em Probab. Theory Related Fields}, 156(3-4):535--580, 2013.

\bibitem{ACS21}
S.~Andres, A.~Chiarini, and M.~Slowik.
\newblock Quenched local limit theorem for random walks among time-dependent
  ergodic degenerate weights.
\newblock {\em Probab. Theory Related Fields}, 179(3-4):1145--1181, 2021.

\bibitem{ACK24}
S.~Andres, D.~A. Croydon, and T.~Kumagai.
\newblock Heat kernel fluctuations and quantitative homogenization for the
  one-dimensional {B}ouchaud trap model.
\newblock {\em Stochastic Process. Appl.}, 172:Paper No. 104336, 20, 2024.

\bibitem{ADS15}
S.~Andres, J.-D. Deuschel, and M.~Slowik.
\newblock {Invariance principle for the random conductance model in a
  degenerate ergodic environment}.
\newblock {\em Ann. Probab.}, 43(4):1866--1891, 2015.

\bibitem{ADS16}
S.~Andres, J.-D. Deuschel, and M.~Slowik.
\newblock {Harnack inequalities on weighted graphs and some applications to the
  random conductance model}.
\newblock {\em Probab. Theory Related Fields}, 164(3-4):931--977, 2016.

\bibitem{ADS20}
S.~Andres, J.-D. Deuschel, and M.~Slowik.
\newblock Green kernel asymptotics for two-dimensional random walks under
  random conductances.
\newblock {\em Electron. Commun. Probab.}, 25:Paper No. 58, 14, 2020.

\bibitem{AT21}
S.~Andres and P.~A. Taylor.
\newblock Local limit theorems for the random conductance model and
  applications to the {G}inzburg-{L}andau {$\nabla\phi$} interface model.
\newblock {\em J. Stat. Phys.}, 182(2):Paper No. 35, 35, 2021.

\bibitem{AD18}
S.~Armstrong and P.~Dario.
\newblock Elliptic regularity and quantitative homogenization on percolation
  clusters.
\newblock {\em Comm. Pure Appl. Math.}, 71(9):1717--1849, 2018.

\bibitem{Ba04}
M.~T. Barlow.
\newblock {Random walks on supercritical percolation clusters}.
\newblock {\em Ann. Probab.}, 32(4):3024--3084, 2004.

\bibitem{BH09}
M.~T. Barlow and B.~M. Hambly.
\newblock {Parabolic {H}arnack inequality and local limit theorem for
  percolation clusters}.
\newblock {\em Electron. J. Probab.}, 14:no. 1, 1--27, 2009.

\bibitem{BS20}
P.~Bella and M.~Sch\"{a}ffner.
\newblock Quenched invariance principle for random walks among random
  degenerate conductances.
\newblock {\em Ann. Probab.}, 48(1):296--316, 2020.

\bibitem{BS22}
P.~Bella and M.~Sch\"{a}ffner.
\newblock Non-uniformly parabolic equations and applications to the random
  conductance model.
\newblock {\em Probab. Theory Related Fields}, 182(1-2):353--397, 2022.

\bibitem{Bi99}
P.~Billingsley.
\newblock {\em Convergence of probability measures}.
\newblock Wiley Series in Probability and Statistics: Probability and
  Statistics. John Wiley \& Sons, Inc., New York, second edition, 1999.
\newblock A Wiley-Interscience Publication.

\bibitem{Bi11}
M.~Biskup.
\newblock {Recent progress on the random conductance model}.
\newblock {\em Probab. Surv.}, 8:294--373, 2011.

\bibitem{Bi20}
M.~Biskup.
\newblock Extrema of the two-dimensional discrete {G}aussian free field.
\newblock In {\em Random graphs, phase transitions, and the {G}aussian free
  field}, volume 304 of {\em Springer Proc. Math. Stat.}, pages 163--407.
  Springer, Cham, [2020] \copyright 2020.

\bibitem{BS11}
M.~Biskup and H.~Spohn.
\newblock Scaling limit for a class of gradient fields with nonconvex
  potentials.
\newblock {\em Ann. Probab.}, 39(1):224--251, 2011.

\bibitem{BD03}
D.~Boivin and J.~Depauw.
\newblock Spectral homogenization of reversible random walks on
  {$\mathbb{Z}^d$} in a random environment.
\newblock {\em Stochastic Process. Appl.}, 104(1):29--56, 2003.

\bibitem{BKM15}
O.~Boukhadra, T.~Kumagai, and P.~Mathieu.
\newblock Harnack inequalities and local central limit theorem for the
  polynomial lower tail random conductance model.
\newblock {\em J. Math. Soc. Japan}, 67(4):1413--1448, 2015.

\bibitem{CI03}
P.~Caputo and D.~Ioffe.
\newblock Finite volume approximation of the effective diffusion matrix: the
  case of independent bond disorder.
\newblock {\em Ann. Inst. H. Poincar\'{e} Probab. Statist.}, 39(3):505--525,
  2003.

\bibitem{CCK15}
Z.-Q. Chen, D.~A. Croydon, and T.~Kumagai.
\newblock Quenched invariance principles for random walks and elliptic
  diffusions in random media with boundary.
\newblock {\em Ann. Probab.}, 43(4):1594--1642, 2015.

\bibitem{CN21}
A.~Chiarini and M.~Nitzschner.
\newblock Disconnection and entropic repulsion for the harmonic crystal with
  random conductances.
\newblock {\em Comm. Math. Phys.}, 386(3):1685--1745, 2021.

\bibitem{CR24}
L.~Chiarini and W.~M. Ruszel.
\newblock Stochastic homogenization of {G}aussian fields on random media.
\newblock {\em Ann. Henri Poincar\'{e}}, 25(3):1869--1895, 2024.

\bibitem{CDH19}
A.~Cipriani, B.~Dan, and R.~S. Hazra.
\newblock The scaling limit of the membrane model.
\newblock {\em Ann. Probab.}, 47(6):3963--4001, 2019.

\bibitem{CHR18}
A.~Cipriani, R.~S. Hazra, and W.~M. Ruszel.
\newblock Scaling limit of the odometer in divisible sandpiles.
\newblock {\em Probab. Theory Related Fields}, 172(3-4):829--868, 2018.

\bibitem{CvG20}
A.~Cipriani and B.~van Ginkel.
\newblock The discrete {G}aussian free field on a compact manifold.
\newblock {\em Stochastic Process. Appl.}, 130(7):3943--3966, 2020.

\bibitem{CK12}
C.~Cotar and C.~K\"ulske.
\newblock Existence of random gradient states.
\newblock {\em Ann. Appl. Probab.}, 22(4):1650--1692, 2012.

\bibitem{CK15}
C.~Cotar and C.~K\"ulske.
\newblock Uniqueness of gradient {G}ibbs measures with disorder.
\newblock {\em Probab. Theory Related Fields}, 162(3-4):587--635, 2015.

\bibitem{CH08}
D.~A. Croydon and B.~M. Hambly.
\newblock Local limit theorems for sequences of simple random walks on graphs.
\newblock {\em Potential Anal.}, 29(4):351--389, 2008.

\bibitem{DG21}
P.~Dario and C.~Gu.
\newblock Quantitative homogenization of the parabolic and elliptic {G}reen's
  functions on percolation clusters.
\newblock {\em Ann. Probab.}, 49(2):556--636, 2021.

\bibitem{Da95}
E.~B. Davies.
\newblock {\em Spectral theory and differential operators}, volume~42 of {\em
  Cambridge Studies in Advanced Mathematics}.
\newblock Cambridge University Press, Cambridge, 1995.

\bibitem{DNS18}
J.-D. Deuschel, T.~A. Nguyen, and M.~Slowik.
\newblock {Quenched invariance principles for the random conductance model on a
  random graph with degenerate ergodic weights}.
\newblock {\em Probab. Theory Related Fields}, 170(1-2):363--386, 2018.

\bibitem{DRZ17}
J.~Ding, R.~Roy, and O.~Zeitouni.
\newblock Convergence of the centered maximum of log-correlated {G}aussian
  fields.
\newblock {\em Ann. Probab.}, 45(6A):3886--3928, 2017.

\bibitem{Do94}
J.~L. Doob.
\newblock {\em Measure theory}, volume 143 of {\em Graduate Texts in
  Mathematics}.
\newblock Springer-Verlag, New York, 1994.

\bibitem{DRS14}
A.~Drewitz, B.~R{\'a}th, and A.~Sapozhnikov.
\newblock {On chemical distances and shape theorems in percolation models with
  long-range correlations}.
\newblock {\em J. Math. Phys.}, 55(8):083307, 30, 2014.

\bibitem{El18}
J.~Elstrodt.
\newblock {\em Ma\ss- und {I}ntegrationstheorie}.
\newblock Springer-Lehrbuch. [Springer Textbook]. Springer Spektrum Berlin,
  eighth edition, 2018.
\newblock Grundwissen Mathematik. [Basic Knowledge in Mathematics].

\bibitem{Ev10}
L.~C. Evans.
\newblock {\em Partial differential equations}, volume~19 of {\em Graduate
  Studies in Mathematics}.
\newblock American Mathematical Society, Providence, RI, second edition, 2010.

\bibitem{EG15}
L.~C. Evans and R.~F. Gariepy.
\newblock {\em Measure theory and fine properties of functions}.
\newblock Textbooks in Mathematics. CRC Press, Boca Raton, FL, revised edition,
  2015.

\bibitem{FHS19}
F.~Flegel, M.~Heida, and M.~Slowik.
\newblock Homogenization theory for the random conductance model with
  degenerate ergodic weights and unbounded-range jumps.
\newblock {\em Ann. Inst. Henri Poincar\'e{} Probab. Stat.}, 55(3):1226--1257,
  2019.

\bibitem{FRS21}
S.~Floreani, F.~Redig, and F.~Sau.
\newblock Hydrodynamics for the partial exclusion process in random
  environment.
\newblock {\em Stochastic Process. Appl.}, 142:124--158, 2021.

\bibitem{Ge20}
T.~Gerard.
\newblock Representations of the vertex reinforced jump process as a mixture of
  {M}arkov processes on {$\mathbb{Z}^d$} and infinite trees.
\newblock {\em Electron. J. Probab.}, 25:Paper No. 108, 45, 2020.

\bibitem{GOS01}
G.~Giacomin, S.~Olla, and H.~Spohn.
\newblock {Equilibrium fluctuations for {$\nabla\phi$} interface model}.
\newblock {\em Ann. Probab.}, 29(3):1138--1172, 2001.

\bibitem{KS91}
I.~Karatzas and S.~E. Shreve.
\newblock {\em Brownian motion and stochastic calculus}, volume 113 of {\em
  Graduate Texts in Mathematics}.
\newblock Springer-Verlag, New York, second edition, 1991.

\bibitem{KP87}
U.~Krengel and R.~Pyke.
\newblock {Uniform pointwise ergodic theorems for classes of averaging sets and
  multiparameter subadditive processes}.
\newblock {\em Stochastic Process. Appl.}, 26(2):289--296, 1987.

\bibitem{Mi11}
J.~Miller.
\newblock Fluctuations for the {G}inzburg-{L}andau {$\nabla\phi$} interface
  model on a bounded domain.
\newblock {\em Comm. Math. Phys.}, 308(3):591--639, 2011.

\bibitem{NS97}
A.~Naddaf and T.~Spencer.
\newblock On homogenization and scaling limit of some gradient perturbations of
  a massless free field.
\newblock {\em Comm. Math. Phys.}, 183(1):55--84, 1997.

\bibitem{NW18}
C.~M. Newman and W.~Wu.
\newblock Gaussian fluctuations for the classical {XY} model.
\newblock {\em Ann. Inst. Henri Poincar\'e{} Probab. Stat.}, 54(4):1759--1777,
  2018.

\bibitem{Sa96}
L.~Saloff-Coste.
\newblock {Lectures on finite {M}arkov chains}.
\newblock In {\em {Lectures on probability theory and statistics
  ({S}aint-{F}lour, 1996)}}, volume 1665 of {\em {Lecture Notes in Math.}},
  pages 301--413. Springer, Berlin, 1997.

\bibitem{Sa17}
A.~Sapozhnikov.
\newblock {Random walks on infinite percolation clusters in models with
  long-range correlations}.
\newblock {\em Ann. Probab.}, 45(3):1842--1898, 2017.

\bibitem{SZ24}
F.~Schweiger and O.~Zeitouni.
\newblock The maximum of log-correlated {G}aussian fields in random
  environment.
\newblock {\em Comm. Pure Appl. Math.}, 77(5):2778--2859, 2024.

\bibitem{St11}
D.~W. Stroock.
\newblock {\em Probability theory}.
\newblock Cambridge University Press, Cambridge, second edition, 2011.
\newblock An analytic view.

\bibitem{Ta23}
M.~E. Taylor.
\newblock {\em Partial differential equations {I}. {B}asic theory}, volume 115
  of {\em Applied Mathematical Sciences}.
\newblock Springer, Cham, third edition, [2023] \copyright 2023.

\end{thebibliography}

%
%
%
%

\end{document}